\documentclass[10pt,oneside,a4paper]{article}
\usepackage{titlesec}
\titleformat{\subsection}[runin]{\normalfont\bfseries}{\thesubsection.}{.5em}{}[.~ ]
\titlespacing{\subsection}{0pt}{1.5ex plus .1ex minus .2ex}{0pt}
\titlespacing*{\section}{0pt}{8pt}{5pt}
\usepackage{graphicx}
\usepackage{caption}
\usepackage{subcaption}
\usepackage{wrapfig}
\usepackage{amsfonts,amssymb,amsmath,amsthm}
\usepackage{enumerate}
\usepackage{url}
\usepackage[mathcal,mathscr]{euscript}
\usepackage[dvipsnames]{xcolor}
\usepackage{hyperref}
\hypersetup{colorlinks=true,linkcolor=Magenta,citecolor=PineGreen}

\let\OLDthebibliography\thebibliography \renewcommand\thebibliography[1]{\OLDthebibliography{#1}   \setlength{\parskip}{0pt}\setlength{\itemsep}{0pt plus 0.3ex}}

\newtheoremstyle{thmstyle}
{\topsep}
{\topsep}
{\itshape}
{0pt}
{\bfseries}
{.}
{5pt plus 1pt minus 1pt}
{#2.\hspace{3pt}#1#3}
\newtheoremstyle{defistyle}
{\topsep}
{\topsep}
{}
{0pt}
{\bfseries}
{.}
{5pt plus 1pt minus 1pt}
{#2.\hspace{3pt}#1#3}

\theoremstyle{thmstyle}
\newtheorem{thm}[subsection]{Theorem}
\newtheorem{lemma}[subsection]{Lemma}
\newtheorem{prop}[subsection]{Proposition}
\newtheorem{cor}[subsection]{Corollary}
\theoremstyle{defistyle}
\newtheorem{defi}[subsection]{Definition}
\newtheorem{rmk}[subsection]{Remark}
\newcommand{\trace}{\mathop{\mathrm{tr}}}
\newcommand{\tance}{\mathop{\mathrm{ta}}}
\newcommand{\arc}{\mathop{\mathrm{arc}}}
\newcommand{\Arg}{\mathop{\mathrm{Arg}}}
\newcommand{\area}{\mathop{\mathrm{area}}}
\newcommand{\Area}{\mathop{\mathrm{Area}}}
\newcommand{\SU}{\mathop{\mathrm{SU}}}
\newcommand{\PU}{\mathop{\mathrm{PU}}}
\newcommand{\BV}{{\mathrm B}\,V}
\newcommand{\SV}{{\mathrm S}\,V}
\newcommand{\EV}{{\mathrm E}\,V}
\newcommand{\real}{\mathop{\mathrm{Re}}}
\newcommand{\imag}{\mathop{\mathrm{Im}}}
\newcommand{\PCV}{\mathbb{P}V}
\newcommand{\dist}{\mathop{\mathrm{dist}}}

\title{Special elliptic isometries, relative $\SU(2,1)$-character varieties, and bendings}
\author{Felipe A.~Franco\footnote{Supported by grant 2014/00582-2, S\~ao Paulo Research Foundation (FAPESP), and by CNPq.} \and Carlos H.~Grossi}
\date{}

\begin{document}

\maketitle

\begin{abstract}
We study relations between special elliptic isometries in the complex hyperbolic plane. Relations of lengths $2$, $3$, and $4$ are fully classified. Some relative $\SU(2,1)$-character varieties of the quadruply punctured sphere are described and applied to the study of length $5$ relations.
\end{abstract}

\section{Introduction}

Relations between automorphisms of a given geometric structure play an important role in the construction of manifolds/orbifolds endowed with that geometric structure. Consider, for instance, Poincar\'e's Polyhedron Theorem, which is one of the few known tools for the construction of manifolds/orbifolds equipped with some model geometry (typically, a simply-connected Riemannian manifold). Roughly speaking, the theorem specifies conditions on a polyhedron with side-pairing isometries in the model space $X$ such that the group $H$ generated by these isometries is discrete and $X/H$ is a manifold/orbifold $M$ modelled on $X$. The group $H$ is isomorphic to the fundamental group $\pi_1(M)$ and the theorem provides an explicit presentation of $H$ that comes from the combinatorial structure of the polyhedron with face-pairing isometries. This means that, in a certain sense, in order to construct a polyhedron with side-pairing isometries that have a chance of succeeding as a fundamental polyhedron, some relations between those isometries of $X$ that will play the role of side-pairing isometries must be known {\it a priori.}\footnote{For example, the study of short relations between isometries in the complex hyperbolic plane plays an important role in the construction of complex hyperbolic disc bundles in \cite{SGG2011} and in \cite{SashaGusevskii2005}.}

More generally, the space of representations of the fundamental group $\pi_1(M)$ in some group $G$ of automorphisms of the model space modulo conjugation, i.e., the $G$-character variety of $M$, is closely related to the geometric structures on $M$ inherited from the model space. Hence, it is natural to expect that (relative) character varieties are ubiquitous objects in geometry and that the many questions related to its structure (topology, Hitchin components, nature of the action of the mapping class group, etc.) are sources of great interest. They have been investigated by several authors, and an exhaustive list of references would be too long to compile; so, we only cite a few ones \cite{ALS2018}, \cite{ABL2017},  \cite{GB2012}, \cite{FL2008}, \cite{La2008}, \cite{MPP-T2013}, \cite{Will2017} which are closer to this paper.

Here, our model space is the complex hyperbolic plane $\mathbb H^2_\mathbb C$ with orientation-preserving isometries or, equivalently, the holomorphic $2$-ball with its complex automorphisms; the corresponding group is the projective unitary group $\PU(2,1)$. A rough classification of nontrivial orientation-preserving isometries in the complex hyperbolic plane resembles that of constant curvature hyperbolic geometry: they either have a fixed point in $\mathbb H^2_\mathbb C$ (elliptic isometries), exactly one fixed point in the ideal boundary of $\mathbb H^2_\mathbb C$ (parabolic isometries), or exactly two fixed points in this ideal boundary (loxodromic isometries). Each of these isometry types are divided into several subtypes whose geometric behaviour can be quite different from each other (see Subsection \ref{subsec:isometries}). Of central interest in this paper is the subtype of elliptic isometries known as the {\it special\/} ones. This subtype includes the holomorphic involutions.

Holomorphic involutions generate the group of orientation-preserving isometries of $\mathbb H^2_\mathbb C$. They come in two conjugacy classes: reflections in (negative) points and reflections in complex geodesics (or in positive points). The decomposition of orientation-preserving isometries into the product of involutions is considered in \cite{Sasha2012} and in \cite{Will2017}. An interesting question is to understand to what extent such a decomposition is unique. This naturally leads to the study of relative character varieties that encode all the possible decompositions, modulo conjugation, of a given isometry into the product of involutions \cite[Section 4]{Sasha2012} and to the concept of {\it bendings.} In a nutshell, bendings provide natural coordinates in the mentioned relative character varieties. More precisely, let $R^p$ stand for the reflection in a negative or positive point $p$ and consider a relation $R^{p_n}\ldots R^{p_2}R^{p_1}=1$ between holomorphic involutions in $\PU(2,1)$. If we move the points $p_{i-1},p_i$ along a geodesic that joins them without altering their distance, we obtain new points $q_{i-1},q_i$ satisfying $R^{q_i}R^{q_{i-1}}=R^{p_i}R^{p_{i-1}}$. This alters the original relation $R^{p_n}\ldots R^{p_i}R^{p_{i-1}}\ldots R^{p_2}R^{p_1}=1$ into the new one $R^{p_n}\ldots R^{q_i}R^{q_{i-1}}\ldots R^{p_1}=1$ and is the same as taking an element $C$ in the centralizer of $R^{p_i}R^{p_{i-1}}$ and writing $R^{p_i}R^{p_{i-1}}=(CR^{p_i}C^{-1})(CR^{p_{i-1}}C^{-1})
=R^{Cp_i}R^{Cp_{i-1}}=R^{q_i}R^{q_{i-1}}$.

Sometimes, a relation between holomorphic involutions of the above form can be simplified by bending it and applying afterwards the length $2$ relation $R^pR^p=1$ (a {\it cancellation\/} which can appear when neighbouring involutions in the relation become equal after a bending) or a length $3$ relation known as an {\it orthogonal relation\/} \cite{Sasha2012}. It is worthwhile mentioning that length $5$ relations between holomorphic involutions that cannot be simplified in such a way, that is, {\it basic\/} length $5$ relations, have been linked to discreteness \cite{Sasha2012}, \cite{ABG2007}.

In this paper, we consider relations between special elliptic isometries. Special elliptic isometries can be seen as rotations around (negative) points or rotations around complex geodesics (equivalently, around positive points). Since every orientation-preserving isometry has three lifts to $\SU(2,1)$ that differ by a cube root of unity, a nontrivial special elliptic isometry is determined, at the level of $\SU(2,1)$, by a (negative or positive) point $p$, its {\it centre,} and by a unit complex number $\alpha$ distinct from a cube root of unity, its {\it parameter.} Throughout the paper, we deal with elements in $\SU(2,1)$; so, we write a relation between special elliptic isometries in the form $R_{\alpha_n}^{p_n}\ldots R_{\alpha_2}^{p_2}R_{\alpha_1}^{p_1}=\delta$, where $\delta$ is a cube root of unity, and refer to $R_{\alpha_n}^{p_n}\ldots R_{\alpha_2}^{p_2}R_{\alpha_1}^{p_1}=\delta$ as a length $n$ relation.

Relations between special elliptic isometries of lengths $2$ and $3$, as well as the length $4$ ones obtained through bendings, are quite similar to those between holomorphic involutions and are described in Sections  \ref{sec:basic} and \ref{sec:bendings}. On the other hand, the full description of length $4$ relations in Theorem \ref{thm:alll4}, one of the main results of the paper, is much more involved and requires a new set of tools which are particularly technical. For that reason, it is postponed until Section \ref{sec:n-o-solutions}.

Equipped with bendings, we are able to consider the decompositions of regular orientation-preserving isometries (see Definition \ref{defi:regular}) in the product of three special elliptic ones. (This is in fact the main part of the study of length $5$ relations between special elliptic isometries since such relations can be written in the form $R_{\alpha_3}^{p_3}R_{\alpha_2}^{p_2}R_{\alpha_1}^{p_1}=R_{\beta_2}^{q_2}R_{\beta_1}^{q_1}$ and $R_{\beta_2}^{q_2}R_{\beta_1}^{q_1}$ is a regular orientation-preserving isometry.) These decompositions naturally lead to the description, given in Theorem \ref{thm:charactervariety} (see also Theorem \ref{thm:surfaceS}) of some relative $\SU(2,1)$-character varieties consisting of representations $\rho:\pi_1(\Sigma)\to\SU(2,1)$, modulo conjugation, of the rank $3$ free group $\pi_1(\Sigma):=\langle\iota_1,\iota_2,\iota_3,\iota_4\mid\iota_4\iota_3\iota_2\iota_1=
1\rangle$ (the fundamental group of the quadruply punctured sphere~$\Sigma$), where the conjugacy classes of $\rho(\iota_i)$ are those of special elliptic isometries, $i=1,2,3$, and the conjugacy class of $\rho(\iota_4)$ is that of a regular isometry in $\SU(2,1)$. These relative character varieties are (as is typical) semialgebraic surfaces $S$ whose nature, studied in Theorem \ref{thm:sbendingconnect}, allows us to obtain a simple condition guaranteeing that a couple of given points in $S$ lie in a same connected component and, in particular, can be connected, modulo conjugation, by finitely many bendings (Corollary \ref{cor:connectSlox}). Incidently, an unexpected consequence of Theorem \ref{thm:sbendingconnect} is a criterion determining the type of an isometry of the form $R_{\alpha_2}^{p_2}R_{\alpha_1}^{p_1}$ in terms of centres and parameters.

Some experimental observations regarding the semialgebraic surface $S$, as well as many pictures illustrating its behaviour, can be seen on Subsection \ref{subsec:experiments}.

It is worthwhile mentioning that the study of the possible conjugacy classes of the product of a pair of isometries have been considered in \cite{FW2009}, \cite{Paupert2007}, and \cite{Will2017}. It is formulated in terms of the product map $\mathcal C_1\times\mathcal C_2\to\mathcal G$, $(A,B)\mapsto[AB]$, where $\mathcal G$ denotes the space of all $\PU(2,1)$-conjugacy classes, $\mathcal C_1,\mathcal C_2\in\mathcal G$, and $[AB]$ stands for the $\PU(2,1)$-conjugacy class of the product $AB$. In our context, the vertical and horizontal slices in Theorem \ref{thm:sbendingconnect} can be seen as fibres of the product map. Moreover, the image by the product map of a {\it degenerate\/} slice (see Definition \ref{defi:degenerate}) is a point in a reducible wall (see \cite{Paupert2007} for the definition of reducible wall).

Finally, in Section \ref{sec:n-o-solutions}, we complete the classification of length $4$ relations between special elliptic isometries. The \hbox{$f$-{\it bendings\/}} play a major role in this classification. Similarly to bendings, they can also be seen as one-parameter deformations of a given product $R_{\alpha_2}^{p_2}R_{\alpha_1}^{p_1}$ of special elliptic isometries; such deformation is geometrically described in Theorem \ref{thm:geometricinter}. However, during an \hbox{$f$-bending,} both the centres and the parameters change. An $f$-bending preserves the signs of the centres as well as the {\it components\/} (see Definition \ref{defi:samecomponent}) and the product of the parameters. Every (generic) length $4$ relation of the form $R_{\alpha_2}^{p_2}R_{\alpha_1}^{p_1}=R_{\beta_2}^{q_2}R_{\beta_1}^{q_1}$, where $p_i,q_i$ are points of a same sign and $\alpha_i,\beta_i$ are parameters in a same component, is a consequence of bendings and $f$-bendings (Theorem~\ref{thm:all4relations}). Dropping these restrictions on signs and components, we found it necessary to develop some new tools in order to complete the classification of (generic) length $4$ isometries (in total, there are $7$ ``basic'' relations of length $4$). These tools regard the behaviour of some naturally parameterized lines tangent to Goldman's deltoid (Proposition \ref{prop:gdd}) as well as a characterization allowing to determine when regular elliptic isometries of the same trace written as products of two special elliptic isometries belong to the same $\SU(2,1)$-conjugacy class (Proposition \ref{prop:nochangesofsign}).

As applications of the techniques developed in Sections \ref{sec:relative} and \ref{sec:n-o-solutions}, we show that {\it special elliptic pentagons,} i.e., some length $5$ relations between special elliptic isometries, can be connected modulo conjugation by finitely many bendings (see Theorem~\ref{thm:connectpent}) as well as by finitely many bendings and $f$-bendings (see Theorem \ref{thm:connectspecpent}), as long as the appropriate natural conditions are required in each case.

\smallskip

\noindent
{\bf Acknowledgments.}~We are very grateful to the anonymous referee whose careful suggestions have greatly improved the paper.

\section{Complex hyperbolic geometry}
\label{sec:hyperbolicgeometry}

In this section we briefly discuss some basic aspects of plane complex hyperbolic geometry. Our approach essentially follows \cite{SashaGrossi2011}, \cite{SGG2011}, and \cite{Goldman1999}.

Let $V$ be a 3-dimensional $\mathbb{C}$-linear space equipped with a Hermitian form $\left<-,-\right>$ of signature $++-$. We frequently use the same letter to denote both a point in the complex projective plane $\mathbb PV$ and a representative in $V$.

The complex projective plane $\PCV$ is divided into \textit{negative,} \textit{isotropic,} and \textit{positive} points:
$$\BV := \left\{p\in{\mathbb P}V \,|\, \left<p,p\right> <0\right\},\ \ \SV := \left\{p\in{\mathbb P}V \,|\, \left<p,p\right> =0\right\},\ \EV:=\{p\in\mathbb PV\,|\, \left<p,p\right>>0\}.$$
The {\it signature\/} $\sigma p$ of a point $p\in\mathbb PV$ is respectively $-1,0,1$ when $p$ is negative, isotropic, positive. It is easy to see that $\BV$ is a (real) $4$-dimensional open ball whose boundary $\SV$ is a $3$-sphere.

Let $p\in\PCV\setminus\SV$ be a nonisotropic point. There is a well-known natural identification
\begin{equation}\label{eq:tangentspace}
\mathrm T_p\mathbb PV\simeq{\mathrm{Lin}}_{\mathbb C}(\mathbb Cp,p^\bot)=\left<-,p\right>p^\bot
\end{equation}
where $p^\perp$ stands for the linear subspace orthogonal to $p$ and $\langle-,p\rangle$ denotes the linear functional $x\mapsto\langle x,p\rangle$.

Both $\BV$ and $\EV$ are endowed with the Hermitian metric defined by
\begin{equation*}\label{eq:metric}
\langle t_1,t_2\rangle:=-\frac{\big\langle t_1(p),t_2(p)\big\rangle}{\langle p,p\rangle}
\end{equation*}
where $t_1,t_2\in\mathrm{Lin}(\mathbb Cp,p^\perp)$ are tangent vectors at the nonisotropic point $p$. This Hermitian metric is positive-definite on $\BV$ and of signature $+-$ on $\EV$. In particular, we obtain a Riemannian metric on $\BV$. Equipped with such metric, $\BV$ is called the {\it complex hyperbolic plane\/} and is denoted by $\mathbb H_\mathbb C^2$. Its ideal boundary, also known as the {\it absolute,} is the $3$-sphere $\SV$ of isotropic points. Note that $\EV$ is a pseudo-Riemannian manifold; a simple duality, discussed below, shows that it is the space of complex lines intersecting $\BV$.

Let $L\subset\mathbb PV$ be a complex line, i.e., the projectivization $\mathbb PW$ of a complex $2$-dimensional subspace $W\leqslant V$. The point $c:=\mathbb PW^\perp$ is called the {\it polar\/} of $L$. By Sylvester's criterion, the signature of the Hermitian form restricted to $W$ can be $+-$, $++$, or $+0$. The corresponding complex line is respectively called {\it hyperbolic,} \textit{spherical,} or \textit{Euclidean.}
Clearly, a projective line is hyperbolic, spherical, or Euclidean exactly when its polar point is positive, negative, or isotropic. The negative part $L\cap\BV$ of a hyperbolic complex line $L$ is often called a \textit{complex geodesic.} Given two distinct points $p_1,p_2\in\mathbb PV$, the (unique) complex line $\mathbb P(\mathbb Cp_1+\mathbb Cp_2)$ containing $p_1,p_2$ is denoted $\mathrm L(p_1,p_2)$. The following simple facts concerning complex lines will be regularly used throughout the paper:

The restriction $L\cap\BV$ of a hyperbolic line to $\BV$ is a totally geodesic subspace of constant curvature (a Poincar\'e disc). The same holds for $L\cap\EV$. The geometry of a spherical complex line is that of a round sphere.

Arbitrary complex lines $L_1,L_2$ are either equal or have a single common point in $\mathbb PV$. A pair of complex lines $L_1,L_2$ is said to be {\it orthogonal\/} iff the polar point of $L_1$ belongs to $L_2$ (or, equivalently, the polar point of $L_2$ belongs to $L_1$). When the lines are noneuclidean, this means that they are orthogonal in the sense of the Hermitian metric. Given a point $p$ in a complex line $L$, there exists a unique point $\tilde p\in L$ such that $\langle p,\tilde p\rangle=0$ (in the Euclidean case, $p=\tilde p$).

A useful criterion to decide the type of the complex line $L:=\mathrm{L}(p_1,p_2)$ in terms of nonisotropic, nonorthogonal, distinct spanning points $p_1,p_2$ involves the {\it tance}
$$\tance(p_1,p_2):=\frac{\left<p_1,p_2\right>\left<p_2,p_1\right>}
{\left<p_1,p_1\right>\left<p_2,p_2\right>}.$$
By Sylvester's criterion, the line $L$ is hyperbolic when $\tance(p_1,p_2)>1$ or $\tance(p_1,p_2)<0$; Euclidean when $\tance(p_1,p_2)=1$; spherical when $0<\tance(p_1,p_2)<1$. Note that, for $p_1,p_2\in\BV$, we have $\tance(p_1,p_2)\geqslant1$ and $\tance(p_1,p_2)=1$ iff $p_1=p_2$.

An {\it extended geodesic\/} in $\mathbb PV$ is, by definition, the (complex) projectivization $\mathbb PW$ of an $\mathbb R$-linear subspace $W$ of $V$, $\dim_\mathbb RW=2$, such that the Hermitian form, being restricted to $W$, is real and does not vanish ($\mathbb PW$ stands for $\pi(W\setminus\{0\})$, where $\pi:V\setminus\{0\}\to\mathbb PV$ is the canonical projection). Every extended geodesic is a topological circle contained in a unique complex line. The usual Riemannian geodesics in~$\mathbb H_\mathbb C^2$ are the restrictions $\mathbb PW\cap\BV$ \cite[Corollary~5.5]{SashaGrossi2011} (the same holds for the usual pseudo-Riemannian geodesics in $\EV$).

The following simple facts concerning extended geodesics will be used later:

Let $p_1,p_2$ be distinct nonorthogonal points. There exists a unique extended geodesic containing $p_1,p_2$ (it is given by $\mathbb PW$ with $W:=\mathbb Rp_1+\mathbb R\langle p_1,p_2\rangle p_2$). This extended geodesic is denoted by $\mathrm{G}{\wr}p_1,p_2{\wr}$. In what follows, we will refer to an extended geodesic simply as a geodesic.

Let $G_1,G_2$ be geodesics in a same noneuclidean complex line $L$.
Assume that $G_1,G_2$ intersect at a nonisotropic point $p$. The counterclockwise oriented angle from $G_1$ to $G_2$ at $p$ is denoted $\angle_pG_1G_2$. Note that $\angle_{\tilde p}G_1G_2=\angle_pG_2G_1$ where $\tilde p$ stands for the point in $L$ orthogonal to $p$.
We~will only measure oriented angles at intersecting points of geodesics that lie in a same noneuclidean complex line.

Let $p,\tilde p\in L$ be orthogonal points in a projective line $L$. Every geodesic in $L$ containing $p$ also contains $\tilde p$. In particular, every geodesic in a Euclidean line contains the isotropic point which is the polar point of the line. Moreover, if geodesics $G_1,G_2$ in $L$ intersect at a nonisotropic point $q$, then they also intersect at the point $\tilde q$ in $L$ orthogonal to $q$.

Let $p_1,p_2$ be distinct nonisotropic nonorthogonal points of the same signature in a noneuclidean projective line. The {\it geodesic segment\/} from $p_1$ to $p_2$ is the arc in $\mathrm{G}{\wr}p_1,p_2{\wr}$ that joins $p_1$ and $p_2$ and does not contain the point $\tilde p_1\in L:=\mathrm{L}(p_1,p_2)$ orthogonal to $p_1$. We denote the geodesic segment from $p_1$ to $p_2$ by $\mathrm G[p_1,p_2]$. Note that $\mathrm G[p_1,p_2]$ consists of a usual geodesic segment in a Poincar\'e disc $L\cap\BV$ (or $L\cap\EV$) when $L$ is hyperbolic or of a usual minimal geodesic segment in the round sphere when $L$ is spherical. (The definition also works when one allows $p_1,p_2$ to have opposite signatures but we do not need this case.)

Given pairwise distinct pairwise nonorthogonal nonisotropic points $p_1,p_2,p_3$ of the same signature in a noneuclidean projective line $L$, let $\Delta(p_1,p_2,p_3)\subset L$ stand for the oriented geodesic triangle whose vertices are $p_1,p_2,p_3$ and whose sides are the geodesic segments $\mathrm G[p_1,p_2]$, $\mathrm G[p_2,p_3]$, and $\mathrm G[p_3,p_1]$. These are usual oriented geodesic triangles in a Poincar\'e disc $L\cap\BV$ (or $L\cap\EV$) when $L$ is hyperbolic or usual oriented geodesic triangles in the round sphere when $L$ is spherical. We denote by $\Area\Delta(p_1,p_2,p_3)$ the oriented area of the triangle $\Delta(p_1,p_2,p_3)$ (counterclockwise oriented triangles have positive area) and by $\area\Delta(p_1,p_2,p_3):=\big|\!\Area\Delta(p_1,p_2,p_3)\big|$ its area.

The other ``linear'' geometric objects in the complex hyperbolic plane that will be used later are the metric circles, hypercycles, and horocycles. These are obtained projectivizing an $\mathbb R$-linear subspace $W$ of $V$, $\dim_\mathbb RW=2$, such that the symmetric bilinear form $\real\langle-,-\rangle$, being restricted to $W$, is respectively of signatures $--$/$++$ (metric circles), $-+$ (hypercycles), or $0+$/$0-$ (horocycles)~\cite{SashaGrossi2011}. These linear objects are topological circles that give rise, in the obvious way, to the usual metric circles/hypercycles/horocycles in the hyperbolic discs of the forms $L\cap\BV$ and $L\cap\EV$, where $L$ is a hyperbolic complex line, as well as to the usual metric circles in spherical complex lines.

\subsection{Conjugacy classes and the geometry of isometries}
\label{subsec:isometries}

The group of orientation-preserving isometries of the complex hyperbolic plane $\mathbb H_{\mathbb C}^2$ is $\PU(2,1)$, i.e., the projectivization of
$$\mathrm U(2,1):=\{I\in\mathrm{GL}(3,\mathbb C)\mid\langle Iv,Iw\rangle=\langle v,w\rangle\;\text{for every }v,w\in V\}.$$
Let $\SU(2,1)$ stand for the subgroup in $\mathrm U(2,1)$ consisting of elements of determinant $1$. Clearly,
$$\PU(2,1)=\SU(2,1)/\{1,\omega,\omega^2\},$$
where $\omega:=e^{2\pi i/3}$ is a cube root of unity.
Abusing notation, we will also refer to elements in~$\SU(2,1)$ as isometries and will call its eigenvectors fixed points.

Any isometry in $\PU(2,1)$ fixes at least one point in $\overline{\mathbb H}_{\mathbb C}^2:=\BV\cup\SV$. A rough classification of nonidentical orientation-preserving isometries is obtained by observing that exactly one of the following must occur. The isometry has a negative fixed point, exactly one isotropic fixed point, or exactly two isotropic fixed points; it is respectively called \textit{elliptic,} \textit{parabolic,} and \textit{loxodromic.} As is well-known, each of these rough classes can be refined. We will use several subtypes of elliptic and parabolic isometries in the paper and the geometry of such subtypes is briefly explained below, beginning with the elliptic case.

Let $I$ be an elliptic isometry and let $c\in\BV$ be an $I$-fixed point. Then $I$ stabilizes the spherical complex line with polar point $c$. Clearly, $I$ has a pair $p,\tilde p$ of mutually orthogonal fixed points in the spherical line $\mathbb Pc^\perp$. Hence, we have an orthogonal basis given by eigenvectors of $I$. Let $\mu_1,\mu_2,\mu_3\in{\mathbb C}$ with $\mu_1\mu_2\mu_3=1$ be the eigenvalues of $c,p,\tilde p$, respectively. Since none of $c,p,\tilde p$ is isotropic, we have $|\mu_i|=1$ for $i=1,2,3$. An elliptic isometry is \textit{regular\/} if its eigenvectors have pairwise distinct eigenvalues; otherwise, it is called \textit{special.}

Assume that $I$ is regular elliptic. In $\mathbb H^2_\mathbb C$, this isometry fixes the single point $c$ and stabilizes the pair of orthogonal complex geodesics with polar points $p,\tilde p$. These complex geodesics intersect at~$c$ and it is easy to see that $I$ acts on~$\mathbb Pp^\perp\cap\BV$ as a rotation around $c$ by the angle $\Arg(\mu_1^{-1}\mu_3)$, where the function $\Arg$ takes values in $[0,2\pi)$ (an analogous statement holds for the action of $I$ on $\mathbb P\,\tilde p^\perp\cap\BV$).

Suppose that $I$ is special elliptic. We can rewrite the eigenvalues of $c,p,\tilde p$ as $\alpha^{-2},\alpha,\alpha$ or $\alpha,\alpha^{-2},\alpha$ or $\alpha,\alpha,\alpha^{-2}$ with $|\alpha|=1$ and $\alpha^3\neq 1$. In the first case, the spherical complex line $\mathbb Pc^\bot$ is pointwise fixed by $I$. This implies that each hyperbolic complex line passing through $c$ is~$I$-stable; the isometry acts on the corresponding complex geodesic as a rotation around $c$ by the angle $\Arg(\alpha^3)$. In other words, $I$ can be seen as a rotation around the point $c$. In the second case, the hyperbolic complex line $L:=\mathbb Pp^\perp$ is pointwise fixed by $I$. This implies that every complex line intersecting $L$ orthogonally (i.e., containing $p$) is $I$-stable; the isometry acts on a complex geodesic orthogonal to $L$ as a rotation around the intersection point by the angle $\Arg(\alpha^{-3})$. Such special elliptic isometry can be seen as a rotation around the fixed axis $L$. The third case is similar to the second one.

Every special elliptic isometry can be written in the form
\begin{equation}\label{eq:specialelliptic}
R_\alpha^p:x\mapsto (\alpha^{-2}-\alpha){\frac{\left<x,\,p\right>}{\left<p,\,p\right>}}p+\alpha x
\end{equation}
(see \cite{Mostow80}) for some $p\in\mathbb PV\setminus\SV$ and $\alpha\in\mathbb C$, $|\alpha|=1$. The point $p$ will be called the {\it centre\/} of $R_\alpha^p$ and, $\alpha$, its {\it parameter.} Note that, in the complex hyperbolic plane, $R_{-1}^p$ is a reflection in $p$ when $p\in\BV$ or a reflection in the complex geodesic $\mathbb Pp^\bot\cap\BV$ when $p\in\EV$.

A remark about terminology. In \cite{Mostow80}, a special elliptic $R_{\alpha}^p$ of finite order with positive $p$ is called a {\it complex reflection\/} ($\mathbb C$-reflection). This is in line with calling a finite order element of the general linear group of a complex vector space with a pointwise fixed hyperplane a complex reflection, which is usual nomenclature. Later, the term complex reflection began to be used in the literature to specify any special elliptic with positive $p$ (not necessarily the finite order ones) and those with negative $p$ are sometimes called complex reflections about a point. We chose not to use the terminology complex reflection/complex reflection about a point because we often consider products $R_{\alpha_n}^{p_n}\dots R_{\alpha_1}^{p_1}$ of special elliptic isometries such that some of the $p_i$'s are positive and some are negative. Besides, we would like to emphasize the geometric nature of the isometry, which is (not necessarily that of a reflection but) that of a rotation about a fixed axis (positive $p$) or a rotation about a fixed point (negative $p$).

A parabolic isometry can be either \textit{unipotent\/} or {\it ellipto-parabolic.} Being parabolic unipotent means that the isometry can be lifted to a unipotent element of $\SU(2,1)$. There are two kinds of parabolic unipotent isometries. The first is $3$-step unipotent and possesses no fixed point in~$\mathbb PV$ besides the isotropic one (in a certain sense, it is a ``pure'' parabolic isometry). The second is $2$-step unipotent and has a pointwise fixed Euclidean complex line whose polar point is the isotropic fixed one. Therefore, it stabilizes every complex geodesic passing through its fixed point. In each such complex line (a Poincar\'e disc) it acts as a plane parabolic isometry. So, a $2$-step unipotent isometry looks a little bit like a special elliptic isometry whose pointwise fixed complex line is Euclidean. A parabolic isometry that is not unipotent is called \textit{ellipto-parabolic.} It stabilizes exactly two complex lines: the Euclidean line whose polar point is the isotropic fixed point and a hyperbolic line containing the isotropic fixed point. In the latter, it acts as a plane parabolic isometry. So, in a certain sense, an ellipto-parabolic isometry resembles a regular elliptic isometry as it has a couple of orthogonal stable complex lines.

A useful tool in the study of $\SU(2,1)$-conjugacy classes of an orientation-preserving isometry involves the polynomial $f:\mathbb C\to\mathbb R$ defined by
\begin{equation}\label{eq:goldmandeltoid}
f(z)=|z|^4-8\real(z^3)+18|z|^2-27.
\end{equation}
The preimage $f^{-1}(0)$, known as {\it Goldman's deltoid,} has the parameterization $\zeta^{-2}+2\zeta$, where $\zeta\in\mathbb C$, $|\zeta|=1$. A nonidentical isometry $I\in\SU(2,1)$ is regular elliptic iff $f(\trace I)<0$, loxodromic iff $f(\trace I)>0$, and parabolic unipotent iff $f(\trace I)\in\{3,3\omega,3\omega^2\}$, where $\omega:=e^{2\pi i/3}$. When $\trace I\in f^{-1}(0)\setminus\{3,3\omega,3\omega^2\}$, the isometry can be either special elliptic or ellipto-parabolic. A picture involving Goldman's deltoid can be found in page \pageref{fig:deltoid}.

The description of the $\SU(2,1)$-conjugacy classes of a nonidentical orientation-preserving isometry is as follows \cite{Will2016}. Take $z\in\mathbb C$.

\smallskip

$\bullet$ If $f(z)<0$, there exist exactly three distinct $\SU(2,1)$-conjugacy
classes of isometries of trace~$z$. They are all regular elliptic and each
conjugacy class is determined by the eigenvalue of the negative fixed point.

$\bullet$ If $f(z)=0$ and $z\notin\{3,3\omega,3\omega^2\}$, there exist exactly three distinct
$\SU(2,1)$-conjugacy
classes of isometries of trace $z$. Two of them are special elliptic and they are determined by the signature of the centres. The remaining one is ellipto-parabolic.

$\bullet$ If $z\in\{3,3\omega,3\omega^2\}$, then ($f(z)=0$ and) there exist exactly three distinct $\SU(2,1)$-conjugacy
classes of isometries of trace $z$. One is $3$-step unipotent. The other two are $2$-step unipotent and are determined by their actions on the stable complex geodesics (one moves the nonfixed ideal points in stable complex geodesics in the clockwise sense and, the other, in the counterclockwise sense).

$\bullet$ If $f(z)>0$, there exists exactly one $\SU(2,1)$-conjugacy class of isometries of trace $z$ and it is loxodromic.

\begin{defi}\label{defi:regular}
We call an isometry $I\in\SU(2,1)$ {\it regular\/} if each eigenspace has dimension $=1$, i.e., if $I$ does not have a pointwise fixed complex line.
\end{defi}

The above definition is equivalent to the one in \cite{steinberg1974}. This class of isometries contains the regular elliptic, ellipto-parabolic, $3$-step unipotent, and loxodromic ones. It is particularly useful because the trace of a regular isometry determines its type (and, except for the regular elliptic case, also determines its $\SU(2,1)$-conjugacy class).

Finally we express, in terms of centres and parameters, the trace of the product of two and of three special elliptic isometries since these traces will be needed later. The formulae in the next remark follow from \cite[pg.~195]{Mostow80} (see also \cite{Pratoussevitch2005}).

\begin{rmk}
\label{rmk:partform} Let $p_1,p_2,p_3$ be nonisotropic points and let $\alpha_1,\alpha_2,\alpha_3$ be unit complex numbers. Then
$$\trace R_{\alpha_2}^{p_2}R_{\alpha_1}^{p_1}=
\alpha_1\alpha_2+\alpha_1^{-2}\alpha_2+\alpha_1\alpha_2^{-2}+(\alpha_1^{-2}-\alpha_1)(\alpha_2^{-2}-\alpha_2)\tance(p_1,p_2),$$
\begin{equation*}
\begin{split}
\trace R_{\alpha_3}^{p_3}R_{\alpha_2}^{p_2}R_{\alpha_1}^{p_1}= \alpha_1^{-2}\alpha_2\alpha_3+\alpha_1\alpha_2^{-2}\alpha_3+\alpha_1\alpha_2\alpha_3^{-2}
+ (\alpha_1^{-2}-\alpha_1)(\alpha_2^{-2}-\alpha_2)\alpha_3\tance(p_1,p_2)\\
+(\alpha_1^{-2}-\alpha_1)(\alpha_3^{-2}-\alpha_3)\alpha_2\tance(p_1,p_3)
+(\alpha_2^{-2}-\alpha_2)(\alpha_3^{-2}-\alpha_3)\alpha_1\tance(p_2,p_3)\\
+(\alpha_1^{-2}-\alpha_1)(\alpha_2^{-2}-\alpha_2)(\alpha_3^{-2}-\alpha_3) {\frac{g_{12}g_{23}g_{31}}{g_{11}g_{22}g_{33}}},
\end{split}
\end{equation*}
where $g_{ij}:=\langle p_i,p_j\rangle$.
\end{rmk}

Unless otherwise stated, we consider only isometries in $\SU(2,1)$ and only their $\SU(2,1)$-conjugacy classes.

\section{Relations of length $\leqslant$ $3$}\label{sec:basic}

This section is devoted to the classification of (generic) lengths $2$ and $3$ relations between special elliptic isometries.

In what follows, we will denote the circle of unit complex numbers by $\mathbb S^1\subset\mathbb C$ and the set of cube roots of the unity by $\Omega:=\{1,\omega,\omega^2\}\subset\mathbb S^1$, $\omega:=e^{2\pi i/3}$. When a unit complex number $\alpha\in\mathbb S^1\setminus\Omega$ is meant to play the role of the parameter of a special elliptic isometry we will also call it a parameter.

We begin with a simple remark that will be used throughout the paper without reference.

\begin{rmk}

$\bullet$ Let $R_\alpha^p$ be a special elliptic isometry, $\alpha\in\mathbb S^1\setminus\Omega$. A point $q\neq p$ is fixed by $R_\alpha^p$ iff $\left<p,q\right>=0$. In this case, $R_\alpha^pq=\alpha q$.

\smallskip

$\bullet$ Let $R_{\alpha_1}^{p_1},R_{\alpha_2}^{p_2}$ be special elliptic isometries, $p_1\ne p_2$, and let $c$ be the polar point of the line $\mathrm L(p_1,p_2)$. Then $c$ is fixed by $R:=R_{\alpha_2}^{p_2}R_{\alpha_1}^{p_1}$ since $\mathbb Pp_1^\perp\cap\mathbb Pp_2^\perp=\{c\}$. The $R$-eigenvalue of $c$ equals $\alpha_1\alpha_2$ and the line $\mathrm L(p_1,p_2)$ is $R$-stable.

\smallskip

$\bullet$ We have $R_{\delta\alpha}^p=\delta R_{\alpha}^p$ whenever $\delta\in\Omega$.

\smallskip

$\bullet$ $R_\beta^p R_\alpha^p =R_{\alpha\beta}^p$\/ for every $p\in \mathbb PV \setminus\SV$ and every $\alpha,\beta\in\mathbb S^{1}$.

\end{rmk}

The classification of length $2$ relations between special elliptic isometries is a simple consequence of the above remark. Indeed, on one hand, for nonisotropic $p$ and $\alpha\in\mathbb S^1\setminus\Omega$, we have $R_\alpha^pR_{\delta\overline\alpha}^p=\delta$, where $\delta\in\Omega$. On the other hand, if $R:=R_{\alpha_2}^{p_2}R_{\alpha_1}^{p_1}=\delta$ for some nonisotropic points $p_1,p_2$ and parameters $\alpha_1,\alpha_2\in\mathbb S^1\setminus\Omega$, then either $p_1=p_2$ (and $\alpha_1\alpha_2=\delta$) or $\langle p_1,p_2\rangle=0$. In order to see that the latter is impossible it suffices to apply the isometry $R$ to $p_1$ and to $p_2$ as this leads to $\alpha_1^3=\alpha_2^3=1$, a contradiction. We arrive at the following definition.

\begin{defi}\label{defi:cancellations} The length $2$ relations between special elliptic isometries are called {\it cancellations.} They are of the form $R_\alpha^pR_{\delta\overline\alpha}^p=\delta$ where $p$ is nonisotropic, $\alpha\in\mathbb S^1\setminus\Omega$, and $\delta\in\Omega$.
\end{defi}

In order to obtain all (generic) length 3 relations, we need to understand when the product of two special elliptic isometries is special elliptic.

\begin{lemma}\label{lemma:eigenvalue}
Let\/ $p_1,p_2\in\mathbb PV\setminus\SV$ be distinct nonisotropic points
and let\/ $\alpha_1,\alpha_2\in\mathbb S^1\setminus\Omega$ be parameters. The isometry\/
$R_{\alpha_2}^{p_2}R_{\alpha_1}^{p_1}$ has a\/ fixed point in the line\/
$\mathrm L(p_1,p_2)$ with eigenvalue\/ $\alpha_1\alpha_2$ iff\/ $\mathrm L(p_1,p_2)$
is~Euclidean.
\end{lemma}

\begin{proof}
Assume that $R:=R_{\alpha_2}^{p_2}R_{\alpha_1}^{p_1}$ fixes a point $p\in L:=\mathrm L(p_1,p_2)=\mathbb Pc^\perp$ with eigenvalue $\alpha_1\alpha_2$ and that $L$ is noneuclidean. Then $p\ne c$, $Rp=\alpha_1\alpha_2p$, and $Rc=\alpha_1\alpha_2c$; hence, $\trace R=2\alpha_1\alpha_2+\alpha_1^{-2}\alpha_2^{-2}$. It follows from Remark \ref{rmk:partform} that $\tance(p_1,p_2)=1$ which implies that $L$ is Euclidean, a contradiction. The converse is immediate. \end{proof}

\begin{prop}\label{prop:r2r1special}
Let\/ $p_1,p_2\in \mathbb PV\setminus\SV$ be distinct nonisotropic points
such that the line\/ $\mathrm L(p_1,p_2)$ is noneuclidean. Let\/
$\alpha_1,\alpha_2\in\mathbb S^1\setminus\Omega$ be parameters. Then\/
$R_{\alpha_2}^{p_2}R_{\alpha_1}^{p_1}$ is special elliptic iff\/
$\langle p_1,p_2\rangle=0$ and\/ $\alpha_1=\delta\alpha_2$ for some\/ $\delta\in\Omega$.
\end{prop}

\begin{proof}
Let $R:=R_{\alpha_2}^{p_2}R_{\alpha_1}^{p_1}$ and let $c$ be the polar point of the
$R$-stable line $L:=\mathrm L(p_1,p_2)$. Assume that the isometry $R$ is special
elliptic and let $L'$ be its pointwise fixed line. If $L'\ne L$, then $c\in L'$ and
the intersection $L\cap L'$ is a fixed point of $R$ with eigenvalue $\alpha_1\alpha_2$.
This is impossible by Lemma~\ref{lemma:eigenvalue}. So, $L'=L$. Being a fixed point of
$R$, the point $p_2$ is also a fixed point of $R_{\alpha_1}^{p_1}$. Since $p_1\ne p_2$,
we obtain $\langle p_1,p_2\rangle=0$. The $R$-eigenvalues of $p_1,p_2$ are respectively
$\alpha_1^{-2}\alpha_2$, $\alpha_1\alpha_2^{-2}$. These eigenvalues are equal, that is,
$\alpha_1^3=\alpha_2^3$. The converse is immediate.
\end{proof}

Let $p_1,p_2,p_3\in\mathbb PV\setminus\SV$ be nonisotropic pairwise distinct points such that the lines $\mathrm L(p_i,p_j)$ are noneuclidean for distinct $i,j\in\{1,2,3\}$. Let $\alpha_1,\alpha_2,\alpha_3\in\mathbb S^1\setminus\Omega$ be parameters. It is easy to see that, if $p_1,p_2,p_3$ are pairwise orthogonal and
$\alpha_1^{-2}\alpha_2\alpha_3=\alpha_1\alpha_2^{-2}\alpha_3=\alpha_1\alpha_2\alpha_3^{-2}=\delta$,
then $R_{\alpha_3}^{p_3}R_{\alpha_2}^{p_2}R_{\alpha_1}^{p_1}=\delta$. Conversely, assuming $R:=R_{\alpha_3}^{p_3}R_{\alpha_2}^{p_2}R_{\alpha_1}^{p_1}=\delta$, it follows from Proposition \ref{prop:r2r1special} that $p_1,p_2,p_3$ are pairwise orthogonal; applying $R$ to $p_1,p_2,p_3$, we obtain $\alpha_1^{-2}\alpha_2\alpha_3=\alpha_1\alpha_2^{-2}\alpha_3=\alpha_1\alpha_2\alpha_3^{-2}=\delta$. We have just arrived at the following definition.

\begin{defi}
The length $3$ (generic) relations between special elliptic isometries are called {\it{\rm(}length\/ $3${\rm)} orthogonal relations.} They are of the form
$R_{\alpha}^{p_3}R_{\delta_2\alpha}^{p_2}R_{\delta_1\alpha}^{p_1}=\delta$
with pairwise orthogonal nonisotropic $p_1,p_2,p_3$, where $\alpha\in\mathbb S^1\setminus\Omega$, $\delta_1,\delta_2,\delta\in\Omega$, and $\delta_1\delta_2=\delta$.
\end{defi}

\section{Relations of length $4$: bendings}\label{sec:bendings}

As in the case of relations of lengths $2$ and $3$, a length $4$ relation imposes restrictive conditions on centres and parameters.

Given a length $4$ relation, we write it in the form
$R_{\alpha_2}^{p_2}R_{\alpha_1}^{p_1}=\delta R_{\beta_2}^{q_2}R_{\beta_1}^{q_1}$
(of course, we assume that $p_1\neq p_2$ and $q_1\neq q_2$; as usual, $\delta\in\Omega$
is a cube root of unity). If $p_1,p_2,q_1,q_2$ lie in a same complex line $L$, then
$\alpha_1\alpha_2=\delta\beta_1\beta_2$ because the polar point $c$ of $L$ is a fixed
point of $R_{\alpha_2}^{p_2}R_{\alpha_1}^{p_1}$ with eigenvalue $\alpha_1\alpha_2$ and
a fixed point of $\delta R_{\beta_2}^{q_2}R_{\beta_1}^{q_1}$ with eigenvalue $\delta\beta_1\beta_2$.
Generically, the converse also holds.

\begin{lemma}\label{lemma:ortnotortl4}
Let\/ $p_i,q_i \in\mathbb PV\setminus\SV$ be nonisotropic points and
let\/ $\alpha_i,\beta_i\in\mathbb S^1\setminus\Omega$ be parameters, $i=1,2$. Assume
that\/ $p_1\neq p_2$ and\/ $q_1\neq q_2$, that
\begin{equation}\label{eq:relationl4}
R_{\alpha_2}^{p_2}R_{\alpha_1}^{p_1}=\delta R_{\beta_2}^{q_2}R_{\beta_1}^{q_1}
\end{equation}
for some\/ $\delta\in\Omega$, and that at least one of the lines\/ $L_1:=\mathrm{L}(p_1,p_2)$, $L_2:=\mathrm{L}(q_1,q_2)$ is noneuclidean. Then\/ $L_1$ and\/ $L_2$ are either equal\/ {\rm(}in which case\/ $\alpha_1\alpha_2=\delta\beta_1\beta_2${\rm)} or orthogonal\/ {\rm(}in which case\/ $\alpha_1\alpha_2\neq\delta\beta_1\beta_2${\rm).}
\end{lemma}

\begin{proof}
Let $c$, $d$ denote respectively the polar points of $L_1$,	$L_2$. As observed above, $L_1=L_2$ implies $\alpha_1\alpha_2=\delta\beta_1\beta_2$. Conversely, assume $\alpha_1\alpha_2=\delta\beta_1\beta_2$ and $L_1\ne L_2$ (i.e., $c\neq d$). The relation
$R_{\alpha_2}^{p_2}R_{\alpha_1}^{p_1}=\delta R_{\beta_2}^{q_2}R_{\beta_1}^{q_1}$
implies that the line $L:=\mathrm L(c,d)$ is pointwise fixed by
$R_{\alpha_2}^{p_2}R_{\alpha_1}^{p_1}$ with eigenvalue $\alpha_1\alpha_2$.
In particular, the intersections $L_1\cap L$ provides a point in~$L_1$ which is
fixed by $R_{\alpha_2}^{p_2}R_{\alpha_1}^{p_1}$ with eigenvalue
$\alpha_1\alpha_2$. Similarly, we obtain a point in $L_2$ which is fixed by
$\delta R_{\beta_2}^{q_2}R_{\beta_1}^{q_1}$ with eigenvalue
$\delta\beta_1\beta_2$. Thus, by Lemma~\ref{lemma:eigenvalue}, both lines $L_1$,
$L_2$ are Euclidean, a contradiction.

Assume $L_1$ orthogonal to $L_2$. Hence, $L_1\ne L_2$ because at least one of these lines is noneuclidean and, therefore, $\alpha_1\alpha_2\ne\delta\beta_1\beta_2$. Finally, assume that $L_1$ is not orthogonal to $L_2$. By \eqref{eq:specialelliptic}, the relation
$R_{\alpha_2}^{p_2}R_{\alpha_1}^{p_1}=\delta R_{\beta_2}^{q_2}R_{\beta_1}^{q_1}$
implies
\begin{equation*}
\delta\beta_1\beta_2d=R_{\alpha_2}^{p_2}R_{\alpha_1}^{p_1}d=
\alpha_1\alpha_2d+\kappa_1\frac{\left<d,p_1\right>}{\left<p_1,p_1\right>}p_1+
\bigg(\kappa_2\frac{\left<d,p_2\right>}{\left<p_2,p_2\right>}+
\kappa_3\frac{\left<d,p_1\right>\left<p_1,p_2\right>}
{\left<p_1,p_1\right>\left<p_2,p_2\right>}\bigg)p_2,
\end{equation*}
where $\kappa_1=(\alpha_1^{-2}-\alpha_1)\alpha_2$, $\kappa_2=(\alpha_2^{-2}-\alpha_2)\alpha_1$, and $\kappa_3=(\alpha_1^{-2}-\alpha_1)(\alpha_2^{-2}-\alpha_2)$. Since $d\notin L_1$, the points $d,p_1,p_2$ are $\mathbb C$-linearly independent and $\alpha_1\alpha_2=\delta\beta_1\beta_2$.
\end{proof}

Note that, in the previous lemma, $L_1=L_2$ is equivalent to the projective lines $M_1:=\mathbb Pp_1^\perp$, $M_2:=\mathbb Pp_2^\perp$, $M_3:=\mathbb Pq_1^\perp$, $M_4:=\mathbb Pq_2^\perp$ having a common point in $\mathbb PV$. Similarly, $L_1$ orthogonal to $L_2$ is equivalent to the orthogonality of $m_{i_1i_2}$ and $m_{i_3i_4}$, where $m_{ij}:=M_i\cap M_j$ and $\{i_1,i_2,i_3,i_4\}=\{1,2,3,4\}$.

\smallskip

In this section, we focus on \textit{nonorthogonal\/} relations of length $4$.
In view of Lemma~\ref{lemma:ortnotortl4}, this means that we will study relations of the form \eqref{eq:relationl4} satisfying $\mathrm{L}(p_1,p_2)=\mathrm{L}(q_1,q_2)$ or, equivalently, $\alpha_1\alpha_2=\delta\beta_1\beta_2$. We~also assume $\delta=1$, $\alpha_i=\beta_i$, and $\sigma p_i=\sigma q_i$, $i=1,2$, where the signature $\sigma p$ of a nonisotropic point $p$ equals $1$ when $p\in\mathrm{E}V$ and $-1$ when $p\in\mathrm{B}V$ (see the beginning of Section~\ref{sec:hyperbolicgeometry}). In Section \ref{sec:n-o-solutions} we will consider the general case.

Let us apply a known recipe to produce length $4$ relations \cite{Sasha2012}, \cite{ABG2007}. Take an isometry $C$ in the centralizer of $R_{\alpha_2}^{p_2}R_{\alpha_1}^{p_1}$. Then
$$R_{\alpha_2}^{p_2}R_{\alpha_1}^{p_1}=(R_{\alpha_2}^{p_2}R_{\alpha_1}^{p_1})^C=(CR_{\alpha_2}^{p_2}C^{-1})(CR_{\alpha_1}^{p_1}C^{-1})=R_{\alpha_2}^{Cp_2}R_{\alpha_1}^{Cp_1}.$$
Relations of this form are called {\it bending relations.} All length $4$ relations of the form \eqref{eq:relationl4} with $\delta=1$, $\alpha_1=\beta_1,\alpha_2=\beta_2$, and $\sigma p_i=\sigma q_i$, $i=1,2$, are bending relations (Theorem \ref{thm:bendingsarebendings}).

The fact that the line $L:=\mathrm L(p_1,p_2)$ is stable under the isometry $R:=R_{\alpha_2}^{p_2}R_{\alpha_1}^{p_1}$ will allow us to prove that it is not necessary to consider the full centralizer of $R$ in order to obtain all bending relations (indeed, it suffices to take a one-parameter subgroup of this centralizer). First, we need the following lemma.

\begin{lemma}
\label{lemma:prodspecreg}
Let\/ $p_1,p_2\in\mathbb PV\setminus \SV$ be distinct nonorthogonal
points with noneuclidean\/ $L:=\mathrm L(p_1,p_2)$ and let\/ $\alpha_1,\alpha_2\in\mathbb S^1\setminus\Omega$ be parameters. Then
$R:=R_{\alpha_2}^{p_2}R_{\alpha_1}^{p_1}$ is regular {\rm(}and is not\/ $3$-step unipotent\/{\rm).}
\end{lemma}

\begin{proof}
Proposition~\ref{prop:r2r1special} implies that $R$ is not special elliptic. We can assume that $L$ is hyperbolic since, otherwise, $R$ is regular elliptic. Suppose that $R$ is parabolic with isotropic fixed point $v\in L$. By Lemma \ref{lemma:eigenvalue}, the eigenvalue of $v$ is not $\alpha_1\alpha_2$. Hence, $R$ has two distinct eigenvalues (that of $v$ and that of the polar point of $L$) with distinct eigenvectors. It follows that $R$ cannot be parabolic unipotent.
\end{proof}

\begin{prop}\label{prop:bendings}
Let\/ $p_1,p_2 \in\mathbb PV\setminus\SV$ be distinct
nonorthogonal points such that\/ $L:=\mathrm L(p_1,p_2)$ is noneuclidean and
let\/ $\alpha_1,\alpha_2\in\mathbb S^1\setminus\Omega$ be parameters. There exists
a one-parameter subgroup\break $B:\mathbb R\to\SU(2,1)$ such that\/ $B(s)$
commutes with\/ $R_{\alpha_2}^{p_2}R_{\alpha_1}^{p_1}$ and\/
$R_{\alpha_2}^{B(s)p_2}R_{\alpha_1}^{B(s)p_1}=
R_{\alpha_2}^{p_2}R_{\alpha_1}^{p_1}$ for every\/ $s\in\mathbb R$.
Furthermore, given\/ $I$ in the centralizer of\/
$R_{\alpha_2}^{p_2}R_{\alpha_1}^{p_1}$, there exists\/ $s\in\mathbb R$ such
that the equality\/ $Ip=B(s)p$ holds in\/ $\mathbb PV$ for every\/
$p\in L$.
\end{prop}

\begin{proof}
Let $c$ stand for the polar point of $L:=\mathrm L(p_1,p_2)$ and let
$R:=R_{\alpha_2}^{p_2}R_{\alpha_1}^{p_1}$. We will use the description of the full centralizer of $R$ given in \cite[Corollary~8.2]{BasmajianMiner1998} and \cite[Theorem~1.1]{CaoG2011} to explicitly obtain the required one-parameter subgroup.

Assume that the isometry $R$ is elliptic. By Lemma \ref{lemma:prodspecreg}, it is regular elliptic. Then $R$ fixes $c$ as well as two other points $p,q\in L$, $\langle p,q\rangle=0$. An isometry $I\in\SU(2,1)$ commutes with $R$ iff it is an elliptic isometry that fixes $c,p,q$. So, in the orthogonal basis $c,p,q$, such an isometry can be written in the form $I=I(\mu_1,\mu_2):=\left[\begin{smallmatrix} \mu_1 & 0 & 0 \\ 0 & \mu_2 & 0\\ 0 & 0 & \mu_1^{-1}\mu_2^{-1}\end{smallmatrix}\right]$
with $\mu_i\in\mathbb C$, $|\mu_i|=1$. It is not difficult to see that the actions
of $I(\mu_1,\mu_2)$ and $I(\mu_1',\mu_2')$ on $L$ are equal iff
$\mu_1\mu_2^2=\mu_1'\mu_2'^2$. So, it suffices to take the one-parameter subgroup
$B:\mathbb R\to\SU(2,1)$ defined, in the orthogonal basis $c,p,q$, by
$B(s):=\left[ \begin{smallmatrix} 1 & 0 & 0 \\ 0 & e^{\frac{s}{2}i} & 0 \\
0 & 0 & e^{-\frac{s}{2}i} \end{smallmatrix}\right]$.
The equality $I(\mu_1,\mu_2)|_L=B\big(\Arg(\mu_1\mu_2^2)\big)|_L$ holds in
$\mathbb PV$. Indeed, every point in $L\setminus\{q\}$ has a representative of the form $p+\lambda q$ for some $\lambda\in\mathbb C$  and, in $\mathbb PV$, we have
$$I(p+\lambda q)=\mu_2 p+\mu_1^{-1}\mu_2^{-1}\lambda q=\mu_1\mu_2^2p+\lambda q=e^{i\Arg(\mu_1\mu_2^2)}p+\lambda q=B\big(\Arg(\mu_1\mu_2^2)\big)(p+\lambda q).$$

Suppose that $R$ is loxodromic. Then it fixes two isotropic points $v_1,v_2\in L$.
An isometry $I\in\SU(2,1)$ commutes with $R$ iff
it is loxodromic or special elliptic and
$R\big(\mathrm{fix}(I)\big)\subset\mathrm{fix}(I)$. Thus,
$I=I(\mu):=\left[\begin{smallmatrix} \overline{\mu}\mu^{-1} & 0 & 0 \\
0 & \mu & 0\\ 0 & 0 & \overline{\mu}^{-1}\end{smallmatrix}\right]$
in the basis $c,v_1,v_2$, where $0\neq\mu\in\mathbb C$. The actions of
$I(\mu_1)$ and $I(\mu_2)$ on $L$ are equal iff
$\mu_1\overline\mu_1=\mu_2\overline\mu_2$. Hence, we can take the one-parameter
subgroup $B:\mathbb R\to\SU(2,1)$ defined, in the basis $c,v_1,v_2$, by
$B(s):=\left[ \begin{smallmatrix} 1 & 0 & 0 \\ 0 & e^s & 0 \\
0 & 0 & e^{-s}\end{smallmatrix}\right]$.
Clearly, $I(\mu)|_L=B\big(\ln|\mu|\big)|_L$ in $\mathbb PV$.

It remains to consider the case when $R$ is parabolic. By Lemma \ref{lemma:prodspecreg}, $R$ is regular and is not $3$-step unipotent. It fixes an isotropic point $v_1$ and the polar point $c$ of the noneuclidean complex line $L$. We obtain $v_1\in L$. Hence, this line must be hyperbolic and $c$, positive. Let $v_2\in L\cap\SV$, $v_2\neq v_1$. An isometry $I\in\SU(2,1)$ commutes
with $R$ iff $I$ is parabolic with $\mathrm{fix}(R)=\mathrm{fix}(I)$ or $I$ is
special elliptic with $R\big(\mathrm{fix}(I)\big)\subset\mathrm{fix}(I)$. Moreover, by
\cite[Corollary 3.3]{Parker2012}, every eigenvalue of $I$ is unitary. It follows that
$I=I(\mu,t):=\left[\begin{smallmatrix} \mu & it\mu & 0 \\
0 & \mu & 0\\ 0 & 0 & \overline{\mu}^{2}\end{smallmatrix}\right]$
in the basis $v_1,v_2,c$, where $\mu\in\mathbb C$, $|\mu|=1$, and $t\in\mathbb R$.
The actions of $I(\mu_1,t_1)$ and $I(\mu_2,t_2)$ on $L$ are the same iff $t_1=t_2$.
So, we can consider the one-parameter subgroup $B:\mathbb R\to\SU(2,1)$ defined,
in the basis $v_1,v_2,c$, by
$B(s)=\left[ \begin{smallmatrix} 1 & is & 0 \\ 0 & 1 & 0
\\ 0 & 0 & 1\end{smallmatrix}\right]$. We~have $I(\mu,t)|_L=B(t)|_L$ in
$\mathbb PV$.
\end{proof}

We call the elements of the one-parameter subgroups $B$ introduced in Proposition~\ref{prop:bendings} {\it bendings\/} (see Figure \ref{fig:bendings}).

\begin{figure}[h!]
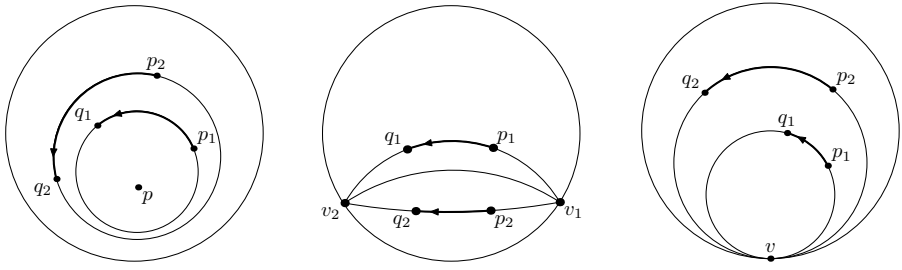

\centering
\includegraphics[scale=.77]{figs/bre01-1.mps}
\hspace{0.5cm}
\includegraphics[scale=.77]{figs/bh01-1.mps}
\hspace{0.5cm}
\includegraphics[scale=.77]{figs/bp01-1.mps}
\caption{Action of bendings in $\mathrm L(p_1,p_2)$ in the elliptic, loxodromic, and parabolic cases. The orbits of $p_1,p_2$ are respectively metric circles, hypercycles, and horocycles.}
\label{fig:bendings}
\end{figure}

In order to prove in Theorem \ref{thm:bendingsarebendings} that bendings provide all length 4 relations of the form \eqref{eq:relationl4} with $\delta=1$, $\alpha_1=\beta_1,\alpha_2=\beta_2$, and $\sigma p_i=\sigma q_i$, $i=1,2$, we express the action of $R_{\alpha_2}^{p_2}R_{\alpha_1}^{p_1}$ on the line $\mathrm L(p_1,p_2)$ as a product of reflections on geodesics. First, we introduce some notation and terminology.

Let $\omega:=e^{2\pi i/3}$, let
$I_0:=\big\{e^{i\theta}\mid0<\theta<2\pi/3\big\}\subset\mathbb S^1$, let $I_1:=\omega I_0$, and let $I_2:=\omega I_1=\omega^2 I_0$. Let $J\subset I_0$ stand for the open arc limited by $1$ and $-\omega^2$. Note that $J^2=I_0$ and $J^6=I_0^3=\mathbb S^1\setminus\{1\}$.

\begin{defi}
\label{defi:samecomponent}
Given a parameter $\alpha\in\mathbb S^1\setminus\Omega$, the (unique) complex number $a\in J$ satisfying $(a^2)^3=\alpha^3$ is called the {\it primitive\/} of $\alpha$. Two parameters $\alpha,\beta\in\mathbb S^1\setminus\Omega$ are in a {\it same component\/} when they lie in a same arc $I_j$. In this case, we write $\alpha\sim\beta$.
\end{defi}

Obviously, being on a same component
is an equivalence relation. Moreover, whenever $\alpha\sim\beta$, we have
$\overline\alpha\sim\overline\beta$ and $\omega^k\alpha\sim\omega^k\beta$
for $k\in\mathbb Z$.

\begin{lemma}\label{lemma:construction}
Let\/ $p_1,p_2\in\mathbb PV\setminus\SV$ be distinct nonorthogonal points, let\/
$\alpha_1,\alpha_2\in\mathbb S^1\setminus\Omega$ be parameters, and let
$a_1,a_2\in J$ be the primitives of~$\alpha_1,\alpha_2$.
Let\/ $L:=\mathrm L(p_1,p_2)$ and let\/ $G\subset L$ stand for the geodesic
through $p_1,p_2$. Then\/ $R_{\alpha_2}^{p_2}R_{\alpha_1}^{p_1}$ acts on\/
$L$ as the product\/ $r_2r_1$, where\/ $r_i$ is the reflection in the
geodesic\/ $G_i$ through $p_i$ such that\/ $\angle_{p_1}G_1G=\Arg(a_1^3)$ and
$\angle_{p_2}GG_2=\Arg(a_2^3)$, $i=1,2$.
\end{lemma}

\begin{proof}
Let $r$ stand for the reflection in the geodesic $G$. Since the restriction of $R_{\alpha_1}^{p_1}$ to $L$ is a usual rotation (in a plane geometry) around $p_1$ by the angle $\Arg(\alpha_1^3)$, the isometry $R_{\alpha_1}^{p_1}$ acts on $L$ as $rr_1$: this product of reflections is a rotation with centre $p_1$ and angle $2\Arg(a_1^3)=\Arg(a_1^3)+\Arg(a_1^3)=\Arg(a_1^2)^3=\Arg(\alpha_1^3)$. (Note that $a_1\in J$ implies $a_1^2\in I_0$; so, $0<\Arg a_1<\pi/3$, that is, $0<\Arg a_1^3<\pi$.) Analogously, $R_{\alpha_2}^{p_2}$ acts on $L$ as $r_2r$. Therefore, being restricted to $L$, the isometry $R_{\alpha_2}^{p_2}R_{\alpha_1}^{p_1}$ equals $r_2rrr_1=r_2r_1$.
\end{proof}

\begin{defi}\label{defi:geoassociated}
We will refer to the geodesics $G, G_1, G_2$ obtained in the previous lemma as those {\it associated\/} to $R_{\alpha_2}^{p_2}R_{\alpha_1}^{p_1}$.
\end{defi}

\begin{thm}
\label{thm:bendingsarebendings}
Let $p_1,p_2$ and $q_1,q_2$ be pairs of distinct nonisotropic nonorthogonal points and
let\/ $\alpha_1,\alpha_2\in\mathbb S^1\setminus\Omega$ be parameters.
Assume that $\mathrm L(p_1,p_2)$ is noneuclidean and that\/ $\sigma p_i=\sigma q_i$, $i=1,2$.
The relation\/ $R_{\alpha_2}^{p_2}R_{\alpha_1}^{p_1}=
R_{\alpha_2}^{q_2}R_{\alpha_1}^{q_1}$ implies that there exists\/ $s\in\mathbb R$
such that\/ $B(s)p_1=q_1$ and\/ $B(s)p_2=q_2$, where\/ $B$ stands for the
one-parameter subgroup introduced in\/ {\rm Proposition~\ref{prop:bendings}.}
\end{thm}

\noindent{\it Proof.}
Let $L_1:=\mathrm L(p_1,p_2)$ and $L_2:=\mathrm L(q_1,q_2)$. By Lemma~\ref{lemma:ortnotortl4}, $L_1=L_2=:L$. Let $G,G_1,G_2$ and $H,H_1,H_2$ be the geodesics respectively associated to $R_{\alpha_2}^{p_2}R_{\alpha_1}^{p_1}$ and to $R_{\alpha_2}^{q_2}R_{\alpha_1}^{q_1}$, (see Definition~\ref{defi:geoassociated}). Let $r_i$ and $r_j'$ stand for the reflections on $G_i$ and $H_j$, $i,j=1,2$. By Proposition~\ref{prop:r2r1special}, $R:=R_{\alpha_2}^{p_2}R_{\alpha_1}^{p_1}=R_{\alpha_2}^{q_2}R_{\alpha_1}^{q_1}$ cannot be special elliptic.

Since $r_2r_1$ and $r_2'r_1'$ are equal on $L$, by bending $R_{\alpha_2}^{q_2}R_{\alpha_1}^{q_1}$ (moving the pair of points $q_1,q_2$ and the geodesics $H,H_1,H_2$) we arrive at the configuration where $H_i=G_i$ and, in particular, $q_i\in G_i$, $i=1,2$. If, in this new configuration, $q_1=p_1$ or $q_2=p_2$, we are done because the relation implies that $q_1=p_1$ iff $q_2=p_2$.

Suppose that $L$ is hyperbolic.

Bending $R_{\alpha_2}^{p_2}R_{\alpha_1}^{p_1}$ if necessary, we can assume $\mathrm G[p_1,q_1]\cap\mathrm G[p_2,q_2]=\varnothing$ (here, $\mathrm G[p_i,q_i]$ stands for the geodesic segment joining $p_i$ and $q_i$, see Section \ref{sec:hyperbolicgeometry}). There are three cases to consider.

First case: $\sigma p_1=\sigma p_2$ and $\mathrm G[p_1,p_2]\cap\mathrm G[q_1,q_2]=\varnothing$. Note that neither $\mathrm G[p_1,q_1]$ nor $\mathrm G[p_2,q_2]$ contains a fixed point of $R$. Indeed, assuming the contrary, a fixed point $x$ of $R$ should be in exactly one of the segments $\mathrm G[p_1,q_1]$, $\mathrm G[p_2,q_2]$ since $\mathrm G[p_1,q_1]\cap\mathrm G[p_2,q_2]=\varnothing$ and $x\in G_1\cap G_2$. This leads to
$$\area\Delta(x,p_1,p_2)+\area\Delta(x,q_1,q_2)=\pi-(A_1+A_2+\pi-B)+\pi-(\pi-A_1+\pi-A_2+B)=-\pi,$$
where $A_1=\angle_{p_1}G_1G=\angle_{q_1}H_1H$, $A_2=\angle_{p_2}GG_2=\angle_{q_2}HH_2$,  $B:=\angle_x{G_1}{G_2}$, and ``$\area$'' stands for the usual hyperbolic area (see Section \ref{sec:hyperbolicgeometry}), a contradiction. It follows that we are in the configuration illustrated in item~({\it a\/}) of the figure below and
$$\area(p_1,p_2,q_2,q_1)=2\pi-(A_1+\pi-A_1+A_2+\pi-A_2)=0,$$
where $A_1$, $A_2$ are defined as above and $\area(p_1,p_2,q_2,q_1)$ stands for the usual hyperbolic area of the geodesic quadrilateral with vertices $p_1,p_2,q_2,q_1$. This implies $p_i=q_i$, $i=1,2$.

\begin{figure}[!htb]
\centering
\includegraphics[scale=.8]{figs/hypercase1.mps}
\hspace{.5cm}
\includegraphics[scale=.8]{figs/hypercase2.mps}
\hspace{0.5cm}
\includegraphics[scale=.8]{figs/hypercase3.mps}
\hspace{.5cm}
\includegraphics[scale=.8]{figs/sphericalcase.mps}
\label{fig:firstcase}
\end{figure}

Second case: $\sigma p_1=\sigma p_2$ and $\mathrm G[p_1,p_2]\cap\mathrm G[q_1,q_2]=\{y\}$. As in the proof of the first case, neither $\mathrm G[p_1,q_1]$ nor $\mathrm G[p_2,q_2]$ contains a fixed point of $R$ and we are in the configuration illustrated in item~({\it b\/}) of the above figure. Let $B:=\angle_yHG$ (respectively, $B:=\angle_yGH$) when the triangle of vertices $p_1,q_1,y$ is clockwise (respectively, counterclockwise) oriented. We have
$$\area\Delta(p_1,q_1,y)=\pi-(A_1+\pi-A_1+B)=-B$$
which implies $B=0$, that is, $p_i=q_i$, $i=1,2$.

Third case: $\sigma p_1\neq \sigma p_2$. Take the points $\tilde p_2,\tilde q_2\in L$ respectively orthogonal to $p_2,q_2$ and proceed as above taking into account that $\angle_{\tilde p_2}G_2G=\angle_{p_2}GG_2$, $\angle_{\tilde q_2}H_2H=\angle_{q_2}HH_2$ (see item~({\it c\/}) of the above figure).
		
Finally, suppose that $L$ is spherical.

Let $x\in L$ be the fixed point of $R$ such that the triangle $\Delta(x,p_1,p_2)$ is clockwise oriented; note that $x\ne p_i$ and that $\langle x,p_i\rangle\ne0$, $i=1,2$, so $\Delta(x,p_1,p_2)$ is well-defined. For the same reasons, the triangle $\Delta(x,q_1,q_2)$ is also well-defined. Assume that $\Delta(x,q_1,q_2)$ is counterclockwise oriented. This implies (see item ({\it d\/}) of the above figure) that $\Delta(x,q_1,q_2)$ is congruent to the polar triangle of $\Delta(x,p_1,p_2)$ (by definition, a vertex of the triangle polar to the triangle $T$ is the pole of a side $s$ of $T$ in the hemisphere which contains the vertex opposite to $s$). A well-known fact from spherical geometry states that the length of a side in a triangle and the length of the corresponding side in its polar triangle add to $\pi$. Together with the fact that $\tance(p_1,p_2)=\tance(q_1,q_2)$ (see Remark \ref{rmk:partform}), this implies that $\mathrm{dist}(p_1,p_2)=\pi/2$, where $\mathrm{dist}$ stands for the usual distance in the sphere $L$. It follows that the triangles $\Delta(x,q_1,q_2)$ and $\Delta(x,\tilde p_1,p_2)$ are congruent (note that $\dist(q_1,q_2)=\dist(\tilde p_1,p_2)=\pi/2$), where $\tilde p_1\in L$ stands for the point in $L$ orthogonal do $p_1$. Therefore, $\dist(q_2,x)=\dist(x,p_2)$ which implies that, bending $R_{\alpha_2}^{q_2}R_{\alpha_1}^{q_1}$, we can make $p_2=q_2$ and, consequently, $p_1=q_1$. This contradicts the orientation of $\Delta(x,q_1,q_2)$. We conclude that both $\Delta(x,q_1,q_2)$ and $\Delta(x,p_1,p_2)$ have the same orientation, so $\Delta(x,q_1,q_2)$ and $\Delta(x,p_1,p_2)$ have the same internal angles and are therefore congruent.~\hfill$\square$

\begin{prop}
\label{prop:sameanglepoints}
Let\/ $B$ stand for the one-parameter subgroup introduced in\/
{\rm Proposition \ref{prop:bendings}.} Nonisotropic distinct points\/ $p_1$
and\/ $p_2$ with the same signature, not lying in a same Euclidean line, 
are in a same\/ $B$-orbit of a bending of\/
$R_{\alpha_2}^{p_2}R_{\alpha_1}^{p_1}$ iff\/ $\alpha_1=\delta\alpha_2$ for
some\/ $\delta\in\Omega$.
\end{prop}

\noindent{\it Proof.}
First, assume that $R:=R_{\alpha_2}^{p_2}R_{\alpha_1}^{p_1}$ is regular elliptic and
let $x$ be the $R$-fixed point in the line $L:=\mathrm L(p_1,p_2)$ with
$\sigma x=\sigma p_1=\sigma p_2$.
So, we can consider the triangle $\Delta(p_1,x,p_2)$ in the line $L$. 
Moreover, the internal angles of $\Delta(p_1,x,p_2)$ at $p_1$ and at $p_2$ are equal 
iff $\dist(x,p_1)=\dist(x,p_2)$, that is, iff $p_1,p_2$ lie in a same $B$-orbit of $R$ 
(a metric circle centred at $x$). But the mentioned internal angles are equal
exactly when $\alpha_1=\delta\alpha_2$ because the primitive angles $a(\beta_1)$ and
$a(\beta_2)$ (see Definition~\ref{defi:samecomponent}) are equal iff
$\beta_1\beta_2^{-1}\in\Omega$, where $\beta_1,\beta_2\in\mathbb S^1\setminus\Omega$.

Assume that $R$ is parabolic. In this case, the line $L$ is hyperbolic.
Consider the ideal triangle $\Delta(p_1,x,p_2)$, where $x\in L$ is the isotropic
$R$-fixed point. Let $\Gamma$ be the geodesic through $x$ orthogonal to 
$\mathrm{G}{\wr}p_1,p_2{\wr}$. If~$\alpha_1=\delta\alpha_2$, the internal angles at 
$p_1$ and $p_2$ are equal. Therefore, by continuity, $\Gamma$ intersects the geodesic 
segment $\mathrm G[p_1,p_2]$; let $y$ stand for such intersection point. 
AAA congruence implies that $\dist(y,p_1)=\dist(y,p_2)$. So, the reflection in 
$\Gamma$ sends $p_1$ to $p_2$. But the reflection in $\Gamma$ stabilizes 
every horocycle centred at $x$. This implies that $p_1$ and $p_2$ lie in a same 
$B$-orbit of $R$. Conversely, assume that $p_1,p_2$ lie in a same horocycle 
centered at $x$. Clearly, $p_1$ and $p_2$ belong to distinct half-spaces determined
by $\Gamma$. So, the reflection in $\Gamma$ sends $p_1$ to $p_2$ and this implies
that the internal angles of $\Delta (p_1,x,p_2)$ at $p_1$ and at $p_2$ are equal.

\begin{wrapfigure}{r}{0.25\textwidth}
\centering
\includegraphics[scale=.9]{figs/loxodromic-fig1.mps}\\
\vspace{.4cm}
\includegraphics[scale=.9]{figs/loxodromic-fig4.mps}
\end{wrapfigure}

Finally, suppose that $R$ is loxodromic. Again, the line $L$ is hyperbolic.
Consider the geodesics $G,G_1,G_2$ associated to $R$ (see 
Definition~\ref{defi:geoassociated}).
By Lemma~\ref{lemma:construction}, the geodesics $G_1,G_2$ are ultraparallel.
Let $H$ be the geodesics in $L$ that is simultaneously orthogonal to both $G_1$ 
and $G_2$, and let $x,y\in L$ be such that $x\in H\cap G_1$, $y\in H\cap G_2$, and
$\sigma x=\sigma y=\sigma p_1=\sigma p_2$.  
If $\alpha_1=\delta\alpha_2$, it is not difficult to see that the quadrilateral 
$p_1,p_2,y,x$ is convex (otherwise, $A_1=A_2$ is impossible). Let $\Gamma$ be the 
geodesic intersecting the geodesic segment $\mathrm G[x,y]$ orthogonally in its middle 
point $q_1$. Continuity implies that $\Gamma\cap\mathrm G[p_1,p_2]\ne\varnothing$.
Let $q_2$ stand for the point $\Gamma\cap\mathrm G[p_1,p_2]$.
The triangles $\Delta (q_1,x,q_2)$ and $\Delta (q_1,y,q_2)$ are congruent by
SAS congruence. So, the internal angles at $x$ and $y$ of the triangles
$\Delta (q_2,x,p_1)$ and $\Delta(q_2,y,p_2)$ are equal. By SAA congruence, these
triangles are congruent and
$\dist(p_1,\mathrm G{\wr}x,y{\wr})=\dist(p_1,x)=\dist(p_2,y)
=\dist(p_2,\mathrm G{\wr}x,y{\wr})$. Conversely, if $p_1$ and $p_2$ are in the same 
hypercycle of $\mathrm G{\wr}x,y{\wr}$, we again consider the geodesic
$\Gamma$ orthogonal to $\mathrm G{\wr}x,y{\wr}$ and the points
$q_1,q_2$ which are necessarily on distinct half-planes determined by $\Gamma$.
The triangles $\Delta (q_1,x,q_2)$ and $\Delta (q_1,y,q_2)$ are congruent. Thus,
by SAS congruence, $\Delta (p_1,x,q_2)$ and $\Delta (p_2,y,q_2)$ are congruent
and the corresponding internal angles at $p_1$ and $p_2$ are equal.\hfill$\square$

\section{Relative character varieties and bendings}\label{sec:relative}

In this section, we will see how bendings naturally provide ``coordinates'' in the relative $\SU(2,1)$-character varieties introduced in Definition \ref{defi:charactervariety}.

Let $F\in\SU(2,1)$ be an isometry. Assume that $F:=R_{\alpha_3}^{p_3}R_{\alpha_2}^{p_2}R_{\alpha_1}^{p_1}$ is a decomposition of $F$ into the product of special elliptic isometries.
If $p_2$ is distinct from and nonorthogonal to $p_1$ and $p_3$, we can modify the triple $p_1,p_2,p_3$ by composing bendings involving $p_1,p_2$ and $p_2,p_3$, obtaining a new decomposition of $F$.
In this section, we determine all such decompositions for fixed parameters $\alpha_i$'s, fixed signs $\sigma_i$'s of points, $\sigma p_i:=\sigma_i$, and fixed trace of $F$.

\begin{defi}
\label{defi:stronglyregular}
Let $\pmb\alpha=(\alpha_1,\alpha_2,\alpha_3)$, $\alpha_i\in\mathbb S^1\setminus\Omega$, be a triple of parameters, let $\pmb\sigma=(\sigma_1,\sigma_2,\sigma_3)$, $\sigma_i\in\{-1,1\}$, be a triple of signs such that at most one is positive, $i=1,2,3$, and let $\tau\in\mathbb C$. The triple $p_1,p_2,p_3\in\mathbb PV\setminus\SV$ is {\it strongly regular\/} with respect to $\pmb\alpha,\pmb\sigma,\tau$ if $\sigma p_i=\sigma_i$; $p_1,p_2,p_3$ are pairwise distinct; $p_2$ is not orthogonal to $p_1$ nor to $p_3$; $p_1,p_2,p_3$ are not in a same complex line; and $\trace R_{\alpha_3}^{p_3}R_{\alpha_2}^{p_2}R_{\alpha_1}^{p_1}=\tau$. When $p_1,p_2,p_3$ is strongly regular with respect to $\pmb\alpha,\pmb\sigma,\tau$, we sometimes say that $R_{\alpha_3}^{p_3}R_{\alpha_2}^{p_2}R_{\alpha_1}^{p_1}$ is strongly regular.
\end{defi}

This definition of strong regularity is closely related to the one in \cite{Sasha2012}, where the case $\alpha_i=-1$ is considered.

In the above definition, we require that at most one of the points $p_1,p_2,p_3$ is positive because, otherwise, it could happen that, after some suitable bendings involving $p_1,p_2$ and $p_2,p_3$, we arrive at the situation where one of the lines $\mathrm L(p_1,p_2)$ or $\mathrm L(p_2,p_3)$ is Euclidean. Moreover, if $p_1,p_2,p_3$ are in the same complex geodesic, it could happen that, after finitely many bendings involving $p_1,p_2$ and $p_2,p_3$, we arrive at a situation where $p_1=p_2$ or $p_2=p_3$ (same signature case) or $\langle p_1,p_2\rangle=0$ or $\langle p_2,p_3\rangle=0$ (in the presence of a positive point).

Our objective is to describe geometrically all strongly regular triples with respect to given $\pmb\alpha=(\alpha_1,\alpha_2,\alpha_3)$, $\pmb\sigma=(\sigma_1,\sigma_2,\sigma_3)$, and $\tau$. We define
\begin{equation*}
\kappa:=\frac{\tau-\alpha_1^{-2}\alpha_2\alpha_3-\alpha_1\alpha_2^{-2}\alpha_3-\alpha_1
\alpha_2\alpha_3^{-2}}{(\alpha_1^{-2}-\alpha_1)(\alpha_2^{-2}-\alpha_2)
(\alpha_3^{-2}-\alpha_3)},\quad \chi_i:=\imag\Big(\frac{\alpha_i}{\alpha_i^{-2}-\alpha_i}\Big),
\end{equation*}
and denote
\begin{equation}
\label{eq:chi}
\arraycolsep=1.4pt\def\arraystretch{2.2}
\begin{array}{c}
\displaystyle
t_1:=\tance(p_1,p_2),\quad t_2:=\tance(p_2,p_3),\quad t_3:=\tance(p_1,p_3),\\
\displaystyle\beta:=\frac{\det[g_{ij}]}{g_{11}g_{22}g_{33}},
\quad\eta:=\frac{g_{12}g_{23}g_{31}}{g_{11}g_{22}g_{33}},\quad t:=\real\eta,
\end{array}
\end{equation}
where $[g_{ij}]$ stands for the Gram matrix of $p_1,p_2,p_3$. By Remark~\ref{rmk:partform},
$$\kappa=\frac{\alpha_3}{\alpha_3^{-2}-\alpha_3}t_1
+\frac{\alpha_1}{\alpha_1^{-2}-\alpha_1}t_2
+\frac{\alpha_2}{\alpha_2^{-2}-\alpha_2}t_3
+\eta;$$
the previous equation in $t_1,t_2,t_3,\eta$ is nothing but $\trace R_{\alpha_3}^{p_3}R_{\alpha_2}^{p_2}R_{\alpha_1}^{p_1}=\tau$. Note that $\real\frac{\alpha_i}{\alpha_i^{-2}-\alpha_i}=-\frac12$ because
$$\real\frac\alpha{\alpha^{-2}-\alpha}=\frac12\Big(\frac\alpha{\alpha^{-2}-\alpha}+
\frac{\overline\alpha}{\overline\alpha^{-2}-\overline\alpha}\Big)=\frac12\Big(\frac{\alpha^3}
{1-\alpha^3}+\frac1{\alpha^3-1}\Big)=-\frac12$$
for every $\alpha\in\mathbb S^1\setminus\Omega$ and that
$t_1t_2t_3=\frac{g_{12}g_{21}}{g_{11}g_{22}}\,\frac{g_{23}g_{32}}{g_{22}g_{33}}\,
\frac{g_{31}g_{13}}{g_{11}g_{33}}=|\eta|^2$
by the definition of tance (see Section \ref{sec:hyperbolicgeometry}).

We will shortly see that the space of strongly regular triples in question can be described in terms of a real algebraic equation in the parameters $t_1,t_2,t$ plus a few inequalities related to the signature of the Hermitian form (see Section \ref{sec:hyperbolicgeometry}). Note that $\kappa$ and $\chi_i$, $i=1,2,3$, are prescribed. We are going to express $t_3$ in terms of $t_1,t_2,t$. This will be accomplished by writing $|\eta|^2$ in terms of $t_1,t_2,t$ and by applying the relation $|\eta|^2=t_1t_2t_3$.

Since
\begin{equation}\label{eq:beta}
\beta=1+2t-t_1-t_2-t_3
\end{equation}
we obtain
$$\real\kappa=-\frac{1}{2}(t_1+t_2+t_3)+t
=-\frac{1}{2}(1+2t-\beta)+t=\frac{1}{2}(\beta-1).$$
On the other hand, it follows from \eqref{eq:beta} that
\begin{align*}
\imag\kappa&=\chi_3t_1+\chi_1t_2+\chi_2t_3+\imag\eta=\chi_3t_1+\chi_1t_2+
\chi_2(1-\beta+2t-t_1-t_2)+\imag\eta\\
&=(\chi_3-\chi_2)t_1+(\chi_1-\chi_2)t_2+2\chi_2t+\chi_2(1-\beta)+\imag\eta.
\end{align*}
The above expressions for $\imag\kappa$ and $\real\kappa$ imply that
\begin{equation}\label{eq:imeta}
\imag\eta=\imag\kappa+(\chi_2-\chi_3)t_1+(\chi_2-\chi_1)t_2-2\chi_2t+2\chi_2\real\kappa.
\end{equation}
So, defining
\begin{equation}
\label{eq:constantsci}
\begin{array}{ll}
c_1:=1+4\chi_2^2&
c_6:=(\chi_1-\chi_2)^2\\
c_2:=-4\chi_2(2\chi_2\real\kappa+\imag\kappa)&
c_7:=2\big((\chi_1-\chi_2)(\chi_3-\chi_2)+\real\kappa\big)\\
c_3:=4\chi_2(\chi_3-\chi_2)&
c_8:=2(\chi_2-\chi_3)(2\chi_2\real\kappa+\imag\kappa)\\
c_4:=4\chi_2(\chi_1-\chi_2)&
c_9:=2(\chi_2-\chi_1)(2\chi_2\real\kappa+\imag\kappa)\\
c_5:=(\chi_3-\chi_2)^2&
c_{10}:=(2\chi_2\real\kappa+\imag\kappa)^2
\end{array}
\end{equation}
and using \eqref{eq:imeta}, we arrive at
$$|\eta|^2=c_1t^2+c_2t+c_3t_1t+c_4t_2t+c_5t_1^2+c_6t_2^2+(c_7-2\real\kappa)t_1t_2+c_8t_1+c_9t_2+c_{10}.$$
It follows from $|\eta|^2=t_1t_2t_3$ and from equation \eqref{eq:beta} that
\begin{multline}
\label{eq:sufacesrt}
2t-t_1-t_2-\frac{c_1t^2+c_2t+c_3t_1t+c_4t_2t+c_5t_1^2+c_6t_2^2+c_7t_1t_2+c_8t_1+c_9t_2+c_{10}}{t_1t_2}=0.
\end{multline}

Summarizing, given the strongly regular triple $p_1,p_2,p_3$ with respect to $\pmb\alpha,\pmb\sigma,\tau$, the parameters $t_1,t_2,t$ satisfy the above real algebraic equation whose coefficients are determined solely by the $\alpha_i$'s and by $\tau$. Moreover, by Sylvester's Criterion, the inequalities
\begin{equation}
\label{eq:surfaceineq}
\sigma_1\sigma_2t_1>0,\ \ \sigma_1\sigma_2t_1>\sigma_1\sigma_2,\ \ \sigma_2\sigma_3t_2>0,\ \ \sigma_2\sigma_3t_2>\sigma_2\sigma_3,\ \ \sigma_1\sigma_2\sigma_3(2\real\kappa+1)<0
\end{equation}
must also hold.

Conversely, assume that $(t_1,t_2,t)$ satisfies \eqref{eq:sufacesrt} and \eqref{eq:surfaceineq}. We take $g_{ii}:=\sigma_i$, $g_{12}=g_{21}:=\sqrt{\sigma_1\sigma_2t_1}$, $g_{23}=g_{32}:=\sqrt{\sigma_2\sigma_3t_2}$,
$$g_{13}:=\sigma_1\sigma_2\sigma_3\frac{t-i\left[ \imag\kappa+(\chi_2-\chi_3)t_1+(\chi_2-\chi_1)t_2-2\chi_2t+2\chi_2\real\kappa \right]}{\sqrt{\sigma_1\sigma_2t_1}\sqrt{\sigma_2\sigma_3t_2}},$$
and $g_{31}=\overline{g_{13}}$. We arrive at a Gram matrix $[g_{ij}]$ that satisfies $\frac{\det[g_{ij}]}{g_{11}g_{22}g_{33}}=2\real\kappa+1$. By Sylvester's Criterion and by the explicit construction of the Gram matrix, there exists a geometrically unique strongly regular triple $p_1,p_2,p_3$ with respect to $\pmb\alpha,\pmb\sigma,\tau$ whose Gram matrix equals $[g_{ij}]$. We have just proved the following theorem.

\begin{thm}
\label{thm:surfaceS} Let\/ $\pmb\alpha=(\alpha_1,\alpha_2,\alpha_3)$, $\alpha_i\in\mathbb S^1\setminus\Omega$, be a triple of parameters, let\/ $\pmb\sigma=(\sigma_1,\sigma_2,\sigma_3)$, $\sigma_i\in\{-1,1\}$, be a triple of signs such that at most one is positive, $i=1,2,3$, and let\/ $\tau\in\mathbb C$. Geometrically, all strongly regular triples\/ $p_1,p_2,p_3\in\mathbb PV\setminus\SV$ with respect to\/ $\pmb\alpha,\pmb\sigma,\tau$ are parameterized by the real semialgebraic surface $S_{\pmb\alpha,\pmb\sigma,\tau}\subset\mathbb R^3(t_1,t_2,t)$ given by the equation
\begin{multline}
\label{eq:surfaceS}
t_1^2t_2+t_1t_2^2-2t_1t_2t+c_5t_1^2+c_6t_2^2+c_1t^2+\\+c_7t_1t_2+c_3t_1t+c_4t_2t+c_8t_1+c_9t_2+c_2t+c_{10}=0
\end{multline}
and by the inequalities \eqref{eq:surfaceineq}.
\end{thm}

Take a strongly regular triple $p_1,p_2,p_3$ with respect to $\pmb\alpha,\pmb\sigma,\tau$. Let $S_{\pmb\alpha,\pmb\sigma,\tau}$ be the corresponding surface given in Theorem \ref{thm:surfaceS}. The {\it vertical\/} and {\it horizontal} slices $V_{r_1}$ and $H_{r_2}$ are defined by the equations $t_1=r_1$ and $t_2=r_2$ for $r_1,r_2\in\mathbb R$. More specifically, the vertical slice $V_{r_1}$ is the conic given by the equation
\begin{equation}\label{eq:conicv}
d_1 t_2^2+d_2 t_2 t+d_3 t^2+d_4 t_2+d_5 t+d_6=0,
\end{equation}
where
\begin{equation*}
\begin{array}{ll}
d_1:=r_1+c_6&
d_4:=r_1^2+c_7r_1+c_9\\
d_2:=-2r_1+c_4&
d_5:=c_3r_1+c_2\\
d_3:=c_1&
d_6:=c_5 r_1^2+c_8r_1+c_{10}\\
\end{array}
\end{equation*}
and the horizontal slice $H_{r_2}$ is the conic given by
\begin{equation}\label{eq:conich}
e_1 t_1^2+e_2 t_1 t+e_3 t^2+e_4 t_1+e_5 t+e_6=0,
\end{equation}
where
\begin{equation}
\begin{array}{ll}
e_1:=r_2+c_5&
e_4:=r_2^2+c_7r_2+c_8\\
e_2:=-2r_2+c_3&
e_5:=c_4r_2+c_2\\
e_3:=c_1&
e_6:=c_6 r_2^2+c_9r_2+c_{10}\\
\end{array}
\end{equation}
(note that, by \eqref{eq:constantsci}, the coefficients $d_3$ and $e_3$ of the quadratic term $t^2$ do not vanish). Vertical and horizontal slices can be empty.
Define
\begin{equation}\label{eq:vemca}
k_1:=\frac{c_1+c_4+\sigma_1\sigma_2\sqrt{c_1(c_1+2c_4+4c_6)}}{2}=\frac{1+4\chi_1\chi_2+
\sigma_1\sigma_2\sqrt{(1+4\chi_1^2)(1+4\chi_2^2)}}2
\end{equation}
and
$$k_2:=\frac{c_1+c_3+\sigma_2\sigma_3\sqrt{c_1(c_1+2c_3+4c_5)}}{2}=\frac{1+4\chi_2\chi_3+
\sigma_2\sigma_3\sqrt{(1+4\chi_2^2)(1+4\chi_3^2)}}2.$$

\begin{rmk}\label{rmk:discriminant}
Writing the discriminants of the conics \eqref{eq:conicv} and \eqref{eq:conich} and observing that $k_1>0$ if $\sigma_1\sigma_2>0$ and $k_1<0$ if $\sigma_1\sigma_2<0$, one can see that $V_{r_1}$ is contained (perhaps, vacuously) in: an ellipse or a single point if $\sigma_1\sigma_2 r_1< \sigma_1\sigma_2 k_1$; a parabola or a pair of parallel lines (not necessarily distinct) if $r_1=k_1$; and a hyperbola or a pair of concurrent lines if $\sigma_1\sigma_2 r_1> \sigma_1\sigma_2 k_1$. Taking $k_2,r_2,\sigma_3$ in place of $k_1,r_1,\sigma_1$ we obtain the analogous facts concerning a horizontal slice $H_{r_2}$. (The exact description of vertical and horizontal slices is presented in the next theorem.)
\end{rmk}

Let $Q\subset\mathbb R(t_1,t_2)$ be the region defined by the inequalities \eqref{eq:surfaceineq}. Considering equation \eqref{eq:surfaceS} as a quadratic equation in $t$ (by (\ref{eq:constantsci}), $c_1\ne0$), we can see that, whenever $Q$ is nonempty, the projection $R\subset\mathbb R(t_1,t_2)$ of the surface $S_{\pmb\alpha,\pmb\sigma,\tau}$ into $Q$ is given by the inequality
$$(-2t_1t_2+c_3t_1+c_4t_2+c_2)^2
-4c_1(t_1^2t_2+t_1t_2^2+c_5t_1^2+c_6t_2^2+c_7t_1t_2+c_8t_1+c_9t_2+c_{10})\geq 0.$$
In this case, let $C$ be the curve in $S_{\pmb\alpha,\pmb\sigma,\tau}$ defined by
\begin{equation}\label{eq:curve}
t=\frac{2t_1t_2-c_3t_1-c_4t_2-c_2}{2c_1}.
\end{equation}
The projection $S_{\pmb\alpha,\pmb\sigma,\tau}\to R$ is a double covering ramified along $C$.

\begin{thm}
\label{thm:sbendingconnect} Let\/ $\pmb\alpha=(\alpha_1,\alpha_2,\alpha_3)$, $\alpha_i\in\mathbb S^1\setminus\Omega$, be a triple of parameters, let\/ $\pmb\sigma=(\sigma_1,\sigma_2,\sigma_3)$, $\sigma_i\in\{-1,1\}$, be a triple of signs such that at most one is positive, $i=1,2,3$, and let\/ $\tau\in\mathbb C$. Consider the semialgebraic surface\/ $S_{\pmb\alpha,\pmb\sigma,\tau}$ parameterizing, modulo conjugation, all strongly regular triples\/ $p_1,p_2,p_3$ with respect to $\pmb\alpha,\pmb\sigma,\tau$ {\rm(}see {\rm Theorem \ref{thm:surfaceS}}{\rm).} The following holds\/{\rm:}

\smallskip

{\rm (}i\/{\rm )} A nonempty vertical slice can intersect a nonempty horizontal slice in at most two points. When they intersect in a single point, this point belongs to\/ $C$.

{\rm (}ii\/{\rm )} Nonempty vertical and horizontal slices of\/ $S_{\pmb\alpha,\pmb\sigma,\tau}$ correspond, respectively, to bendings involving\/ $p_1,p_2$ and\/ $p_2,p_3$.

{\rm (}iii\/{\rm )}  When\/ $p_3$ {\rm(}respectively, $p_1${\rm)} is not orthogonal to any fixed point of\/ $R:=R_{\alpha_2}^{p_2}R_{\alpha_1}^{p_1}$ {\rm(}respectively, of\/ $R:=R_{\alpha_3}^{p_3}R_{\alpha_2}^{p_2}${\rm),} a nonempty vertical {\rm(}respectively, horizontal\/{\rm)} slice is an ellipse if\/ $R$ is regular elliptic, a branch of a hyperbola when\/ $R$ is loxodromic, and a parabola when\/ $R$ is ellipto-parabolic. It intersects\/ $C$ in exactly two points in the first case and in exactly one point in the other cases.

{\rm (}iv\/{\rm )}  When\/ $p_3$ {\rm(}respectively, $p_1${\rm)} is orthogonal to a fixed point of\/ $R$, a nonempty vertical {\rm(}respectively, horizontal\/{\rm)} slice is a single point belonging to\/ $C$ when\/ $R$ is elliptic or ellipto-parabolic and a pair of open rays not lying in a same straight line and sharing a common point in their topological closures when\/ $R$ is loxodromic {\rm(}this common point does not belong to\/ $S_{\pmb\alpha,\pmb\sigma,\tau}$ and the pair of open rays does not intersect\/ $C${\rm).} The surface\/ $S_{\pmb\alpha,\pmb\sigma,\tau}$ contains at most one vertical {\rm(}horizontal\/{\rm)} slice of the last type.
\end{thm}

\begin{proof} We remind the reader that $t_1:=\tance(p_1,p_2)$, $t_2:=\tance(p_2,p_3)$, and $t:=\real\eta$, where $\eta$ is defined in \eqref{eq:chi}. A vertical slice and a horizontal slice have at most two intersection points because equation \eqref{eq:surfaceS} is quadratic in $t$ (by \eqref{eq:constantsci}, $c_1\ne0$) for fixed $t_1,t_2$. By the definition of $C$, if a vertical slice $V_{r_1}$ and a horizontal slice $H_{r_2}$ intersect at a point $(r_1,r_2,t)$ that does not belong to $C$, there exists $t'\ne t$ such that $(r_1,r_2,t),(r_1,r_2,t')\in V_{r_1}\cap H_{r_2}$. This concludes the proof of ({\it i\/}).

Bendings involving $p_1,p_2$ (respectively, $p_2,p_3$) preserve $t_1$ (respectively, $t_2$) as well as the fact that $p_1,p_2,p_3$ is strongly regular with respect to $\pmb\alpha,\pmb\sigma,\tau$. Moreover, they keep the product $R_{\alpha_3}^{p_3}R_{\alpha_2}^{p_2}R_{\alpha_1}^{p_1}$; so, such bendings provide curves inside vertical (respectively, horizontal) slices. More specifically, let $V_r$ be a vertical slice and let $p_1,p_2,p_3$ be a strongly regular triple in $V_r$. Bendings involving $p_1,p_2$ constitute the subset of $V_r$ that corresponds to the strongly regular triples of the form $B(s)p_1,B(s)p_2,p_3$, where $B$ is the one-parameter group introduced in Proposition~\ref{prop:bendings}.

Consider the functions $t_2,t:\mathbb R\to \mathbb R$, $t_2=t_2(s)$, $t=t(s)$, such that the point $\big(r,t_2(s),t(s)\big)\in V_r$ corresponds to the triple $B(s)p_1,B(s)p_2,p_3$.
Let $\widetilde B:\mathbb R\to V_r$ stand for the function $\widetilde B(s):=\big(r,t_2(s),t(s)\big)$. We will prove that the image of $\widetilde B$ is the entire $V_r$. We begin by finding explicit expressions for $t_2(s)$ and $t(s)$. By Lemma \ref{lemma:prodspecreg}, $R:=R_{\alpha_2}^{p_2}R_{\alpha_1}^{p_1}$ is regular; so, there are three cases to consider ($R$ is regular elliptic, loxodromic, or ellipto-parabolic) and the strategy will be to prove the corresponding parts of ({\it ii\/}), ({\it iii}), and ({\it iv\/}) while considering each of these cases.

Case 1: $R:=R_{\alpha_2}^{p_2}R_{\alpha_1}^{p_1}$ is regular elliptic. Let $p,q\in\mathrm{L}(p_1,p_2)$ be respectively a negative and a positive fixed point of $R$ (these points exist because $p_1,p_2,p_3$ is strongly regular). Note that neither $p_1$ nor $p_2$ is fixed by $R$ (otherwise, $\langle p_1,p_2\rangle=0$ or $p_1=p_2$), so $p_1\ne p,q$ and $p_2\ne p,q$. Take representatives such that $\langle p,p\rangle=-1$, $\left<q,q\right>=1$, $\langle p_3,p_3\rangle=\sigma_3$, $p_1=p+\lambda_1 q$, $p_2=p+\lambda_2q$, $|\lambda_1|,|\lambda_2|\ne0,1$.

As in Proposition~\ref{prop:bendings}, $B(s)p=e^{-\frac{s}{2}i}p$ and $B(s)q=e^{\frac{s}{2}i}q$. We write $z_1:=\left<p,p_3\right>$, $z_2:=\left<q,p_3\right>$, $w_1:=\lambda_1\overline z_1 z_2$, and $w_2:=\lambda_2\overline z_1 z_2$. Let us show that $\widetilde B$ is an immersion when $V_r$ is not a single point. Let $\simeq$ stand for ``(nonnull) proportionality modulo a constant factor''. By a straightforward calculation
$$t_2(s)=\tance(B(s)p_2,p_3)\simeq\real(w_2e^{si})$$
and
$$t(s)=\real\frac{\left<B(s)p_1,B(s)p_2\right>\left<B(s)p_2,p_3
\right>\left<p_3,B(s)p_1\right>}{\left<B(s)p_1,B(s)p_1\right>\left<B(s)p_2,B(s)p_2\right>
\left<p_3,p_3\right>}\simeq\real\Big(\big(w_1(|\lambda_2|^2-1)+w_2(|\lambda_1|^2-1)
\big)e^{si}\Big).$$
Clearly, $t_2'(s)=0$ iff $w_2e^{si}\in\mathbb R$ and $t'(s)=0$ iff $\big(w_1(|\lambda_2|^2-1)+w_2(|\lambda_1|^2-1)\big)e^{si}\in\mathbb R$.

When $w_2=0$, we have $w_1=0$ because $\lambda_1,\lambda_2\ne0$. This case corresponds to $p_3$ being orthogonal to a fixed point of $R$; both $t_2(s)$ and $t(s)$ are constant and $V_r$ is, if nonempty, a single point. In particular, this point must lie on $C$.

Assume $w_2\ne0$. We want to prove that $t_2'(s)=t'(s)=0$ never happens. Assuming the contrary, one obtains
$$\frac{\big(w_1(|\lambda_2|^2-1)+w_2(|\lambda_1|^2-1)\big)e^{si}}{w_2e^{si}}\in\mathbb R$$
which implies $w_1/w_2=\lambda_1/\lambda_2\in\mathbb R$. Hence, both $p_1,p_2$ belong to the geodesic $\mathbb Rp+\mathbb R\lambda_1 q$ joining $p$ and $q$. This is impossible because it would lead to $G=G_1=G_2$, where $G,G_1,G_2$ are the geodesics associated to $R_{\alpha_2}^{p_2}R_{\alpha_1}^{p_1}$ (see Definition \ref{defi:geoassociated}). In other words, when $V_r$ is not a point, $\widetilde B$ is an immersion. Together with the fact that $\widetilde B$ is a periodic function, this implies that $V_r$ is an ellipse and that $\widetilde B(\mathbb R)=V_r$.

Looking at the intersections of this ellipse with the straight lines $t_2=\mathrm{constant}$ in the plane $(r,t_2,t)$, it is clear that exactly two distinct such lines will be tangent to the ellipse; these lines give rise to the intersection $V_r\cap C$.

Case 2: $R:=R_{\alpha_2}^{p_2}R_{\alpha_1}^{p_1}$ is loxodromic. Let $v_1,v_2$ be the isotropic fixed points of $R$. We choose representatives such that $\langle v_1,v_2\rangle=-1/2$, $p_1=v_1+\lambda_1 v_2$, $p_2=v_1+\lambda_2 v_2$, and $\langle p_3,p_3\rangle=\sigma_3$. We write $z_1:=\left<v_1,p_3\right>$, $z_2:=\left<v_2,p_3\right>$, $w_1:=\lambda_1\overline z_1 z_2$, and $w_2:=\lambda_2\overline z_1 z_2$. Note that $z_1,z_2$ cannot vanish simultaneously since this would imply that $p_3$ is the polar point of the line $\mathrm{L}(p_1,p_2)$, contradicting the strong regularity of the triple. Moreover, $\langle p_i,p_i\rangle=-\real\lambda_i\ne0$.

As in Proposition \ref{prop:bendings}, $B(s)v_1=e^{s}v_1$ and $B(s)v_2=e^{-s}v_2$.
Then,
$$t_2(s)=-\frac{|z_1|^2e^{2s}+2\real{w_1}+|\lambda_2|^2|z_2|^2e^{-2s}}
{\sigma_3\real{\lambda_2}}$$
and
$$t(s)=-\frac12\frac{|z_1|^2\real(\lambda_1+\lambda_2)e^{2s}+|z_2|^2(
|\lambda_1|^2\real\lambda_2+|\lambda_2|^2\real\lambda_1)e^{-2s}+2\real\big((\overline
\lambda_2+\lambda_1)(\overline w_1+w_2)\big)}{\sigma_3\real\lambda_1\real\lambda_2}.$$
As before, let us show that $\widetilde B$ is an immersion. We can assume $z_1,z_2\ne0$ since, otherwise, $t_2'(s)$ never vanishes. This implies that $t_2'(s)=0$ iff
$s=s_0:=\frac{1}{4}\ln\big(\frac{|\lambda_2|^2|z_2|^2}{|z_1|^2}\big)$.
Requiring $t'(s_0)=0$, and using the fact that $\real\lambda_2\ne0$ (since $\langle p_2,p_2\rangle=-\real\lambda_2$), we arrive at $|\lambda_1|=|\lambda_2|$. But, as we shall see, this implies that the geodesic through $p_1,p_2$ is orthogonal to the geodesic through $v_1,v_2$ (which is impossible because it leads to $G=G_1=G_2$ where $G,G_1,G_2$ are the geodesics associated to $R_{\alpha_2}^{p_2}R_{\alpha_1}^{p_1}$).

Assume $|\lambda_1|=|\lambda_2|$ and consider the curves $\gamma,\psi:\mathbb R\to\mathbb PV$ with lifts $\gamma_0,\psi_0:\mathbb R\to V$ given by $\gamma_0(t)=e^tv_1+e^{-t}v_2$ and $\psi_0(s)=e^{si}v_1+e^{-si}\rho v_2$, where $\rho=|\lambda_1|=|\lambda_2|$. Then $\gamma$ parameterizes the negative part of $\mathrm{G}{\wr}v_1,v_2{\wr}$ and $\psi$ parameterizes $\mathrm{G}{\wr}p_1,p_2{\wr}$ (in order to verify the last claim, one can use the equation of a geodesic in \cite[Subsection 3.4]{SashaGrossi2011}). These curves intersect at $q=\gamma(a)=\psi(0)$, where $e^{-2a}=\rho$. By \cite[Lemma 4.1.4]{SGG2011} we have
$\dot\gamma(a)=\left<-,q\right>(e^{-a}v_2-e^av_1)$ and $\dot\psi(0)=\left<-,q\right>i\big(v_2-\frac{1}{\rho}v_1\big)$
(see the identification in \eqref{eq:tangentspace}).
It follows that
$$\real\langle\dot\gamma(a),\dot\sigma(0)\rangle=\real\bigg(-\langle q,q\rangle\Big\langle e^{-a}v_2-e^av_1,
i\big(v_2-\frac{1}{\rho}v_1\big)\Big\rangle\bigg)=0.$$
We conclude that $\widetilde B$ is an immersion. Therefore, when $z_1,z_2\ne0$, $V_r=\widetilde B(\mathbb R)$ is a branch of a hyperbola when nonempty. Indeed, the terms in $e^{2s}$ and $e^{-2s}$ in $t_2(s)$ are both nonnull. We have $\lim_{s\to\pm\infty}t_2(s)=\infty$ or $\lim_{s\to\pm\infty}t_2(s)=-\infty$; together with the fact that $\widetilde B$ is an immersion, this implies that $\big(t_2(s),t(s)\big)$ cannot parameterize a straight line. So, $V_r$ contains at least one branch of a hyperbola. It cannot contain the other due to the fact that the coefficients of $e^{2s}$ and $e^{-2s}$ in $t_2(s)$ have the same sign: when $|r_2|$ grows, the straight line $t_2=r_2$ in the plane $(r,t_2,t)$ has to intersect $V_r$ in two distinct points. In particular, there is a single value of $r_2$ such that the straight line in question is tangent to the branch of hyperbola; this is exactly the intersection $V_r\cap C$.

Let $z_1=0$ (the reasoning is the same for $z_2=0$) and assume that $V_r$ is nonempty. Take $x\in V_r$. The component of $V_r$ containing $x$ is clearly an open ray $\gamma$ approaching the point $(t_2,t)=(0,0)$. The slice $V_r$ contains at most a second component which must be another open ray approaching the point $(0,0)$ lying in a straight line distinct from the one containing $\gamma$. Indeed, it is easy to see that, if a second such ray exists, it cannot be in the same straight line as $\gamma$ because this would make $t_2>0$ (impossible: $\sigma_3=1$ since $p_3$ is orthogonal to an isotropic point and the triple is strongly regular). Now, the conic \eqref{eq:conicv} has to be a pair of concurrent lines intersecting at the point $(0,0)$ in the plane $t_1=r$ and so there cannot exist further components because this would lead to distinct open rays contained in a same straight line.

We will soon show that $V_r$ necessarily contains a second component but, first, we need to prove that $S_{\pmb\alpha,\pmb\sigma,\tau}$ has at most one vertical slice which is a pair of open rays as above. Assume that $V_{r_1}$ is such a vertical slice for some $r_1=\tance(p_1,p_2)\ne0$. In this case, the conic \eqref{eq:conicv} is a pair of concurrent lines intersecting at the point $(0,0)$ in the plane $t_1=r_1$ and we have $d_4=d_5=d_6=0$ in \eqref{eq:conicv}. We can assume that $c_5=c_8=c_{10}=0$ because, otherwise, there exists at most one possible value of $r_1$ such that $d_6=0$ (note that, by \eqref{eq:constantsci}, $c_8^2-4c_5c_{10}=0$). This implies $c_9=0$ and then it follows from $d_4=0$ that $r_1=-c_7$.

Finally, in order to establish that $V_r$ contains a second component, all we need is to prove that $\gamma\setminus C\ne\varnothing$ since, in this case, there exists $t'\ne t$ with $(r,t_2,t),(r,t_2,t')\in V_r$ and $(r,t_2,t)\in\gamma\setminus C$. Then $(r,t_2,t')\notin\gamma$ due to the fact that $t_2$ does not vanish. So, we obtain a point in $V_r\setminus\gamma$ and, consequently, the mentioned second component of $V_r$.

Note that, once $t_1=r$ is fixed, equation \eqref{eq:curve} allows to express $t$ as a linear function of $t_2$ and this turns equation \eqref{eq:surfaceS} into an equation of degree at most $2$ in $t_2$. Therefore, either $\gamma\cap C$ has at most $2$ points and we are done or we can assume $C\subset\gamma$. In this case, bending $R_{\alpha_3}^{p_3}R_{\alpha_2}^{p_2}$ and keeping $R$ loxodromic, we obtain a nonempty slice $V_{r+\varepsilon}$ for a sufficiently small $\varepsilon$. If this slice is a branch of a hyperbola, then we get an extra point in $C$ as above. Otherwise, it contains an open ray passing through $(r+\varepsilon,0,0)$. By the definition of $C$, a point $(r+\varepsilon,r',t')\in V_{r+\varepsilon}$ belongs to $C$ exactly when the line $t_2=r'$ in the plane $t_1=r+\varepsilon$ intersects $V_{r+\varepsilon}$ only in $(r+\varepsilon,r',t')$. Hence, $V_{r+\varepsilon}\cap C=\varnothing$ implies the existence of a solution of (\ref{eq:surfaceS}) with $t_2=0$ and this contradicts strong regularity.

Case 3: $R:=R_{\alpha_2}^{p_2}R_{\alpha_1}^{p_1}$ is ellipto-parabolic. Let $v_1$ be the isotropic fixed of $R$ and let $v_2\in\mathrm L(p_1,p_2)$, $v_2\ne v_1$, be another isotropic point. We choose representatives such that $\left<v_1,v_2\right>=\frac{1}{2}$, $p_1=v_1+\lambda_1 v_2$, $p_2=v_1+\lambda_2 v_2$, and $\langle p_3,p_3\rangle=\sigma_3$. Note that $\langle p_i,p_i\rangle=\real\lambda_i\ne0$. We write $z_1:=\langle v_1,p_3\rangle$,  $z_2:=\langle v_2,p_3\rangle$. As in Proposition \ref{prop:bendings}, $B(s)v_1=v_1$ and $B(s)v_2=isv_1+v_2$.
Therefore,
$$t_2(s)\simeq|\lambda_2|^2|z_1|^2s^2-2\big(|z_1|^2\imag\lambda_2+|\lambda_2|^2\imag(z_1\overline z_2)\big)s$$
and
$$t(s)\simeq\big(|\lambda_1|^2\real\lambda_2+|\lambda_2|^2\real\lambda_1\big)|z_1|^2s^2
-2\big(|z_1|^2\imag(\lambda_1\lambda_2)+\imag(z_1\overline z_2)(|\lambda_1|^2\real\lambda_2+|\lambda_2|^2\real\lambda_1)\big)s.$$

If $z_1=0$ (that is, $p_3$ is orthogonal to the isotropic fixed point of $R$), $V_r$ is, if nonempty, a single point which must lie on $C$.

Assume $z_1\ne0$. Let us show that, in this case, $\widetilde B$ is an immersion. Indeed, there is a single value
$$s=s_0:=\frac{|z_1|^2\imag\lambda_2+|\lambda_2|^2\imag(z_1\overline z_2)}{|\lambda_2|^2|z_1|^2}$$
such that $t_2'(s)=0$. By a straightforward calculation, $t'(s_0)=0$ implies
$|\lambda_1|^2\imag\lambda_2=|\lambda_2|^2\imag\lambda_1$. But the latter equation implies that $p_1,p_2,v_1$ are in a same geodesic (this can be directly verified using the equation of a geodesic in
\cite[Subsection 3.4]{SashaGrossi2011}) which is impossible because it implies $G_1=G_2=G$, where $G,G_1,G_2$ stand for the geodesics related to $R$.

Since $z_1\ne0$, the coefficient in $s^2$ of $t_2(s)$ is nonnull. We have $\lim_{s\to\pm\infty}t_2(s)=\infty$ or $\lim_{s\to\pm\infty}t_2(s)=-\infty$; considering that $\widetilde B$ is an immersion, this implies that $\big(t_2(s),t(s)\big)$ cannot parameterize a straight line. We obtain that $V_r=\widetilde B(\mathbb R)$ is a parabola. Reasoning as in the previous case, one readily sees that the intersection $V_r\cap C$ is a single point.
\end{proof}

\begin{lemma}\label{lemma:p1ortp3ort}
Let\/ $p_1,p_2,p_3$ be a strongly regular triple with respect to $\pmb\alpha,\pmb\sigma,\tau$. The isometry\/ $R:=R_{\alpha_3}^{p_3}R_{\alpha_2}^{p_2}R_{\alpha_1}^{p_1}$ is regular iff\/ $p_1$ is not orthogonal to any fixed point of\/ $R_2:=R_{\alpha_3}^{p_3}R_{\alpha_2}^{p_2}$ or\/ $p_3$ is not orthogonal to any fixed point of\/ $R_1:=R_{\alpha_2}^{p_2}R_{\alpha_1}^{p_1}$.
\end{lemma}

\begin{proof}
Assume that $R$ is nonregular. We write $R_{\alpha_3}^{p_3}R_{\alpha_2}^{p_2}R_{\alpha_1}^{p_1}=U$, where $U$ is either $2$-step unipotent or special elliptic. Let $v$ stand for the polar point of the pointwise fixed complex line of $U$. If $v=p_3$, then $U$ has to be special elliptic and we obtain either a relation of length $3$ of the form $R_{\beta}^{p_3}R_{\alpha_2}^{p_2}R_{\alpha_1}^{p_1}=1$ for some parameter $\beta\notin\Omega$ or a relation of length 2 of the form $R_{\alpha_2}^{p_2}R_{\alpha_1}^{p_1}=\delta$, $\delta\in\Omega$. So, it follows from the classification of length $2$ and $3$ relations (see Section \ref{sec:basic}) that the assumption $v=p_3$ contradicts strong regularity.

Assume $v\ne p_3$ and consider the relation
$R_{\alpha_2}^{p_2}R_{\alpha_1}^{p_1}=R_{\overline\alpha_3}^{p_3}U$.
By Lemma \ref{lemma:prodspecreg}, the left side (hence, the right side) of the above equation is regular. So, we obtain a pair of regular isometries that stabilize two complex lines: the noneuclidean line $L_1:=\mathrm{L}(p_1,p_2)$ and the line $L_2:=\mathrm{L}(p_3,v)$, where $v$ stands for the polar point of the pointwise fixed complex line of $U$. This implies that $L_1$ and $L_2$ are equal or orthogonal. In the first case, $p_1,p_2,p_3$ belong to a same complex line (this contradicts strong regularity); in the second case, the polar point of $L_2$, which is a fixed point of $R_1$, is orthogonal to $p_3$. The same reasoning implies that $p_1$ is orthogonal to a fixed point of $R_2$.

Conversely, let $a,b$ stand respectively for fixed points of $R_1$ and $R_2$ such that $\langle p_1,b\rangle=\langle p_3,a\rangle=0$. If $a=b$, then $a$ is a fixed point of each $R_{\alpha_i}^{p_i}$ which implies that $a$ is the polar point of $L_1$ because $p_1\ne p_2$ and $\langle p_1,p_2\rangle\ne0$. But then $p_3\in L_1$ and this is impossible. Hence, $a\ne b$. Moreover, $a$ and $b$ cannot be both isotropic since this implies that $p_1$ and $p_3$ are positive. So, assuming that $R$ is regular, we obtain $\langle a,b\rangle=0$. Take the line $L:=\mathrm{L}(p_1,a)$ whose polar point is $b$. Since $a$ is not the polar point of $L_1$, we have $L=L_1$. So, $\langle p_2,b\rangle=0$. It follows that $b$ is a fixed point of $R_{\alpha_1}^{p_1}$, of $R_{\alpha_2}^{p_2}$, and of $R_2$. Then it is a fixed point of $R_{\alpha_3}^{p_3}$. As above, this leads to a contradiction.
\end{proof}

\begin{defi}\label{defi:degenerate}
Let $S_{\pmb\alpha,\pmb\sigma,\tau}$ be the surface parameterizing, modulo conjugation, all strongly regular triples $p_1,p_2,p_3$ with respect to $\pmb\alpha,\pmb\sigma,\tau$. A vertical/horizontal slice of $S_{\pmb\alpha,\pmb\sigma,\tau}$ is {\it degenerate\/} when it is of the form described in item ({\it iv\/}) of Theorem \ref{thm:sbendingconnect}. A point $(t_1,t_2,t)\in S_{\pmb\alpha,\pmb\sigma,\tau}$ is {\it degenerate\/} when both the vertical slice $V_{t_1}$ and the horizontal slice $H_{t_2}$ through it are degenerate.
\end{defi}

\begin{rmk}\label{rmk:nodegpoints}
There are no degenerate points in $S_{\pmb\alpha,\pmb\sigma,\tau}$ when $f(\tau)\ne0$, where $f$ is the function defined in \eqref{eq:goldmandeltoid}. Indeed, assuming the contrary, take a degenerate point in such a surface $S_{\pmb\alpha,\pmb\sigma,\tau}$. By Theorem \ref{thm:sbendingconnect}, this point provides a strongly regular triple $p_1,p_2,p_3$ such that $p_1$ is orthogonal to a fixed point of $R_{\alpha_3}^{p_3}R_{\alpha_2}^{p_2}$ and $p_3$ is orthogonal to a fixed point of $R_{\alpha_2}^{p_2}R_{\alpha_1}^{p_1}$. By Lemma \ref{lemma:p1ortp3ort}, the isometry $F:=R_{\alpha_3}^{p_3}R_{\alpha_2}^{p_2}R_{\alpha_1}^{p_1}$ is not regular and, therefore, $f(\trace F)=f(\tau)=0$, a contradiction.
\end{rmk}

\begin{defi}\label{defi:charactervariety}
Let $\pmb\alpha=(\alpha_1,\alpha_2,\alpha_3)$, $\alpha_i\in\mathbb S^1\setminus\Omega$, be a triple of parameters, let $\pmb\sigma=(\sigma_1,\sigma_2,\sigma_3)$, $\sigma_i\in\{-1,1\}$, be a triple of signs such that at most one is positive, $i=1,2,3$, and let $[F]$ be the conjugacy class of a regular isometry $F\in\SU(2,1)$. We denote by $\mathcal V_{\pmb\alpha,\pmb\sigma,[F]}$ the relative $\SU(2,1)$-character variety consisting of representations $\rho:\pi_1(\Sigma)\to\SU(2,1)$, modulo conjugation, of the rank $3$ free group $\pi_1(\Sigma):=\langle \iota_1,\iota_2,\iota_3,\iota_4\mid\iota_4\iota_3\iota_2\iota_1=1\rangle$, where $\pi_1(\Sigma)$ stands for the fundamental group of the quadruply punctured sphere $\Sigma$. The conjugacy classes of $\rho(\iota_i)$ correspond to those of the special elliptic isometries $R_{\alpha_i}^{p_i}$, where $\sigma p_i=\sigma_i$, $i=1,2,3$, and the conjugacy class of $\rho(\iota_4)$ is~$[F]$.
\end{defi}

Since isometries with the same trace may belong to distinct conjugacy classes (see Subsection~\ref{subsec:isometries}), the above character varieties are not exactly the same objects as the surfaces constructed in Theorem~\ref{thm:surfaceS}:

\begin{thm}\label{thm:charactervariety}
Let\/ $S_{\pmb\alpha,\pmb\sigma,\tau}$ be the surface parameterizing, modulo conjugation, all strongly regular triples\/ $p_1,p_2,p_3$ with respect to $\pmb\alpha,\pmb\sigma,\tau$.

\smallskip

$\bullet$ When\/ $f(\tau)>0$, $\mathcal V_{\pmb\alpha,\pmb\sigma,[F]}=S_{\pmb\alpha,\pmb\sigma,\tau}$, where $F\in\SU(2,1)$ is a loxodromic isometry such that\/ $\trace(F^{-1})=\tau$.

\smallskip

$\bullet$ When\/ $f(\tau)=0$, $\mathcal V_{\pmb\alpha,\pmb\sigma,[F]}=S_{\pmb\alpha,\pmb\sigma,\tau}\setminus S'$, where $F\in\SU(2,1)$ is an ellipto-parabolic isometry such that\/ $\trace(F^{-1})=\tau$. Here, $S'$ stands for the set of degenerate points in $S_{\pmb\alpha,\pmb\sigma,\tau}$.

\smallskip

$\bullet$ When\/ $f(\tau)<0$, $\mathcal V_{\pmb\alpha,\pmb\sigma,[F]}\subset S_{\pmb\alpha,\pmb\sigma,\tau}$, where $F\in\SU(2,1)$ is a regular elliptic isometry such that\/ $\trace(F^{-1})=\tau$.
\end{thm}

\begin{proof}
The proof that $\mathcal V_{\pmb\alpha,\pmb\sigma,[F]}\subset S_{\pmb\alpha,\pmb\sigma,\tau}\setminus S'$ is the same in all cases: a relation $FR_{\alpha_3}^{p_3}R_{\alpha_2}^{p_2}R_{\alpha_1}^{p_1}=1$ provides a point in the surface $S_{\pmb\alpha,\pmb\sigma,\tau}$ which is nondegenerate by Lemma \ref{lemma:p1ortp3ort}. Conjugating the relation does not change the obtained point in $S_{\pmb\alpha,\pmb\sigma,\tau}$. It remains to observe that, by Remark \ref{rmk:nodegpoints}, $S'=\varnothing$ in the first and third cases. Conversely, take a nondegenerate point in $S_{\pmb\alpha,\pmb\sigma,\tau}$ and consider the corresponding isometry $F^{-1}:=R_{\alpha_3}^{p_3}R_{\alpha_2}^{p_2}R_{\alpha_1}^{p_1}$. By Lemma \ref{lemma:p1ortp3ort}, $F$ is regular and, therefore, its conjugacy class $[F]$  is determined by $\tau$ when $f(\tau)\geqslant0$.
\end{proof}

It will be important to have a criterion allowing to determine, under certain circumstances, whether or not two points in $S_{\pmb\alpha,\pmb\sigma,\tau}\setminus S'$ lie in a common connected component (see the above theorem). In order to obtain such criterion, we need the following corollary of Theorem \ref{thm:sbendingconnect}.

\begin{cor}
\label{cor:tancetype}
Let\/ $\alpha_1,\alpha_2\in\mathbb S^1\setminus\Omega$ be parameters and let\/ $p_1,p_2\in\PCV\setminus\SV$ be distinct non\-orthogonal points such that the complex line\/ $L:=\mathrm{L}(p_1,p_2)$ is noneuclidean. Then\/ $R:=R_{\alpha_2}^{p_2}R_{\alpha_1}^{p_1}$~is

\smallskip

$\bullet$ regular elliptic iff either\/ $0<\tance(p_1,p_2)<1$ or\/ {\rm(}$L$ is hyperbolic and\/ $\sigma_1\sigma_2\tance(p_1,p_2)<\sigma_1\sigma_2 k_1${\rm)};

$\bullet$ ellipto-parabolic iff\/ $L$ is hyperbolic and\/ $\tance(p_1,p_2)=k_1$;

$\bullet$ loxodromic iff\/ $L$ is hyperbolic and\/ $\sigma_1\sigma_2\tance(p_1,p_2)>\sigma_1\sigma_2k_1$,

\smallskip

\noindent
where\/ $k_1$ is the constant defined in\/ \eqref{eq:vemca} {\rm(}see also\/ \eqref{eq:chi} for the definition of the terms $\chi_i${\rm)} and\/ $\sigma_i:=\sigma p_i$, $i=1,2$.
\end{cor}

\begin{proof}
When $0<\tance(p_1,p_2)<1$, the complex line $L$ is spherical and $R$ is regular elliptic because, by Lemma \ref{lemma:prodspecreg}, $R$ is regular.

Assume that at most one of $p_1,p_2$ is positive. Take a negative point $p_3$ and a parameter $\alpha_3\in\mathbb S^1\setminus\Omega$ such that $p_1,p_2,p_3$ is strongly regular and $p_3$ is not orthogonal to a fixed point of~$R$. Consider the surface $S_{\pmb\alpha,\pmb\sigma,\tau}$ corresponding to $\pmb\alpha=(\alpha_1,\alpha_2,\alpha_3)$, $\pmb\sigma=(\sigma_1,\sigma_2,-1)$, and $\tau:=\trace R_{\alpha_3}^{p_3}R_{\alpha_2}^{p_2}R_{\alpha_1}^{p_1}$. The vertical slice corresponding to the triple $p_1,p_2,p_3$ is determined by item ({\it iii}\/) in Theorem \ref{thm:sbendingconnect}. The result now follows from Remark \ref{rmk:discriminant}.

It remains to consider the case where $p_1,p_2$ are both positive points and the complex line $L$ is hyperbolic. We have $\tance(p_1,p_2)>1$. By Lemma \ref{lemma:prodspecreg}, the isometry $R$ is regular; so, its type depends only on its trace. By the corresponding trace formula on Remark \ref{rmk:partform}, the trace in question is determined by $\tance(p_1,p_2)$. Taking negative points $q_1,q_2$ such that $\tance(q_1,q_2)=\tance(p_1,p_2)$ we reduce the fact to a case already considered.
\end{proof}

\begin{cor}
\label{cor:connectSlox}
Let\/ $p_1,p_2,p_3$ and\/ $q_1,q_2,q_3$ be strongly regular triples with respect to the same $\pmb\alpha,\pmb\sigma,\tau$. Assume that\/ $R_{\alpha_3}^{p_3}R_{\alpha_2}^{p_2}R_{\alpha_1}^{p_1}$ and\/ $R_{\alpha_3}^{q_3}R_{\alpha_2}^{q_2}R_{\alpha_1}^{q_1}$ are in the same conjugacy class. If at least one of\/ $R_{\alpha_3}^{p_3}R_{\alpha_2}^{p_2},R_{\alpha_2}^{p_2}R_{\alpha_1}^{p_1}$ and at least one of\/ $R_{\alpha_3}^{q_3}R_{\alpha_2}^{q_2},R_{\alpha_2}^{q_2}R_{\alpha_1}^{q_1}$ is loxodromic, these triples can be connected, modulo conjugacy, by finitely many bendings.
\end{cor}

\begin{proof}
Suppose that (say) $R_{\alpha_3}^{p_3}R_{\alpha_2}^{p_2}$ is loxodromic. We can bend it in order to make $\big|\tance(p_1,p_2)\big|$ as big as wanted; by Corollary \ref{cor:tancetype}, $R_{\alpha_2}^{p_2}R_{\alpha_1}^{p_1}$ becomes loxodromic. So, we can assume that $R_1:=R_{\alpha_3}^{p_3}R_{\alpha_2}^{p_2}$ and $R_2:=R_{\alpha_3}^{q_3}R_{\alpha_2}^{q_2}$ are loxodromic. Bending $R_1$ and $R_2$, we can send $\tance(p_2,p_3),\tance(q_2,q_3)$ both to $\infty$ or $-\infty$ (depending on the given signs). Hence, we make $t_2:=\tance(p_2,p_3)=\tance(q_2,q_3)$. Now, the points in $S_{\pmb\alpha,\pmb\sigma,\tau}$ (see Theorem \ref{thm:surfaceS}) corresponding to the triples $p_1,p_2,p_3$ and $q_1,q_2,q_3$ lie in a same horizontal slice $H_{t_2}$. By item~({\it iv}\/) in Theorem \ref{thm:sbendingconnect}, $H_{t_2}$ is nonconnected only for a single value of $t_2$. Changing $t_2$ if necessary, we arrive at a connected $H_{t_2}$.
\end{proof}

\subsection{Experimental observations}
\label{subsec:experiments}

Here, we discuss a few experimental observations regarding the semialgebraic surface in Theorem \ref{thm:surfaceS} and the relative $\SU(2,1)$-character in Definition \ref{defi:charactervariety}.

Each picture in this section is given by the algebraic equation \eqref{eq:surfaceS} in Theorem \ref{thm:surfaceS}; parameters and trace are fixed and vertical/horizontal slices are also displayed. A surface $S_{\pmb\alpha,\pmb\sigma,\tau}$ of the type described in the theorem appears when one requires inequalities \eqref{eq:surfaceineq} to hold. When $2\real\kappa+1<0$, there are three possible combinations of signs leading to a surface $S_{\pmb\alpha,\pmb\sigma,\tau}$ and each of these surfaces is marked with a different color in the picture. If $2\real\kappa+1>0$, then all signs must be negative and there is a single surface $S_{\pmb\alpha,\pmb\sigma,\tau}$ in the picture (it is also indicated by a distinguished color). The gray part of the pictures do not satisfy the required inequalities and, therefore, do not correspond to strongly regular triples; however, they are very useful in understanding the dependence of $S_{\pmb\alpha,\pmb\sigma,\tau}$ on the choices involved (signs, parameters, and trace) and may be related to the study of more general character varieties (in this regard they should be compared, say, to those in \cite{ABL2017} and \cite{GB2012}).

\begin{figure}[!h]
\centering
\subcaptionbox{Surfaces for $\tau=0$.\label{fig:tauzero}}[.31\linewidth]{\includegraphics[scale=.3,trim={1.45cm .7cm .4cm 1cm},clip]{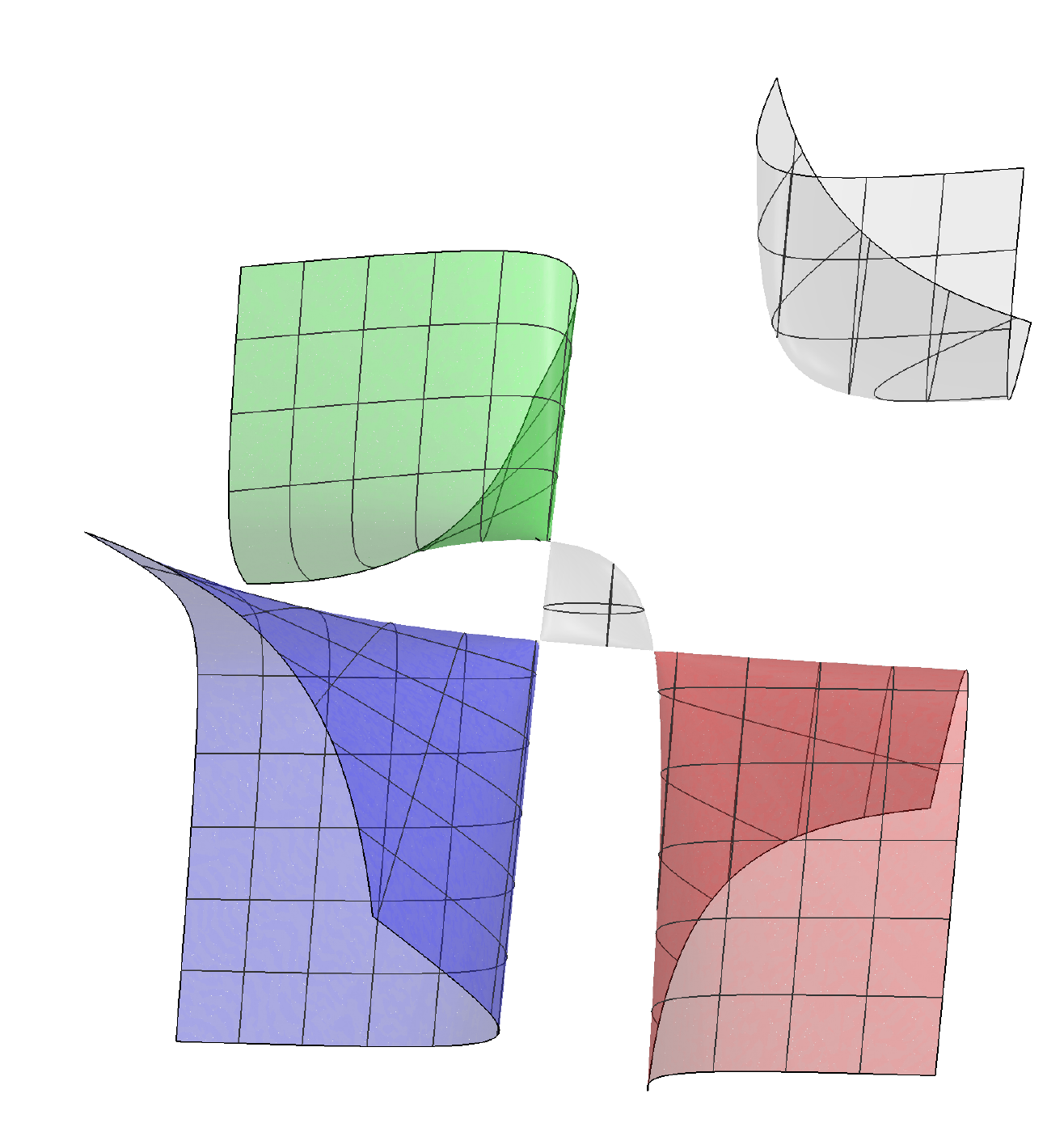}}
\subcaptionbox{Evolution of surfaces passing through an ellipto-parabolic case with a degenerate point.\label{fig:deg}}[.68\linewidth]{\includegraphics[scale=.3,trim={1cm 0 1cm 1cm},clip]{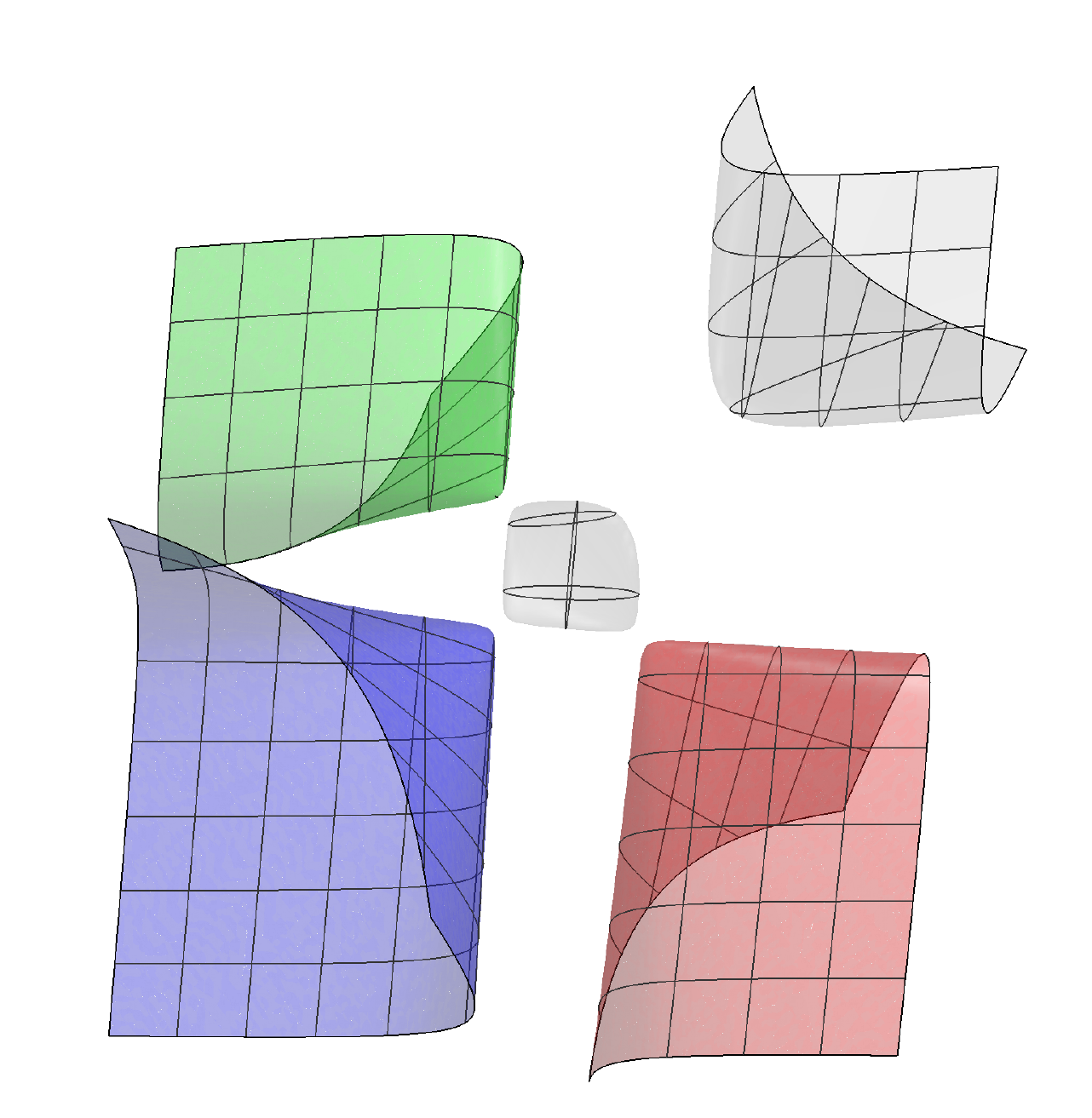}
\hspace{.5cm}
\includegraphics[scale=.3,trim={.5cm .5cm 0 1cm},clip]{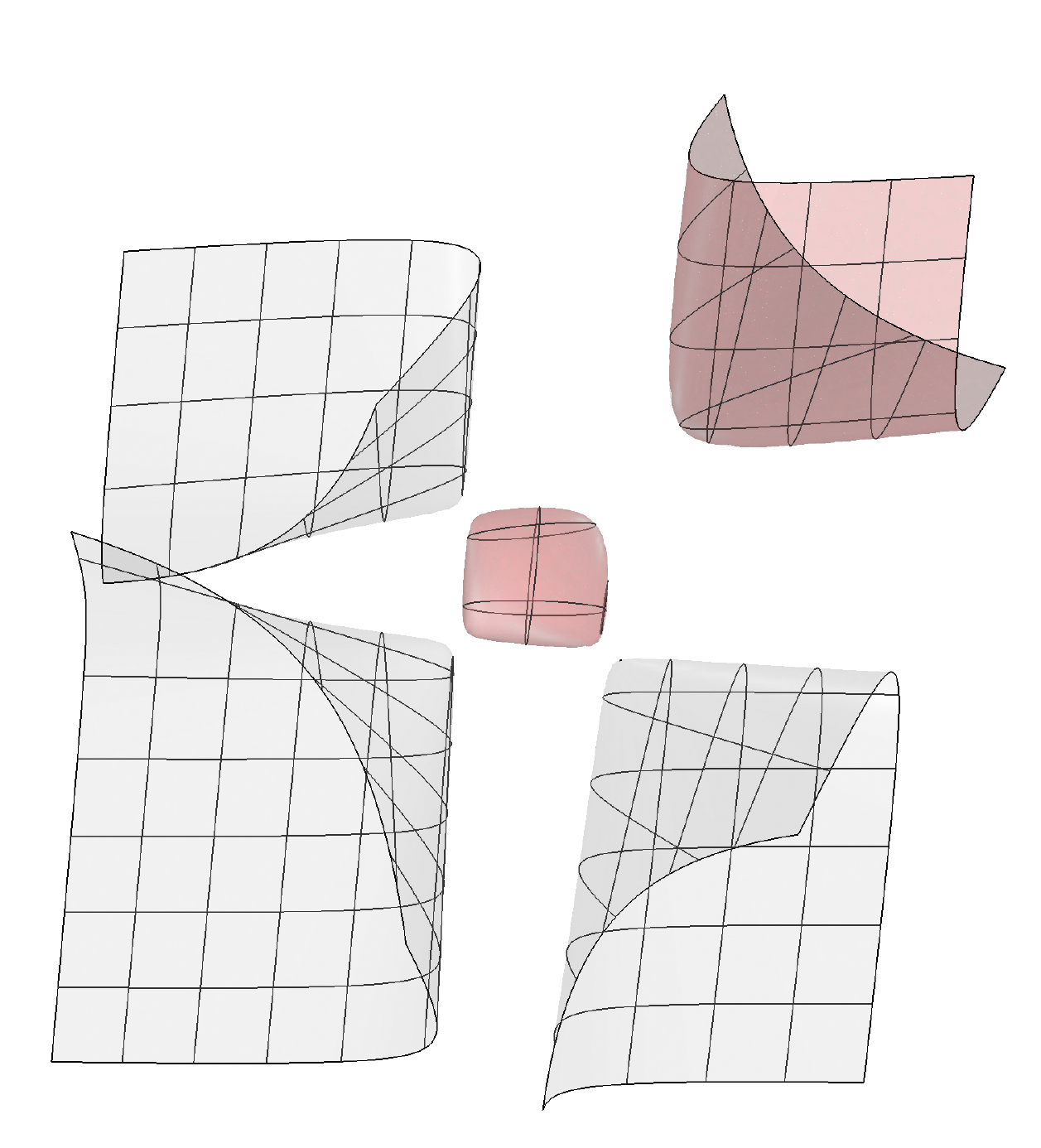}

\vspace{.1cm}

\includegraphics[scale=.3,trim={.5cm .6cm 0 .5cm},clip]{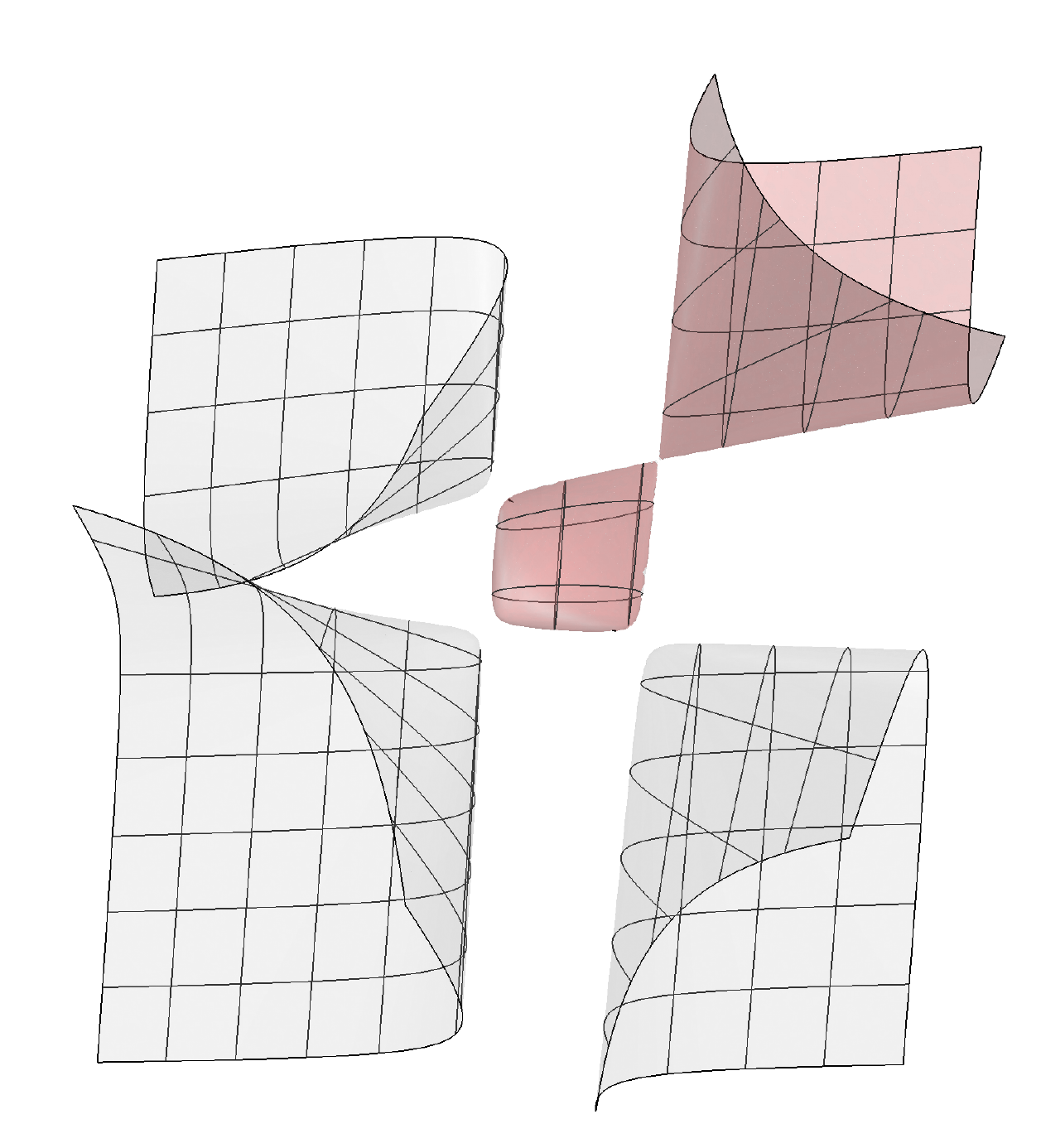}
\hspace{.5cm}
\includegraphics[scale=.3,trim={.5cm .6cm 0 .5cm},clip]{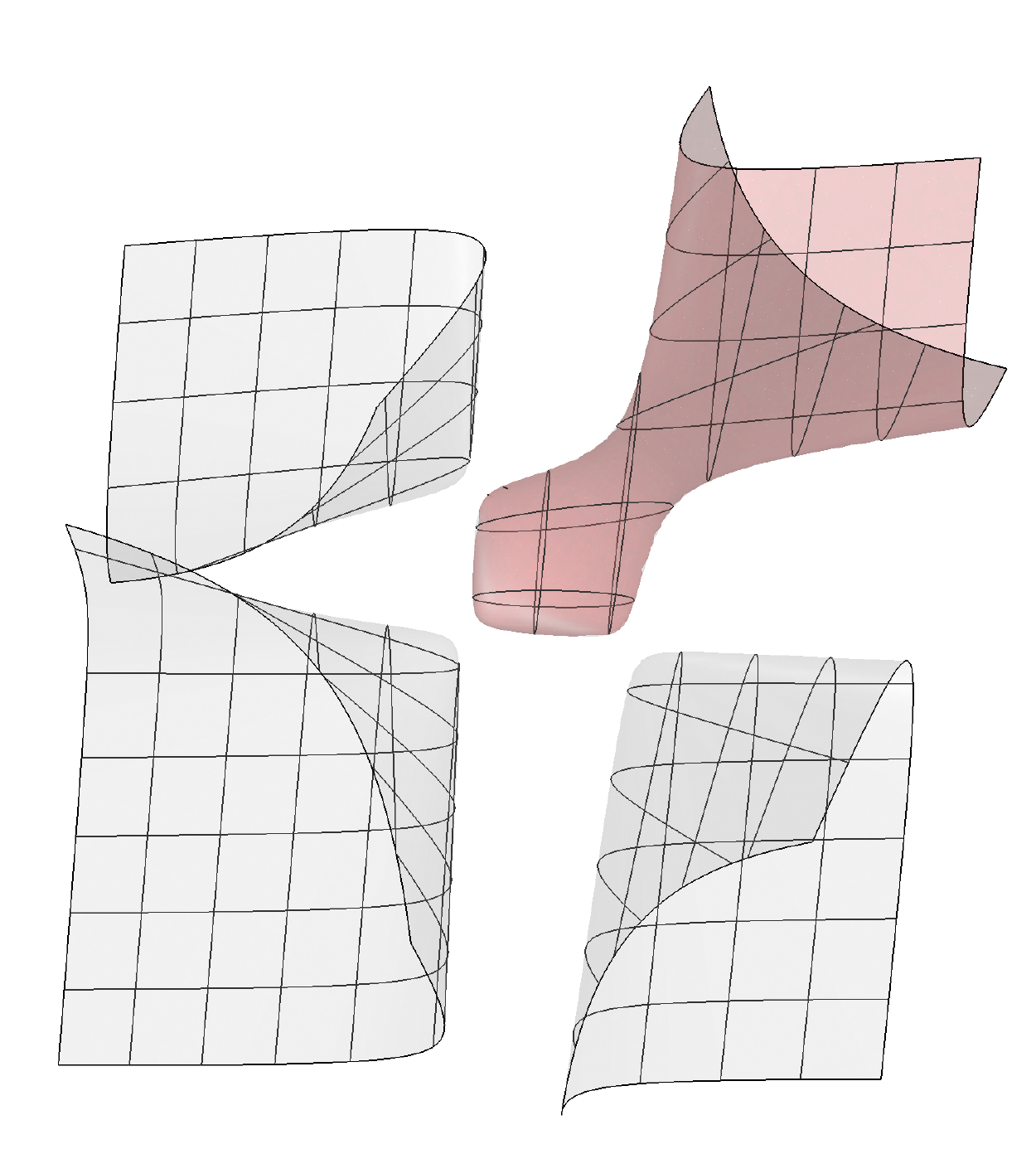}}
\caption{}\label{fig:tau0deg}
\end{figure}

\begin{figure}[!h]
\centering
\includegraphics[scale=.3,trim={.3cm .5cm 0 1.3cm},clip]{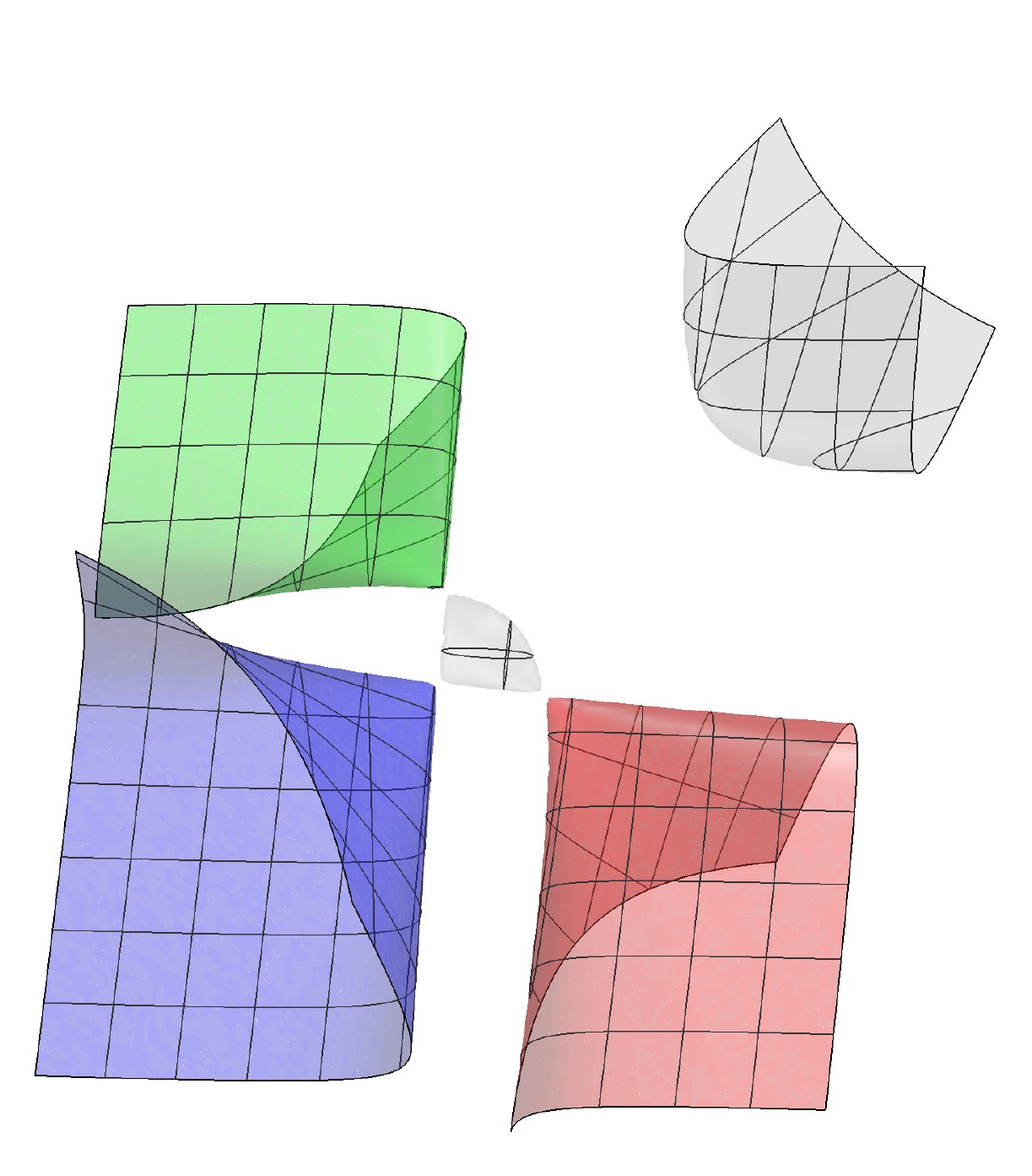}
\hspace{.5cm}
\includegraphics[scale=.3,trim={0 0 0 1.2cm},clip]{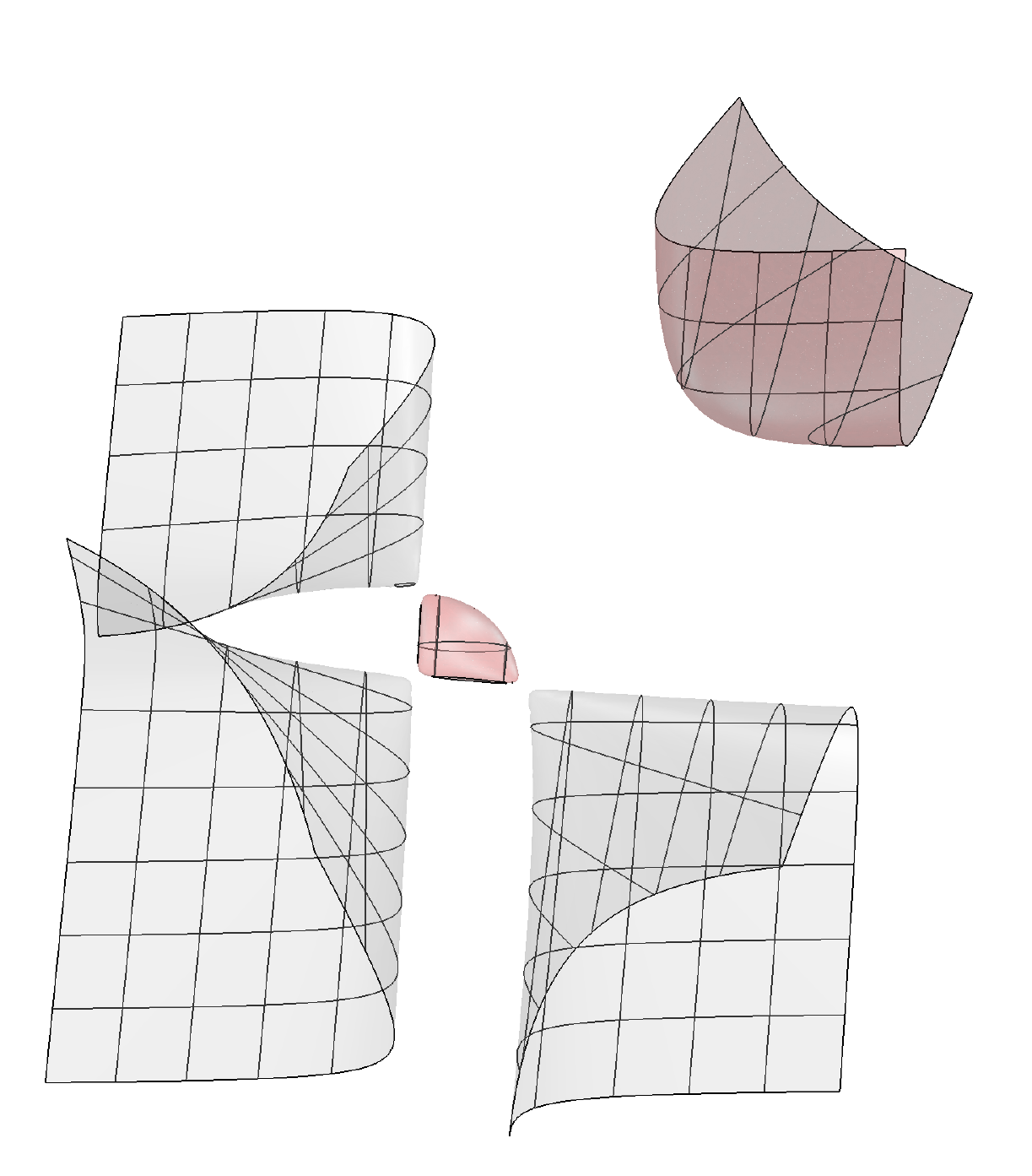}

\vspace{.1cm}

\includegraphics[scale=.3,trim={.3cm .6cm 0 .5cm},clip]{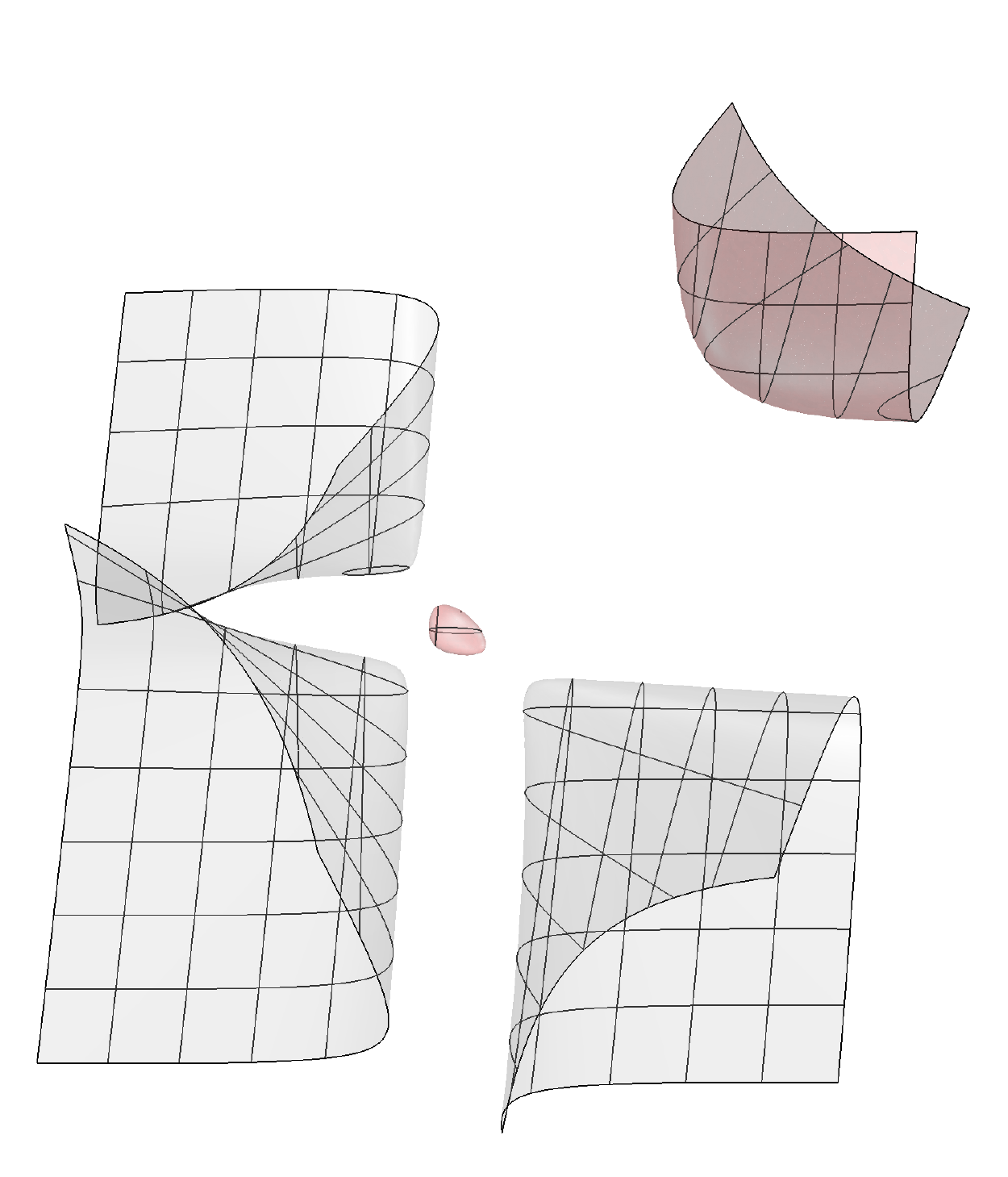}
\hspace{.5cm}
\includegraphics[scale=.3,trim={.3cm .6cm 0 .5cm},clip]{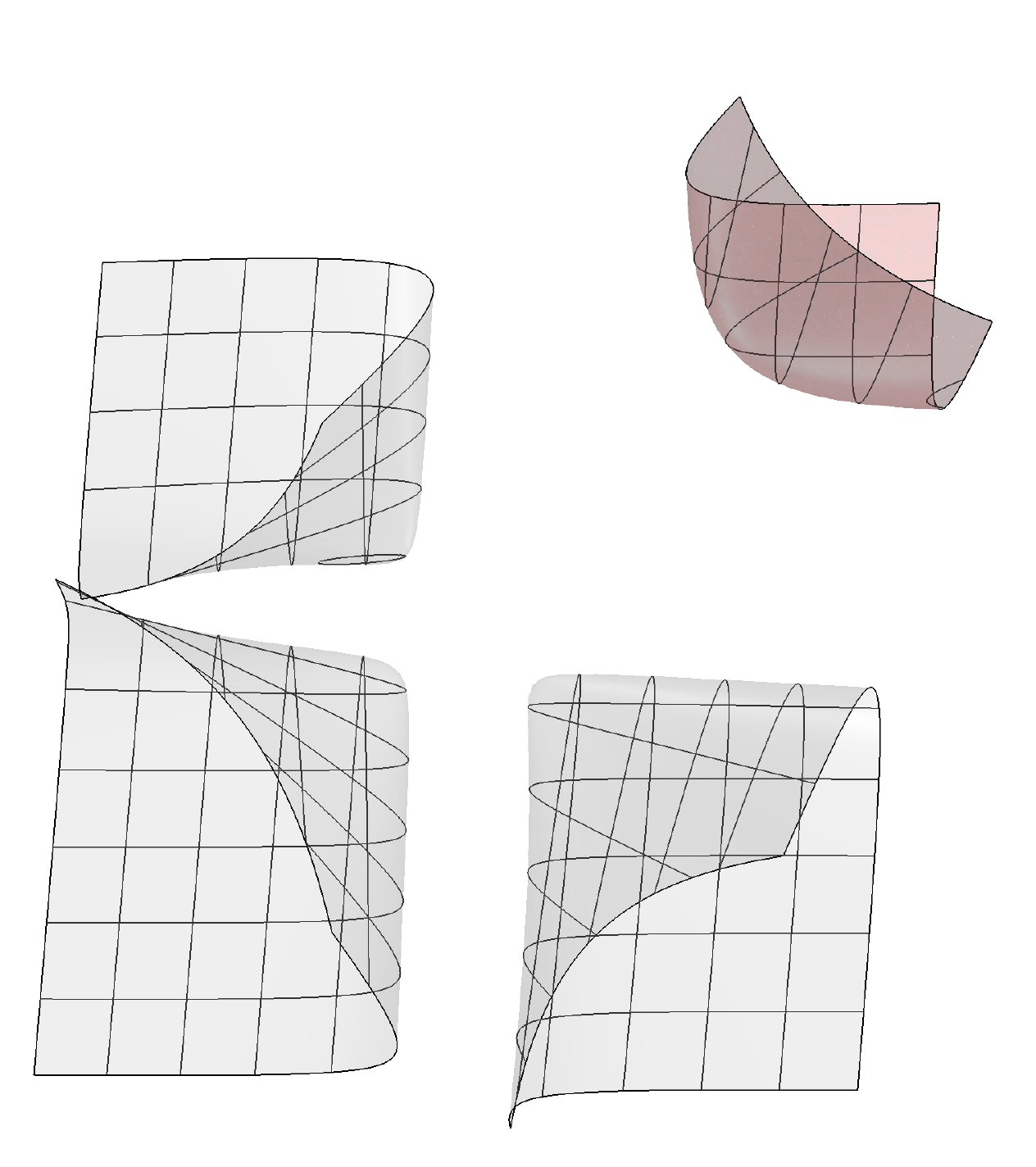}
\caption{Evolution of surfaces reaching an ellipto-parabolic case without a degenerate point.}
\label{fig:parabolicnondeg}
\end{figure}

In Figure \ref{fig:tauzero}, $\tau=0$. All surfaces in the picture correspond to a same $\SU(2,1)$-conjugacy class of a regular elliptic isometry of trace $0$. This case is distinguished: $\tau=0$ is the only trace value inside the deltoid to which corresponds a single $\PU(2,1)$-conjugacy class. Beginning with this case, we choose a direction in the complex plane and slowly change the trace in this direction until it (reaches and) leaves the deltoid. Each one of Figures~\ref{fig:deg}--\ref{fig:realtraces} display the behaviour of the surfaces during the trace deformation. They seem to include, from a qualitative point of view, every possible variant (for every choice of parameters).

The surfaces in the first and second pictures in Figure~\ref{fig:deg} contain the $\SU(2,1)$-character varieties in Theorem \ref{thm:charactervariety} where the class of the isometry $F$ is regular elliptic. In the second picture, points in distinct connected components correspond to distinct $\SU(2,1)$-conjugacy classes of regular elliptic isometries of the same trace. So, the inclusion in the third item of Theorem \ref{thm:charactervariety} is strict. The third picture corresponds to the conjugacy class of an ellipto-parabolic $F$. This surface contains a degenerate point. The remaining picture illustrates the loxodromic case.

Similarly to Figure~\ref{fig:deg}, pictures in Figure~\ref{fig:parabolicnondeg} range from the regular elliptic case (first three pictures) to an ellipto-parabolic one (the loxodromic case is not displayed because it looks almost identical to the ellipto-parabolic one). The situation is quite different from the previous one: the compact component simply vanishes when the trace reaches the deltoid instead of ``merging'' with a noncompact component.

Figure~\ref{fig:unipotent} illustrates a deformation passing through an unipotent class (second picture) instead of an ellipto-parabolic one. Here, the ``compact component'' (not belonging to any of the surfaces) collapses to a point when the trace reaches the deltoid.

Finally, in Figure~\ref{fig:realtraces}, the traces are always real. In this case, the ``compact component'' (not belonging to any of the surfaces) is always linked to the surfaces; it merges with another component in the ellipto-parabolic case (third picture).

\begin{figure}[!t]
\centering
\includegraphics[scale=.3,trim={.3cm .5cm .5cm 1cm},clip]{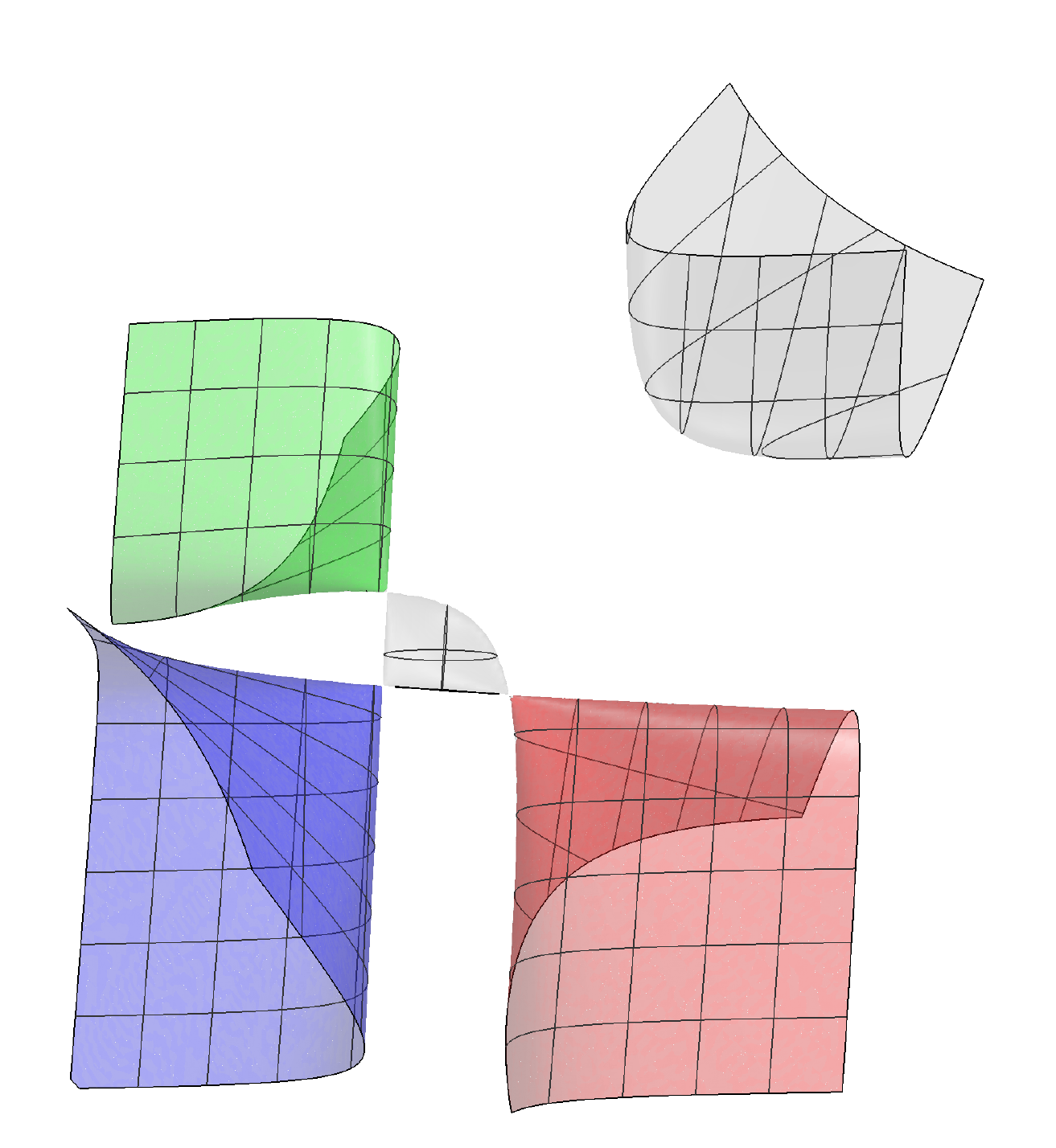}
\hspace{.5cm}
\includegraphics[scale=.3,trim={.2cm .5cm 0 1cm},clip]{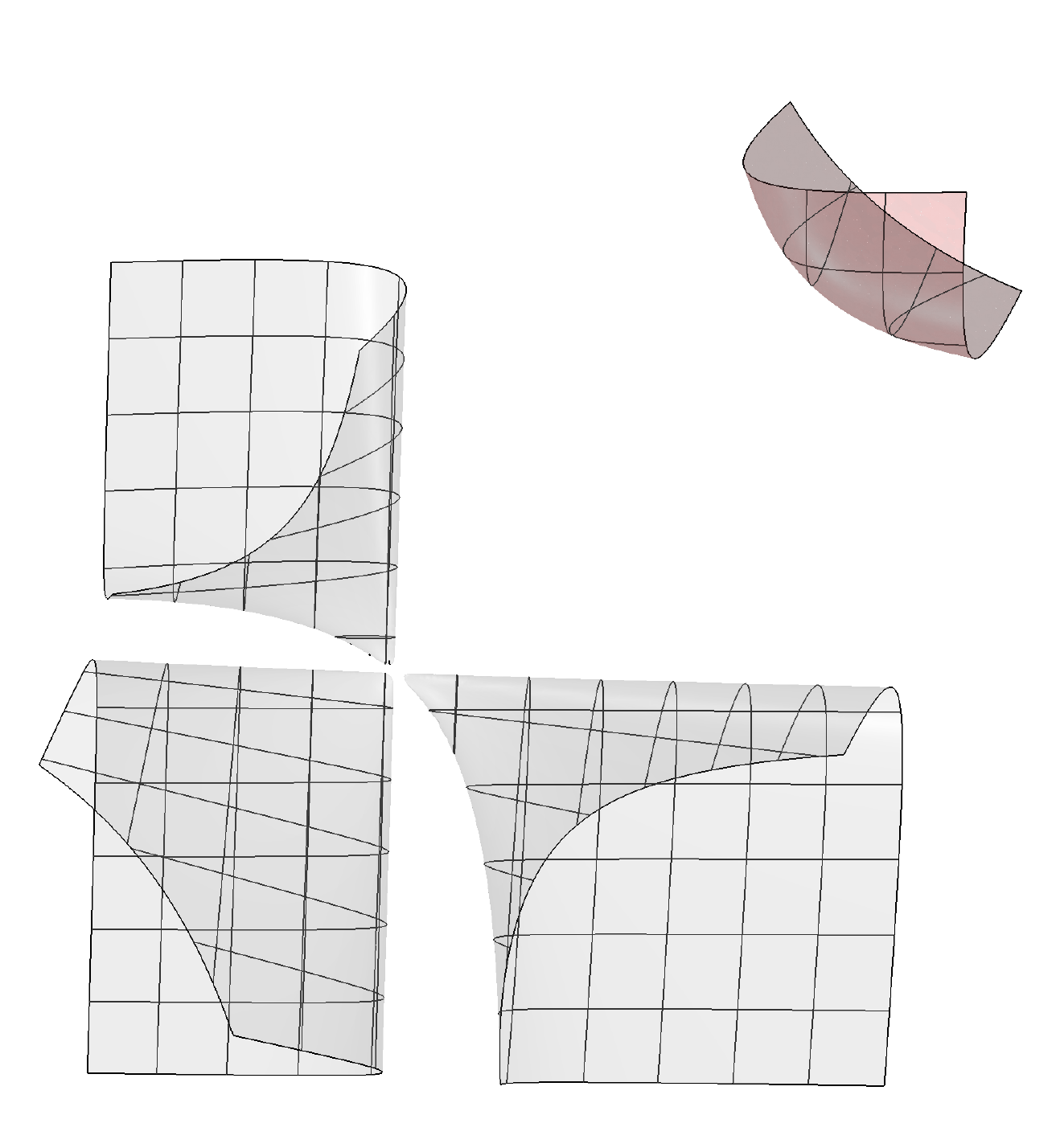}
\hspace{.5cm}
\includegraphics[scale=.3,trim={.2cm .5cm 0 1cm},clip]{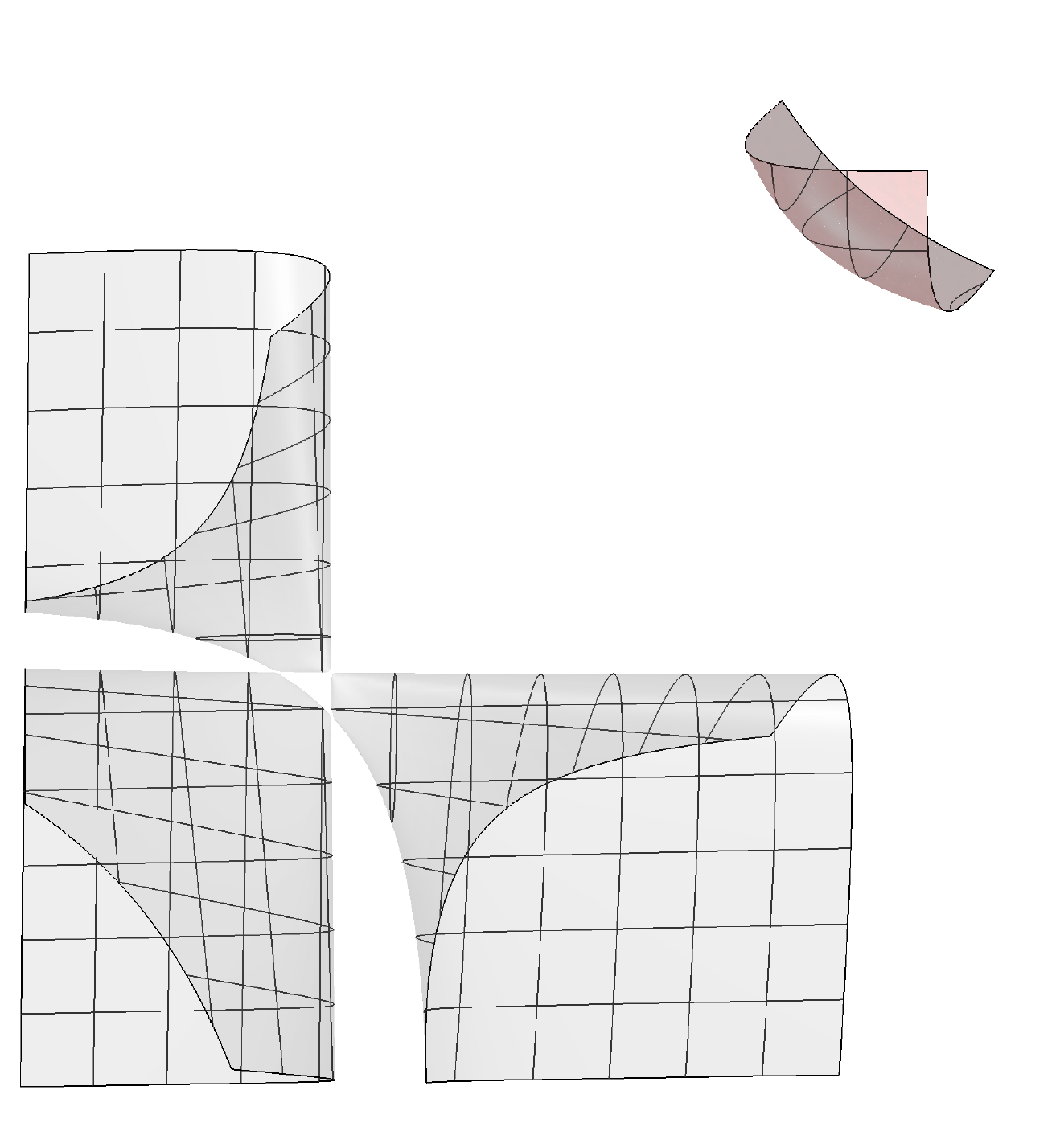}
\caption{Evolution of surfaces passing through a $3$-step unipotent case.}
\label{fig:unipotent}
\end{figure}

\begin{figure}[!t]
\centering
\includegraphics[scale=.3,trim={.4cm 1cm 0 1cm},clip]{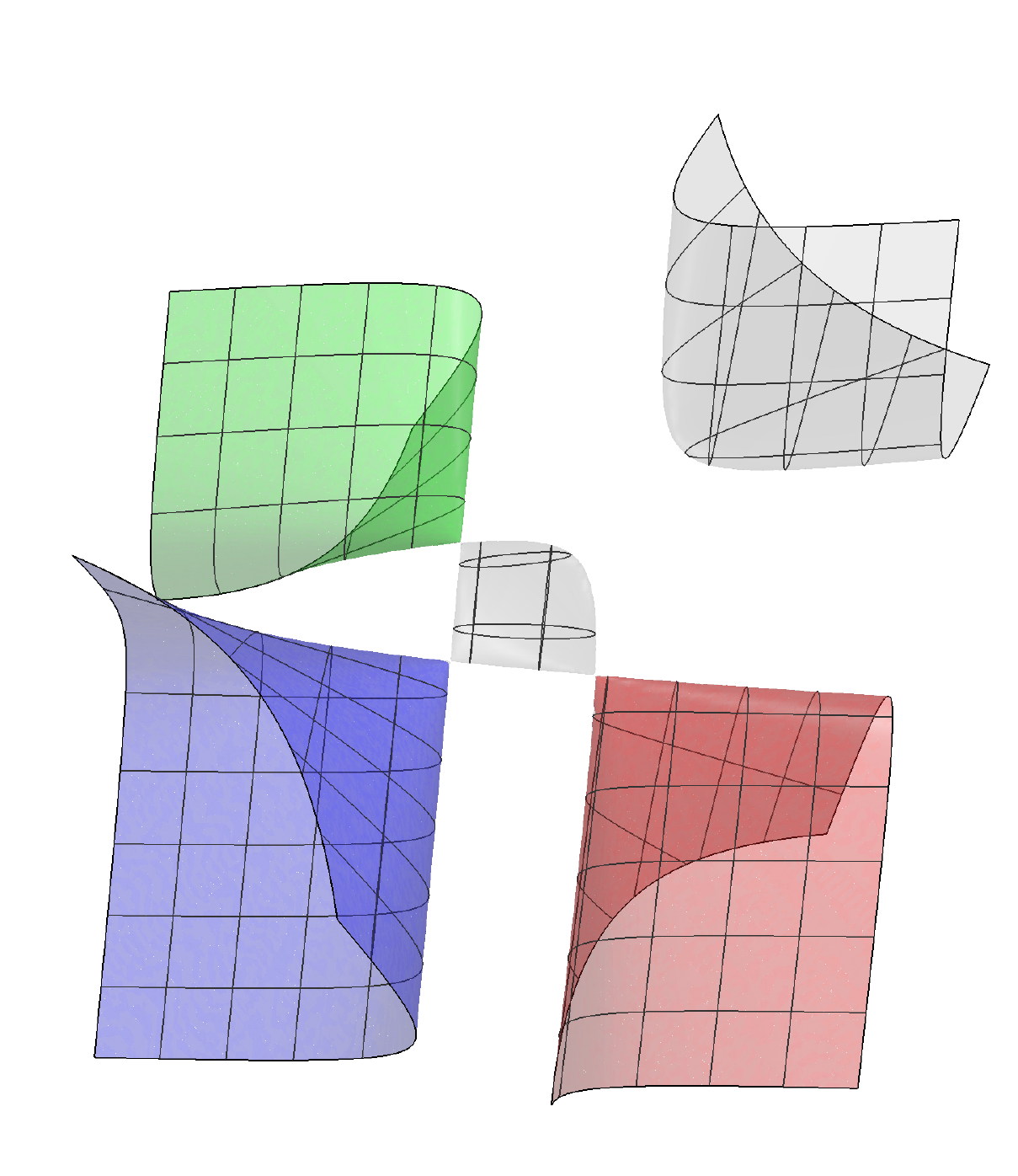}
\hspace{.5cm}
\includegraphics[scale=.3,trim={0 .5cm .5cm 1cm},clip]{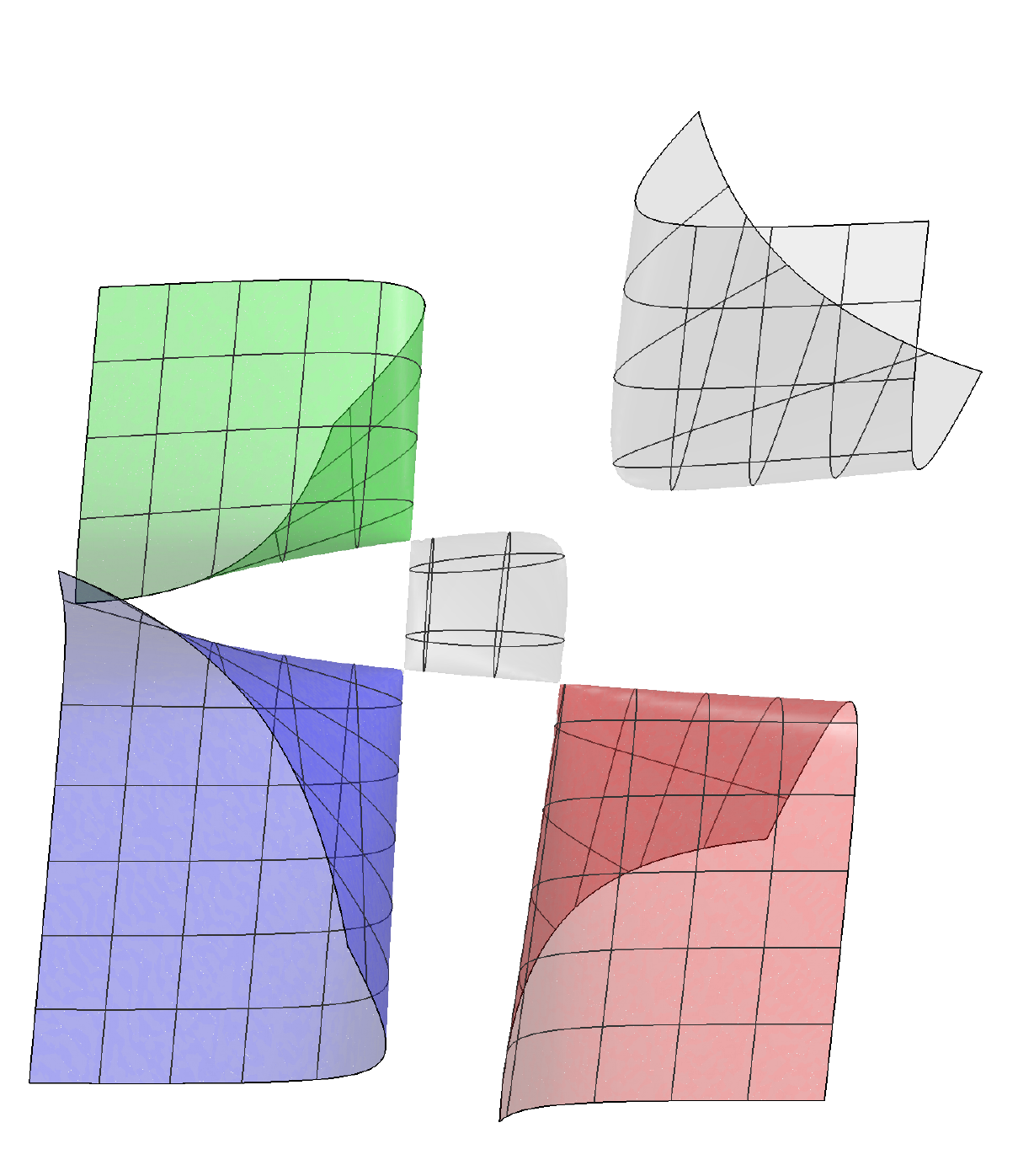}

\includegraphics[scale=.3,trim={.3cm .5cm 0 0},clip]{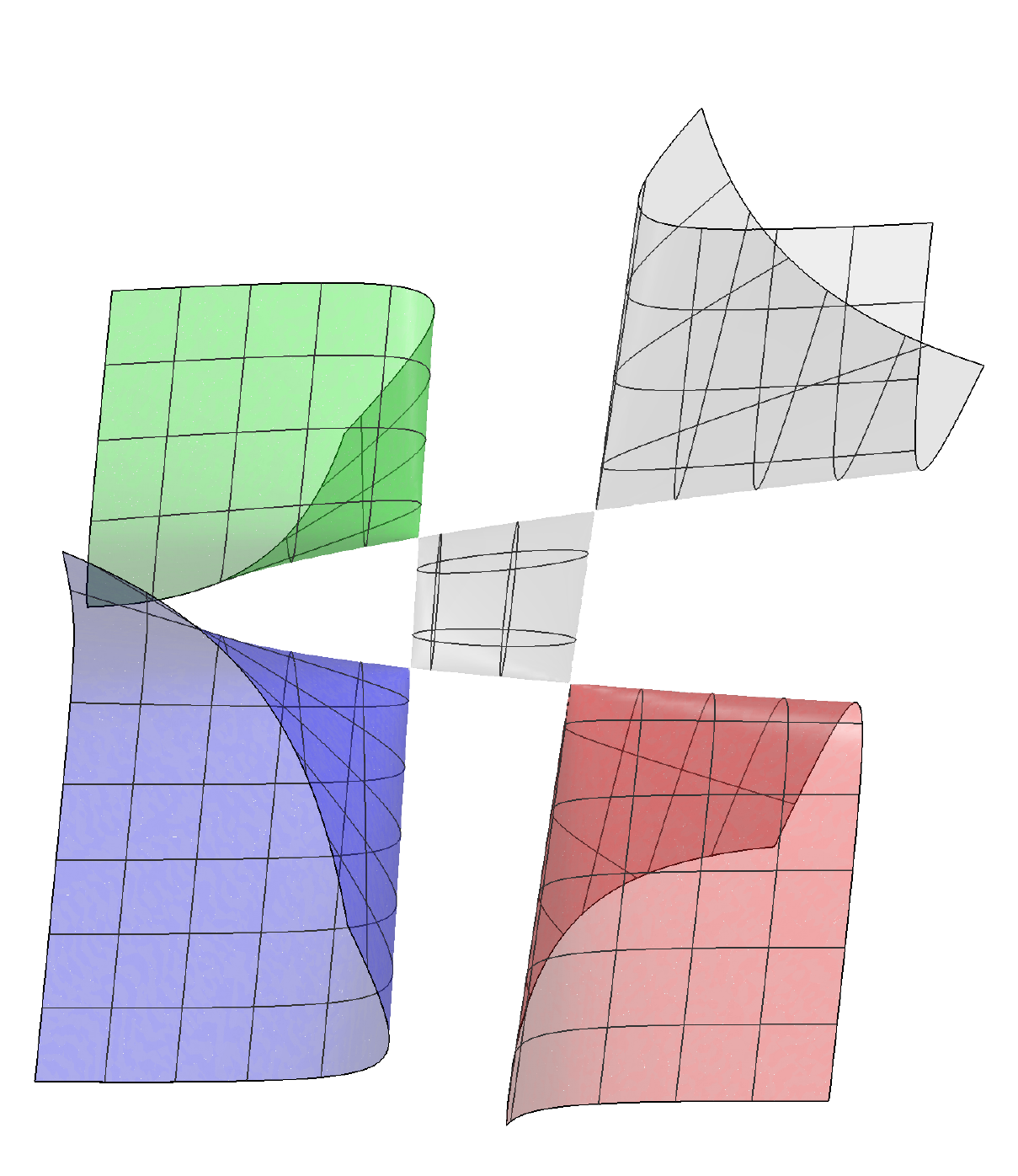}
\hspace{.5cm}
\includegraphics[scale=.3,trim={.3cm .5cm 0 0},clip]{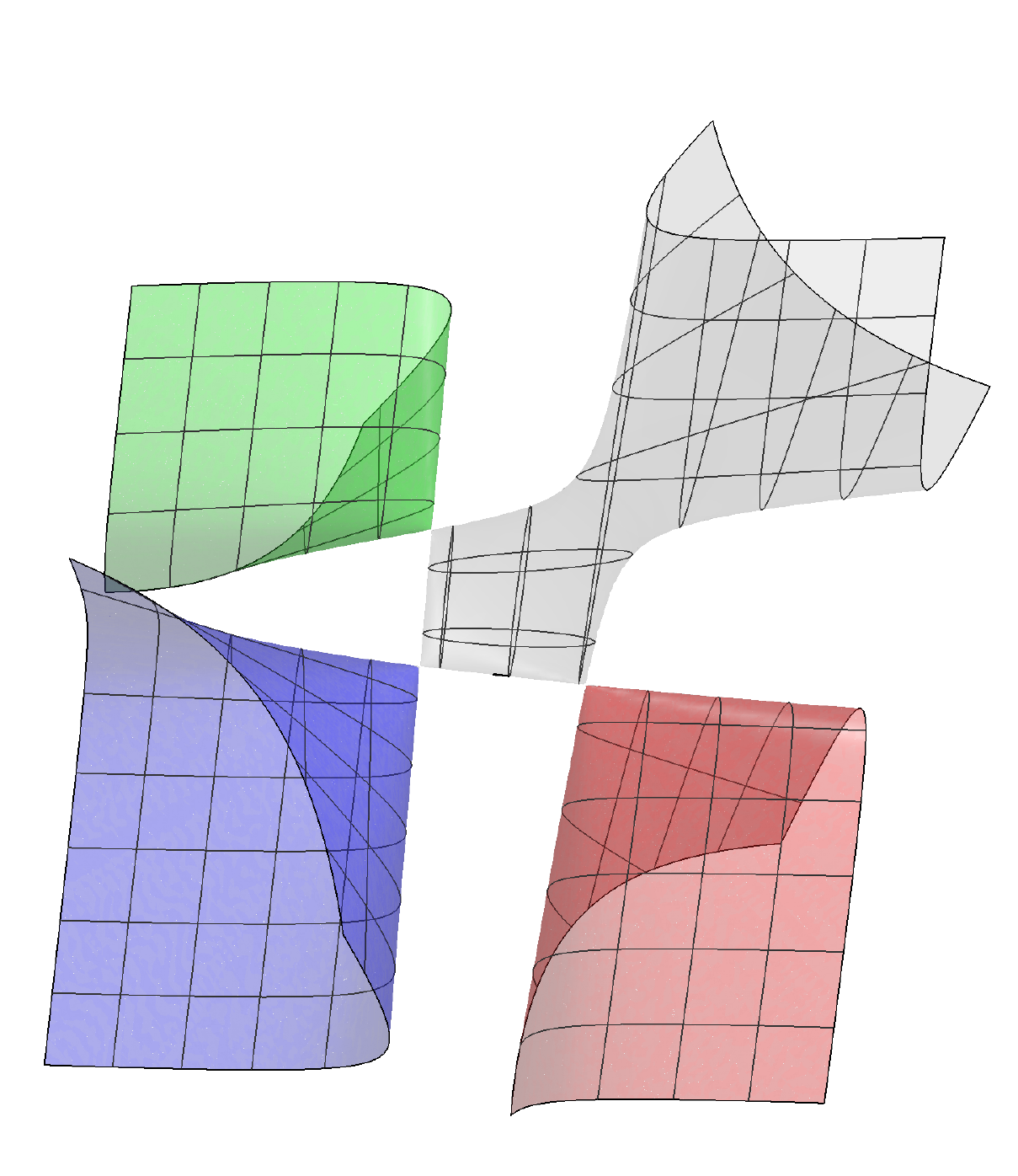}
\caption{Evolution of surfaces corresponding to real traces.}
\label{fig:realtraces}
\end{figure}

\section{Relations of length $4$: {\it f\/}-bendings and others}\label{sec:n-o-solutions}

If we are to keep the conjugacy classes of the special elliptic isometries, bendings provide all nonorthogonal generic length $4$ relations $R_{\alpha_2}^{p_2}R_{\alpha_1}^{p_1}=R_{\alpha_2}^{q_2}R_{\alpha_1}^{q_1}$ with $\delta=1$ and $\sigma p_i=\sigma q_i$ (see Theorem \ref{thm:bendingsarebendings}). Here, we describe all the remaining generic length $4$ relations.

\subsection{{\it f\/}-bendings}\label{subsec:fbendings}
First, let us describe the nonorthogonal generic length $4$ relations with $\delta=1$ and $\sigma p_i=\sigma q_i$ which are not bendings. Roughly speaking, they are constructed as follows: given a product $R_{\alpha_2}^{p_2}R_{\alpha_1}^{p_1}$, we find a one-parameter family of products $R_{\beta_2(t)}^{q_2(t)}R_{\beta_1(t)}^{q_1(t)}$ acting on the line $\mathrm L(p_1,p_2)$ as well as on its polar point in the same way as $R_{\alpha_2}^{p_2}R_{\alpha_1}^{p_1}$ does. As the next lemma shows, this in principle does not guarantee that we arrive at new length $4$ relations; however, as we show in Proposition~\ref{prop:constructfbendings}, a continuity argument settles the question.

\begin{lemma}
\label{lemma:regularprodspecreflec}
Let\/ $R,S\in\SU(2,1)$ be isometries that stabilize a noneuclidean line\/
$L:=\mathbb Pc^\bot$. Suppose that the actions of\/ $R$ and\/ $S$ on\/ $L$ coincide and that\/ $Rc=Sc$ {\rm(}that is, $c$ is an eigenvector of same eigenvalue for both\/ $R$ and\/ $S${\rm)}. Then\/ $R=S$ or\/ $R=R_{-1}^cS$.
\end{lemma}

\noindent{\it Proof.}
Assume that $R$ and $S$ are elliptic. Let $\lambda_i$
and $\mu_i$, $i=1,2,3$, be their respective eigenvalues with $\mu_1=\lambda_1$. Let $\eta\in\mathbb C$, $|\eta|=1$, be such that $\eta\lambda_2=\mu_2$. The fact that the actions of $R$ and $S$ on $L$ are the same implies that $\lambda_2\overline\lambda_3=\mu_2\overline\mu_3$; hence, $\eta\lambda_3=\mu_3$. Moreover, $\mu_2\mu_3=\lambda_2\lambda_3$ because $R,S\in\SU(2,1)$. It follows that $\eta=\pm1$ which concludes the proof in this case.

If $R$ and $S$ are loxodromic, we can write
$R=\left[\begin{smallmatrix}\lambda^{-1}\overline\lambda & 0 & 0\\
0 & \lambda & 0 \\
0 & 0 & \overline\lambda^{-1}
\end{smallmatrix}\right]$
and
$S=\left[\begin{smallmatrix}\mu^{-1}\overline\mu & 0 & 0 \\ 0 & \mu & 0 \\
0 & 0 & \overline\mu^{-1}
\end{smallmatrix}
\right]$,
where $\lambda^{-1}\overline\lambda$ and $\mu^{-1}\overline\mu$ are, respectively, the (equal) eigenvalues associated to $c$. Let $\eta\in\mathbb C$ be such that $\eta\lambda=\mu$. The fact that the actions of $R$ and $S$ on $L$ are the same implies that $\lambda\overline\lambda=\mu\overline\mu$. So, $|\eta|=1$. It now follows from $\lambda^{-1}\overline\lambda=\mu^{-1}\overline\mu$ that $\eta=\pm1$.

Finally, assume that $R$ and $S$ are parabolic. Since $R$ and $S$ stabilize a noneuclidean complex line, neither of them can be $3$-step unipotent (see Subsection \ref{subsec:isometries}). As in the proof of Proposition~\ref{prop:bendings}, we write
$R=\left[\begin{smallmatrix}\lambda^{-2} & 0 & 0\\
0 & \lambda & ir\lambda \\
0 & 0 & \lambda
\end{smallmatrix}\right]$
and
$S=\left[\begin{smallmatrix}\mu^{-2} & 0 & 0 \\ 0 & \mu & is\mu \\
0 & 0 & \mu
\end{smallmatrix}
\right]$,
where $|\lambda|=1$, $|\mu|=1$, $r,s\in\mathbb R$, and $\lambda^{-2}$ and
$\mu^{-2}$ are, respectively, the (equal) eigenvalues related to $c$. (The $2$-step unipotent case corresponds to $\lambda=1$ or $\mu=1$.)
Hence, $\mu=\pm\lambda$. The above matrices are written in a basis $c,v_1,v_2$, where $v_1,v_2\in L\cap\SV$, $v_1\ne v_2$, and $v_1$ is the isotropic fixed point of both $R$ and $S$. The action of (say) $R$ on $L\setminus\{v_2\}$ is given by $v_1+\tau v_2\mapsto v_1+\big(\frac{\tau}{1+ir\tau}\big) v_2$. By hypothesis, $R$ and $S$ act in the same way on $L$; so, $r=s$.$\hfill\square$

\begin{wrapfigure}{r}{0.2\textwidth}
\centering
\includegraphics[scale=0.7]{figs/fly.mps}
\end{wrapfigure}

\phantom{x}
\vspace{-2ex}

\begin{defi}\label{defi:fbending}
A relation $R_{\alpha_2}^{p_2}R_{\alpha_1}^{p_1}=R_{\beta_2}^{q_2}R_{\beta_1}^{q_1}$ is called an {\it $f$-bending\/} if: ({\it i\/})~$q_i\in G_i$, $i=1,2$, where $G,G_1,G_2$ are geodesics associated to  $R:=R_{\alpha_2}^{p_2}R_{\alpha_1}^{p_1}$ {\rm(}see Definition \ref{defi:geoassociated}{\rm)}; ({\it ii\/}) $\sigma q_i=\sigma p_i$ and $\langle p_i,q_i\rangle\ne0$, $i=1,2$; ({\it iii\/})~$\alpha_1\alpha_2=\beta_1\beta_2$ and $\alpha_i\sim\beta_i$ (see Definition \ref{defi:samecomponent}); ({\it iv\/}) neither $\mathrm G[p_1,q_1]$ nor $\mathrm G[p_2,q_2]$ contains a fixed point of $R$.
\end{defi}

It is worthwhile observing that, when $R_{\alpha_2}^{p_2}R_{\alpha_1}^{p_1}=R_{\beta_2}^{q_2}R_{\beta_1}^{q_1}$ is an $f$-bending relation, the condition $q_i\in G_i$ implies $G_i=H_i$, where $G,G_1,G_2$ and $H,H_1,H_2$ are the geodesics respectively associated to $R_1:=R_{\alpha_2}^{p_2}R_{\alpha_1}^{p_1}$ and to $R_2:=R_{\beta_2}^{q_2}R_{\beta_1}^{q_1}$.

In order to characterize $f$-bendings in Proposition \ref{prop:constructfbendings} and Theorem \ref{thm:geometricinter}, we need the following simple lemma whose proof is straightforward.

\begin{lemma}
\label{lemma:anglesequivalence}
Let $\alpha_i,\beta_i\in\mathbb S^1\setminus\Omega$ with $\alpha_i\sim\beta_i$ and let $a_i,b_i$ be the primitives of $\alpha_i,\beta_i$, $i=1,2$ {\rm(}see\/ {\rm Definition \ref{defi:samecomponent}).} Then $a_1a_2=b_1b_2$ iff $\alpha_1\alpha_2=\beta_1\beta_2$.
\end{lemma}

In what follows, we obtain $f$-bending relations via a certain deformation. Such a deformation will be shown to exist in Theorem \ref{thm:geometricinter}.

An $f$-{\it configuration\/} consists of two tuples $(p_1,p_2,\alpha_1,\alpha_2)$, $(q_1,q_2,\beta_1,\beta_2)$ satisfying the following conditions: ({\it i\/}) $p_1,p_2$ is a pair of distinct nonisotropic nonorthogonal points and the same holds for $q_1,q_2$; ({\it ii\/}) $\sigma p_i=\sigma q_i$ and $\langle p_i,q_i\rangle\ne0$, $i=1,2$; ({\it iii\/}) $\alpha_i,\beta_i\in\mathbb S^1\setminus\Omega$; ({\it iv\/}) $G_i=H_i$, $\angle_{p_1}G_1G=\Arg a_1^3$, $\angle_{p_2}GG_2=\Arg a_2^3$, $\angle_{q_1}G_1H=\Arg b_1^3$, $\angle_{q_2}HG_2=\Arg b_2^3$, where $a_i,b_i$ are the primitive angles of $\alpha_i,\beta_i$ and $G,G_1,G_2$ (respectively, $H,H_1,H_2$) are the geodesics associated to $R_{\alpha_2}^{p_2}R_{\alpha_1}^{p_1}$ (respectively, to $R_{\beta_2}^{q_2}R_{\beta_1}^{q_1}$); and ({\it v\/}) neither $\mathrm G[p_1,q_1]$ nor $\mathrm G[p_2,q_2]$ contains an intersection point of $G_1,G_2$. In other words, an $f$-configuration is essentially a pair of products $R_1:=R_{\alpha_2}^{p_2}R_{\alpha_1}^{p_1}$ and $R_2:=R_{\beta_2}^{q_2}R_{\beta_1}^{q_1}$ whose actions on $\mathrm{L}(p_1,p_2)=\mathrm{L}(q_1,q_2)$ coincide and whose associated geodesics $G_i,H_i$ have been made equal after a bending.

Let $(p_1,p_2,\alpha_1,\alpha_2)$, $(q_1,q_2,\beta_1,\beta_2)$ be an $f$-configuration. We take parameterizations $\gamma_i:[a,b]\to G_i$ such that $\sigma\gamma_i(t)=\sigma p_i$ and $\big\langle\gamma_1(t),\gamma_2(t)\big\rangle\ne0$ for all $t\in[a,b]$. Let $h_i=h_i(t)\in J$ be defined by $\angle_{\gamma_1(t)}G_1G'=\Arg h_1(t)^3$ and $\angle_{\gamma_2(t)}G'G_2=\Arg h_2(t)^3$, where $G'=G'_t:=\mathrm{G}{\wr}\gamma_1(t),\gamma_2(t){\wr}$. Moreover, let $\eta_i=\eta_i(t)\in\mathbb S^1$ be defined by $\eta_i\sim\alpha_i$ and $\eta_i(t)^3=(h_i(t)^2)^3$, $t\in[a,b]$. Hence, $\eta_i(t)$ is the parameter in the same component as $\alpha_i$ whose primitive equals $h_i(t)$.

We say that a given $f$-configuration is $f$-{\it connected\/} if there exist continuous parameterizations $\gamma_i$ as above such that $\gamma_i(a)=p_i$,
$\gamma_i(b)=q_i$, and $\eta_1(t)\eta_2(t)=\alpha_1\alpha_2$ for all $t\in[a,b]$. By Lemma \ref{lemma:anglesequivalence}, it is equivalent to require $h_1(t)h_2(t)=a_1a_2$ for all $t\in[a,b]$. Note that $\alpha_1\alpha_2=\beta_1\beta_2$ and $\alpha_i\sim\beta_i$ if the $f$-configuration is $f$-connected. The converse also holds:

\begin{thm}
\label{thm:geometricinter}
An\/ $f$-configuration\/ $(p_1,p_2,\alpha_1,\alpha_2)$, $(q_1,q_2,\beta_1,\beta_2)$ is\/ $f$-connected iff\/ $\alpha_1\alpha_2=\beta_1\beta_2$ and\/ $\alpha_i\sim\beta_i$, $i=1,2$.
\end{thm}

\noindent{\it Proof.} We already know that $f$-connectedness implies $\alpha_1\alpha_2=\beta_1\beta_2$ and $\alpha_i\sim\beta_i$.
Assume that $\alpha_1\alpha_2=\beta_1\beta_2$ and $\alpha_i\sim\beta_i$. Clearly, we can assume $p_i\ne q_i$, $i=1,2$, since $p_i=q_i$ is equivalent to $\alpha_i=\beta_i$.

Suppose that $\sigma p_1=\sigma p_2$. It follows from $\alpha_1\alpha_2=\beta_1\beta_2$ and Lemma \ref{lemma:anglesequivalence} that $A_1+A_2=B_1+B_2$, where $A_i:=\Arg a_i^3$, $B_i:=\Arg b_i^3$, and $a_i,b_i$ are the respective primitives of $\alpha_i,\beta_i$. It is now easy to see that the geodesic segments $\mathrm{G}[p_1,p_2]$ and $\mathrm{G}[q_1,q_2]$ intersect at a point $y$ and that the triangles $\Delta(p_1,y,q_1)$ and $\Delta(q_2,y,p_2)$ have opposite orientations and the same area (see item ({\it a\/}) in the figure below).

Let $\gamma_i:[0,\ell_i]\to\mathrm G[p_i,q_i]$ be the arc length parameterization of the geodesic segment $\mathrm G[p_i,q_i]$, $i=1,2$. Taking into account condition ({\it v\/}) in the definition of $f$-configuration, it follows from simple considerations of plane hyperbolic (or, according to the nature of the projective line $\mathrm L(p_1,p_2)$, of spherical) geometry that, given $s\in[0,\ell_1]$, there exists a unique $t=t(s)\in[0,\ell_2]$ such that the geodesic segments $\mathrm G[p_1,p_2]$ and $\mathrm G\big[\gamma_1(s),\gamma_2(t)\big]$ intersect at a point $y_s$ and the triangles $\Delta\big(p_1,y_s,\gamma_1(s)\big)$ and $\Delta\big(\gamma_2(t),y_s,p_2\big)$ have opposite orientations and the same area (see item~({\it b\/}) in the figure below). In this way, we obtain a strictly increasing bijection (therefore, continuous) $[0,\ell_1]\to[0,\ell_2]$ which leads to a new parameterization $\gamma_2:[0,\ell_1]\to\mathrm G[p_2,q_2]$. Now, $\gamma_1,\gamma_2$ satisfy the conditions in the definition of $f$-connectedness because $\eta_1(t)\eta_2(t)=\alpha_1\alpha_2$ follows from an argument analogous to the one in the previous paragraph.

\begin{figure}[!ht]
\centering
\includegraphics[scale=0.8]{figs/fbend01.mps}
\hspace{.3cm}
\includegraphics[scale=0.8]{figs/fbend02.mps}
\hspace{.3cm}
\includegraphics[scale=0.8]{figs/fbend03.mps}
\hspace{.3cm}
\includegraphics[scale=0.8]{figs/fbend04.mps}
\end{figure}

Suppose that $\sigma p_1\ne\sigma p_2$ (in particular, $L:=\mathrm L(p_1,p_2)$ is hyperbolic) and take the points $\tilde p_2,\tilde q_2\in L$ orthogonal to $p_2,q_2$. Let us show that $\mathrm{G}[p_1,\tilde p_2]\cap\mathrm{G}[q_1,\tilde q_2]=\varnothing$. Assume that the mentioned segments intersect at a point $y$. The areas of the triangles $\Delta(y,p_1,q_1)$ and $\Delta(y,\tilde p_2,\tilde q_2)$ equal $\pm(A_1-B_1)-C>0$ and $\pm(A_2-B_2)-C>0$ (depending on the orientation of the triangles) where $A_i,B_i$ are defined as above and $C$ is the interior angle of the triangles at $y$. It follows from Lemma~\ref{lemma:anglesequivalence} that $A_1+A_2=B_1+B_2$; this leads to $C<0$, a contradiction. Therefore, the quadrilateral of vertices $p_1,q_1,\tilde q_2,\tilde p_2$ is simple and its oriented area equals $\pm2(B_1-A_1)=\pm2(A_2-B_2)$.

Assume $\Arg a_1^3<\Arg b_1^3$ and take $[a,b]:=[\Arg a_1^3,\Arg b_1^3]$. Let $\gamma_1:[0,\ell_1]\to\mathrm G[p_1,q_1]$ and $\tilde\gamma_2:[0,\ell_2]\to\mathrm G[\tilde p_2,\tilde q_2]$ be arc length parameterizations. Given $t\in[a,b]$, it follows from usual arguments in plane hyperbolic geometry that: (1) there exists unique $s_1=s_1(t)\in[0,\ell_1]$, $s_2=s_2(t)\in[0,\ell_2]$ such that the oriented area of the quadrilateral of vertices $p_1,\gamma_1(s_1),\tilde\gamma_2(s_2),\tilde p_2$ equals $\pm2(A_1-t)$ (in order to see this, consider the geodesic $\Gamma_{t,s}$ through the point $\gamma_1(s)$ such that $\angle_{\gamma_1(s)}G_1\Gamma_{t,s}=t$ and vary $s\in[0,\ell_1]$); (2) the bijections $t\mapsto s_i(t)$, $i=1,2$, are strictly increasing (hence, continuous). We obtain new parameterizations $\gamma_1:[a,b]\to\mathrm G[p_1,q_1]$ and $\tilde\gamma_2:[a,b]\to\mathrm G[\tilde p_2,\tilde q_2]$ that imply the $f$-connectedness of the $f$-configurations in question because $\Area\big(p_1,\gamma_1(s_1),\tilde\gamma_2(s_2),\tilde p_2\big)=\pm2(A_1-t)$ is equivalent to $\eta_1(t)\eta_2(t)=\alpha_1\alpha_2$, where $\Area\big(p_1,\gamma_1(s_1),\tilde\gamma_2(s_2),\tilde p_2\big)$ denotes the oriented area of the oriented quadrilateral of vertices $p_1,\gamma_1(s_1),\tilde\gamma_2(s_2),\tilde p_2$.$\hfill\square$

\begin{prop}
\label{prop:constructfbendings}
An\/ $f$-configuration\/ $(p_1,p_2,\alpha_1,\alpha_2),(q_1,q_2,\beta_1,\beta_2)$ is\/ $f$-connected iff\/ $R_{\alpha_2}^{p_2}R_{\alpha_1}^{p_1}=R_{\beta_2}^{q_2}R_{\beta_1}^{q_1}$ is an\/ $f$-bending relation.
\end{prop}

\begin{proof}
If $R_{\alpha_2}^{p_2}R_{\alpha_1}^{p_1}=R_{\beta_2}^{q_2}R_{\beta_1}^{q_1}$ is an $f$-bending relation, one can readily see that $(p_1,p_2,\alpha_1,\alpha_2)$, $(q_1,q_2,\beta_1,\beta_2)$ is an $f$-configuration with $\alpha_1\alpha_2=\beta_1\beta_2$ and $\alpha_i\sim\beta_i$. It follows from Theorem \ref{thm:geometricinter} that the $f$-configuration is $f$-connected. Conversely, let $\gamma_i:[a,b]\to\mathrm{G}[p_i,q_i]$, be the parameterizations associated to the $f$-connectedness of the $f$-configuration. By construction, $G'_t,G_1,G_2$, where $G'_t:=\mathrm{G}{\wr}\gamma_1(t),\gamma_2(t){\wr}$, are the geodesics associated to $S=S(t):=R_{\eta_2(t)}^{\gamma_2(t)}R_{\eta_1(t)}^{\gamma_1(t)}$ (see the above definition of $f$-connectedness for the definition of the functions $\eta_i$). Lemma \ref{lemma:construction} implies that the actions of $R:=R_{\alpha_2}^{p_2}R_{\alpha_1}^{p_1}$ and of $S$ on $L:=\mathrm L(p_1,p_2)$ coincide. Furthermore, it follows from $\eta_1(t)\eta_2(t)=\alpha_1\alpha_2$ that $Rc=S(t)c$, where $c$ is the polar point of $L$. So, by Lemma~\ref{lemma:regularprodspecreflec} and by continuity, either $R=S(t)$ or $R=R_{-1}^cS(t)$ for all $t$. It remains to take $t=a,b$.
\end{proof}

In view of Theorem \ref{thm:geometricinter}, it is natural to ask the following. Given a product $R_{\alpha_2}^{p_2}R_{\alpha_1}^{p_1}$ and parameters $\beta_i\in\mathbb S^1\setminus\Omega$ satisfying $\alpha_1\alpha_2=\beta_1\beta_2$ and $\alpha_i\sim\beta_i$, do there exist $q_i\in G_i$, $\sigma q_i=\sigma p_i$, such that $R_{\alpha_2}^{p_2}R_{\alpha_1}^{p_1}=R_{\beta_2}^{q_2}R_{\beta_1}^{q_1}$? The answer is affirmative at least in the particular cases dealt with in the next couple of propositions.

\begin{prop}
\label{prop:fbendingangles}
Let\/ $p_i\in\PCV\setminus\SV$ be distinct nonorthogonal points of the same signature such that\/ $\mathrm L(p_1,p_2)$ is hyperbolic and let\/
$\alpha_i,\beta_i\in\mathbb S^1\setminus\Omega$ be parameters, $i=1,2$.
Assume that\/ $\alpha_1\alpha_2=\beta_1\beta_2$ and\/ $\alpha_i\sim\beta_i$, $i=1,2$. There exists an\/ $f$-bending relation\/ $R_{\alpha_2}^{p_2}R_{\alpha_1}^{p_1}=R_{\beta_2}^{q_2}R_{\beta_1}^{q_1}$.
\end{prop}

\begin{proof}
Let $a_i,b_i$ be the primitives of $\alpha_i,\beta_i$. We divide the proof in two cases:

\smallskip

\noindent
(1) $\Arg (a_1a_2)^3\leqslant\pi$, or, equivalently, $a_1a_2$ lie in the arc $(1,-\omega^2]$ (if $\Arg(a_1a_2)^3=0$, then $a_1=a_2=-\omega^2$)

\noindent
(2) $\Arg(a_1a_2)^3>\pi$, or, equivalently, $a_1a_2$ lie in the arc $(-\omega^2,\omega)$.

\smallskip

Assume that the first case holds. Proceeding as in the proof of Theorem \ref{thm:geometricinter}, we continuously move the point $p_i$ along $G_i$. Abusing notation, we denote the new obtained points by $q_i$ and new corresponding parameters by $\eta_i$. The deformation is performed in such a way that the $f$-configurations $(p_1,p_2,\alpha_1,\alpha_2),(q_1,q_2,\eta_1,\eta_2)$ are $f$-connected. Due to $a_1a_2\in\arc(1,-\omega^2]$, sending $q_1$ to the absolute makes the angle $\Arg h_1^3$ as small as desired; here, $h_i$ stands for the primitive of $\eta_i$. So, $h_1$ assumes every value in $\arc(1,a_1)$. Since $h_1h_2=a_1a_2$ during the deformation, $h_2$ assumes every value in  $\arc(a_2,a_1a_2)$ (and the point $q_2$ tends to some limit point in $G_2$). Changing the roles of $q_1,q_2$ in the deformation, that is, sending $q_2$ to the absolute, we can see that there exists $q_1,q_2$ corresponding to every value of $h_1$ in the arc $(1,a_1a_2)$. Since $\alpha_1\alpha_2=\beta_1\beta_2$ and $\alpha_i\sim\beta_i$, Lemma~\ref{lemma:anglesequivalence} implies that $a_1a_2=b_1b_2$ and, therefore,
$b_1,b_2\in\arc(1,a_1a_2)$. It remains to take $h_1=b_1$ (which implies $h_2=b_2)$: the corresponding $q_1,q_2$ are the points we are looking for.

The second case is similar. Sending $q_1$ to the absolute makes the angle $\Arg h_1^3$ tend to $\pi$ and, consequently, $h_1$ assumes every value in $\arc(a_1,-\omega^2)$; correspondingly, $h_2$ assumes every value in $\arc(-\omega a_1a_2,a_2)$. So, there exist $q_1,q_2$ corresponding to every value of $h_1$ in $\arc(-\omega a_1a_2,-\omega^2)$. Again, $a_1a_2=b_1b_2$ implies that $b_1\in\arc(-\omega a_1a_2,-\omega^2)$.
\end{proof}

\begin{prop}
\label{prop:hypfbendingangles}
Let\/ $p_1,p_2\in\mathbb PV\setminus\SV$ be distinct nonorthogonal points of distinct signatures, $\sigma p_1\ne\sigma p_2$, and let\/ $\alpha_1,\alpha_2\in\mathbb S^1\setminus\Omega$ be such that\/ $R_{\alpha_2}^{p_2}R_{\alpha_1}^{p_1}$ is loxodromic. Let\/ $\beta_1,\beta_2\in\mathbb S^1$ be parameters satisfying\/ $\alpha_1\alpha_2=\beta_1\beta_2$ and\/ $\alpha_i\sim\beta_i$, $i=1,2$. There exists an\/ $f$-bending relation\/ $R_{\alpha_2}^{p_2}R_{\alpha_1}^{p_1}=R_{\beta_2}^{q_2}R_{\beta_1}^{q_1}$.
\end{prop}
\begin{proof}
The proof is a direct adaptation of that of Proposition \ref{prop:fbendingangles}. The only change that is needed is to replace the deformation to the appropriate one describe in the proof of Theorem \ref{thm:geometricinter}.
\end{proof}

\begin{rmk}\label{rmk:sameangles}
In the conditions of Propositions \ref{prop:fbendingangles} or \ref{prop:hypfbendingangles}, there always exist an $f$-bending relation $R_{\alpha_2}^{p_2}R_{\alpha_1}^{p_1}=R_{\beta_2}^{q_2}R_{\beta_1}^{q_1}$ such that the primitives of $\beta_1$ and $\beta_2$ are the same, i.e., $\beta_1=\delta\beta_2$, $\delta\in\Omega$ (this is a direct consequence of the corresponding proofs). Moreover, when $p_1,p_2\in\mathbb PV$ are distinct nonorthogonal points such that $\mathrm L(p_1,p_2)$ is spherical, and given parameters $\alpha_1,\alpha_2\in\mathbb S^1\setminus\Omega$, it is easy to see that there exists an $f$-bending relation $R_{\alpha_2}^{p_2}R_{\alpha_1}^{p_1}=R_{\beta_2}^{q_2}R_{\beta_1}^{q_1}$ such that the primitives of $\beta_1$ and $\beta_2$ are the same.
\end{rmk}

\begin{rmk}\label{rmk:flimits}
Let $\alpha_1,\alpha_2\in\mathbb S^1\setminus\Omega$ be parameters and let $p_1,p_2\in\mathbb PV\setminus\SV$ be distinct nonisotropic nonorthogonal points. Assume that $\sigma p_1=\sigma p_2$ with $\mathrm{L}(p_1,p_2)$ hyperbolic or that $R_{\alpha_2}^{p_2}R_{\alpha_1}^{p_1}$ is loxodromic. Let $I_j=\arc(\omega^{j-1},\omega^j)$ be the component of $\alpha_1$ and let $\delta\in\Omega$ be such that $\delta(\alpha_1\alpha_2)\in I_j$. Among the arcs $\arc\big(\omega^{j-1},\delta(\alpha_1\alpha_2)\big)$ and $\arc\big(\delta(\alpha_1\alpha_2),\omega^j\big)$, let $I_j^0$ be the one containing $\alpha_1$. A~corollary to the proof of Proposition \ref{prop:fbendingangles} is that, by $f$-bending $R_{\alpha_2}^{p_2}R_{\alpha_1}^{p_1}$, it is possible to vary $\alpha_1$ inside the entire $I_j^0$.
\end{rmk}

We can now give a description of all nonorthogonal generic length $4$ relations in terms of bendings and $f$-bendings (there is actually a nongeneric requirement on the signatures of points; the remaining case is dealt with in Subsection \ref{subsec:otherrel}).

\begin{thm}\label{thm:all4relations}
Let\/ $p_1,p_2$ and\/ $q_1,q_2$ be pairs of distinct nonisotropic nonorthogonal points with\/ $\sigma p_i=\sigma q_i$ and let\/ $\alpha_i,\beta_i\in\mathbb S^1\setminus\Omega$ be parameters, $\alpha_i\sim\beta_i$, $i=1,2$. Assume that\/ $\mathrm{L}(p_1,p_2)$ and\/ $\mathrm{L}(q_1,q_2)$ are noneuclidean and nonorthogonal. A length\/ $4$ relation\/ $R_{\alpha_2}^{p_2}R_{\alpha_1}^{p_1}=R_{\beta_2}^{q_2}R_{\beta_1}^{q_1}$ follows from bending and $f$-bending relations.
\end{thm}

\begin{proof}
Let $G,G_1,G_2$ and $H,H_1,H_2$ be the geodesics associated to
$R_{\alpha_2}^{p_2}R_{\alpha_1}^{p_1}$ and $R_{\beta_2}^{q_2}R_{\beta_1}^{q_1}$, respectively.
By Lemma~\ref{lemma:ortnotortl4}, we have $L:=\mathrm L(p_1,p_2)=\mathrm L(q_1,q_2)$ and $\alpha_1\alpha_2=\beta_1\beta_2$. Bending $q_1,q_2$ (if necessary), we make $G_i=H_i$.

Assume that $L$ is hyperbolic. It follows from $\alpha_1\alpha_2=\beta_1\beta_2$ that $A_1+A_2=B_1+B_2$, where $A_1=\angle_{p_1}G_1G$, $A_2=\angle_{p_2}GG_2$, $B_1=\angle_{q_1}G_1H$, $A_2=\angle_{q_2}HG_2$. Now, proceeding as in the proof of Theorem \ref{thm:bendingsarebendings}, we can see that the previous bending can be made in a such a way that there is no fixed point of $R:=R_{\alpha_2}^{p_2}R_{\alpha_1}^{p_1}$ neither in $\mathrm G[p_1,q_1]$ nor in $\mathrm G[p_2,q_2]$. We just arrived at an $f$-bending relation.

Suppose that $L$ is spherical. If the orientations of the triangles $\Delta(x,p_1,p_2)$ and $\Delta(x,q_1,q_2)$, where $x$ is a fixed point of $R$, are the same, then the above bending can be made in such a way that we obtain an $f$-bending relation. So, assume that the triangles have opposite orientations. By Remark \ref{rmk:sameangles}, we can assume $A_1=A_2$ and $B_1=B_2$ modulo $f$-bendings. It follows from $\alpha_1\alpha_2=\beta_1\beta_2$ that $A_1=A_2=B_1=B_2$ and can now proceed as in the proof of Theorem \ref{thm:bendingsarebendings}.
\end{proof}

\subsection{Other length 4 relations}
\label{subsec:otherrel}

In what follows, we introduce and discuss the remaining length $4$ relations. The nonorthogonal ones
come in four flavours: changes of orientation, changes of components, simultaneous changes of
signs, and single changes of sign.

Let $p_1,p_2$ be distinct nonisotropic nonorthogonal points with noneuclidean $\mathrm{L}(p_1,p_2)$ and let $\alpha_1,\alpha_2\in\mathbb S^1\setminus\Omega$. Note that
$$R_{\alpha_2}^{p_2}R_{\alpha_1}^{p_1}=R_{\alpha_2}^{p_2}R_{\alpha_1}^{p_1}
R_{\overline\alpha_2}^{p_2}R_{\alpha_2}^{p_2}=R_{\alpha_1}^{R_{\alpha_2}^{p_2}p_1}
R_{\alpha_2}^{p_2},$$
where we are using a cancellation (see Definition~\ref{defi:cancellations}) in the first equality. We arrive at the following definition.

\begin{defi}
In the conditions above, the relation $R_{\alpha_2}^{p_2}R_{\alpha_1}^{p_1}=R_{\alpha_1}^{R_{\alpha_2}^{p_2}p_1}R_{\alpha_2}^{p_2}$ is called a {\it change of orientation.}
\end{defi}

\begin{defi}
The relations $R_{\alpha_2}^{p_2}R_{\alpha_1}^{p_1}=R_{\omega\alpha_2}^{p_2}R_{\omega^{-1}\alpha_1}^{p_1}=
R_{\omega^{-1}\alpha_2}^{p_2}R_{\omega\alpha_1}^{p_1}$ are called {\it changes of components,} where $\omega=e^{2\pi i/3}$.
\end{defi}

A change of components simultaneously changes the arcs $I_j$, $j=0,1,2$, to which the parameters $\alpha_1,\alpha_2$ belong (see Definition \ref{defi:samecomponent}).

\begin{prop}\label{prop:nochangesofsign} Let\/ $p_1,p_2$ and\/ $q_1,q_2$ be pairs of distinct nonisotropic nonorthogonal points and let\/ $\alpha_1,\alpha_2\in\mathbb S^1\setminus\Omega$ be parameters. Assume that\/ $\sigma p_i=-\sigma q_i$, $i=1,2$, and that\/ $\tance(p_1,p_2)=\tance(q_1,q_2)$ {\rm(}in particular, $\mathrm{L}(p_1,p_2)$ is hyperbolic{\rm).} The isometries\/ $R_1:=R_{\alpha_2}^{p_2}R_{\alpha_1}^{p_1}$ and\/ $R_2:=R_{\alpha_2}^{q_2}R_{\alpha_1}^{q_1}$ are in the same conjugacy class iff they are not regular elliptic.
\end{prop}

\begin{proof}
By Remark \ref{rmk:partform},
$\trace R_1=\trace R_2$ and, by Lemma~\ref{lemma:prodspecreg}, $R_1,R_2$ are regular. Therefore, these isometries have the same type. If $R_1,R_2$ are not regular elliptic, they are conjugate (see Subsection~\ref{subsec:isometries}).

Conversely, suppose that $R_1,R_2$ are regular elliptic. Using orthogonal relations of length $3$ plus the fact that special elliptic isometries with orthogonal centres commute, we have
$$R_{\alpha_1}^{\tilde p_1}R_{\alpha_2}^{\tilde p_2}(R_{\alpha_2}^{p_2}R_{\alpha_1}^{p_1})R_{\overline\alpha_2}^{\tilde p_2}R_{\overline\alpha_1}^{\tilde p_1}=R_{\alpha_1}^{\tilde p_1}R_{\overline\alpha_2}^cR_{\alpha_1}^{p_1}R_{\overline\alpha_2}^{\tilde p_2}R_{\overline\alpha_1}^{\tilde p_1}=R_{\overline\alpha_1}^cR_{\overline\alpha_2}^c
R_{\overline\alpha_2}^{\tilde p_2}R_{\overline\alpha_1}^{\tilde p_1}=R_{\alpha_2}^{p_2}R_{\alpha_1}^{p_1},$$
where $\tilde p_i$ stands for the point in $L:=\mathrm{L}(p_1,p_2)$ orthogonal to $p_i$. In other words, $\widetilde R_1:=R_{\alpha_1}^{\tilde p_1}R_{\alpha_2}^{\tilde p_2}$ commutes with $R_1$. Since $\tance(\tilde p_1,\tilde p_2)=\tance(p_1,p_2)$, the isometry $\widetilde R_1$ is regular elliptic and, by
\hbox{\cite[Corollary~8.2]{BasmajianMiner1998}}, $\mathop\mathrm{fix}\widetilde R_1=\mathop\mathrm{fix}R_1$. Hence, if $R_1,\widetilde R_1$ are conjugate, their actions on $L$ must coincide.
Let $G,G_1,G_2$ be the geodesics associated to $R_1$ and let $\widetilde G,\widetilde G_1,\widetilde G_2$ be the geodesics associated to~$\widetilde R_1$. Clearly, $\widetilde G=G$, $\widetilde G_1=G_2$, and $\widetilde G_2=G_1$.
By Lemma \ref{lemma:construction}, the actions of $R_1$ and of $\widetilde R_1$ on $L\cap\BV$ are respectively given by the products $r_2r_1$ and $r_1r_2$, where $r_i$ denotes the reflection in the geodesic $G_i$. This implies that the actions of $R_1,\widetilde R_1$ coincide on $L$ iff $\angle_x G_1G_2=\angle_x G_2G_1=\frac\pi2$, where $x\in\BV$ stands for the intersection of $G_1,G_2$, a fixed point of both isometries.

Assume $\angle_xG_2G_1=\frac\pi 2$.
Let $\gamma:(-a,a)\to G$ be a parameterization of an open geodesic segment such that $\gamma(0)=p_2$. Let $R(t):=R_{\alpha_2}^{\gamma(t)}R_{\alpha_1}^{p_1}$; we can suppose that $R(t)$ is regular elliptic for all $t$.
Let $G,G_1,G_2(t)$ be the geodesics associated to $R(t)$ and note that $\angle_{\gamma(t)}GG_2(t)=\angle_{p_2}GG_2$; so, $\angle_{x(t)}G_2G_1=\frac{\pi}2$ iff $t=0$. Let $x(t),\tilde x(t)$ be respectively the negative and the positive intersections of $G_1$ and $G_2(t)$. Hence, $x(0)=x$ and $\tilde x(0)=\tilde x$ (clearly, $\tilde x(t)$ is the orthogonal to $x(t)$ in $L$ for all $t$). We denote by $\lambda_1(t),\lambda_2(t)$ the eigenvalues of $R(t)$ corresponding to the eigenvectors $x(t),\tilde x(t)$. It is easy to see that $\lambda_i(t)$ varies continuously with $t$. Moreover, by Lemma \ref{lemma:prodspecreg}, $\lambda_1(t)\ne\lambda_2(t)$ for all $t$.

Consider the isometry $\widetilde R(t):=R_{\alpha_1}^{\tilde p_1}R_{\alpha_2}^{\tilde\gamma(t)}$, $t\in(-a,a)$, where $\tilde\gamma(t)$ stands for the orthogonal to $\gamma(t)$ in $L$ for all $t$. As above, $\widetilde R(t)$ and $R(t)$ have the same set of eigenvalues and $\mathop\mathrm{fix}\widetilde R(t)=\mathop\mathrm{fix}R(t)$. Hence, by the argument in the previous paragraph, we obtain that, for all $t\ne0$, the isometries $R(t)$ and $\widetilde R(t)$ are not conjugate due to $\angle_{x(t)}G_2G_1\ne\frac{\pi}2$ for $t\ne0$. In other words, $\lambda_2(t)$ must be the eigenvalue of $\widetilde R(t)$ corresponding to $x(t)$. By continuity, the eigenvalues of the negative fixed point $x$ of $R=R(0)$ and of $\widetilde R=\widetilde R(0)$ are distinct and, therefore, these isometries cannot be conjugate. (We have $\lambda_1(0)=-\lambda_2(0)$, as it is easy to see.)

Let $I\in\SU(2,1)$ be such that $Iq_i=\tilde p_i$. Thus
$IR_{\alpha_2}^{q_2}R_{\alpha_1}^{q_1}I^{-1}=
R_{\alpha_2}^{\tilde p_2}R_{\alpha_1}^{\tilde p_1}$. Also note that
\begin{equation}\label{eq:conjugate}
R_{\overline\alpha_2}^{\tilde p_2}\big(R_{\alpha_2}^{\tilde p_2}
R_{\alpha_1}^{\tilde p_1}\big)R_{\alpha_2}^{\tilde p_2}
=R_{\alpha_1}^{\tilde p_1}R_{\alpha_2}^{\tilde p_2}.
\end{equation}
Therefore, $R_{\alpha_2}^{q_2}R_{\alpha_1}^{q_1}$ and
$R_{\alpha_1}^{\tilde p_1}R_{\alpha_2}^{\tilde p_2}$ are in the same conjugacy class, which is not that of $R_{\alpha_2}^{p_2}R_{\alpha_1}^{p_1}$.
\end{proof}

\begin{defi}
Let $p_1,p_2$ and $q_1,q_2$ be pairs of distinct nonisotropic nonorthogonal points with $\sigma p_i=-\sigma q_i$ and $\tance(p_1,p_2)=\tance(q_1,q_2)$. Let $\alpha_1,\alpha_2\in\mathbb S^1\setminus\Omega$ be parameters and assume that $R_{\alpha_2}^{p_2}R_{\alpha_1}^{p_1}$ is not regular elliptic. Let $I\in\SU(2,1)$ be such that $R_{\alpha_2}^{p_2}R_{\alpha_1}^{p_1}=
IR_{\alpha_2}^{q_2}R_{\alpha_1}^{q_1}I^{-1}=R_{\alpha_2}^{Iq_2}R_{\alpha_1}^{I q_1}$. We call such relations {\it simultaneous changes of signs\/} and their existence is a consequence of Proposition~\ref{prop:nochangesofsign}. Note that, by Lemma \ref{lemma:ortnotortl4}, $p_1,p_2,Iq_1,Iq_2$ lie in the same complex line.
\end{defi}

\begin{wrapfigure}{r}{0.33\textwidth}\label{fig:deltoid}
\centering
\includegraphics[scale=0.7]{figs/goldmantangent.mps}
\end{wrapfigure}

Given parameters $\alpha_1,\alpha_2\in\mathbb S^1\setminus\Omega$, define
$\tau_{\alpha_1,\alpha_2}:\mathbb R\to\mathbb C$,
$$\tau_{\alpha_1,\alpha_2}(t):=\alpha_1\alpha_2+\alpha_1^{-2}\alpha_2+\alpha_1\alpha_2^{-2}+
(\alpha_1^{-2}-\alpha_1)(\alpha_2^{-2}-\alpha_2)t.$$
By Remark \ref{rmk:partform},
$\trace R_{\alpha_2}^{p_2}R_{\alpha_1}^{p_1}=\tau_{\alpha_1,\alpha_2}(\tance(p_1,p_2))$ for nonisotropic $p_1,p_2$.

\begin{lemma}\label{lemma:tangent}
Let\/ $\alpha_1,\alpha_2\in\mathbb S^1\setminus\Omega$. The line\/ $\ell:=\tau_{\alpha_1,\alpha_2}(\mathbb R)$ is tangent to Goldman's deltoid at\/ $t=1$.
\end{lemma}

\begin{proof}
Note that $\tau_{\alpha_1,\alpha_2}(1)=2\alpha_1\alpha_2+(\alpha_1\alpha_2)^{-2}$ satisfies the equation of the deltoid. Moreover, given $0<t<1$, the isometry $R_{\alpha_2}^{p_2}R_{\alpha_1}^{p_1}$, where $p_1,p_2$ are points such that $\tance(p_1,p_2)=t$, is regular elliptic because the stable line $\mathrm{L}(p_1,p_2)$ is spherical. Now, consider the isometry $R:=R_{\alpha_2}^{p_2}R_{\alpha_1}^{p_1}$, $\tance(p_1,p_2)=t$, for $t>1$. Taking $t$ sufficiently close to $1$, the associated geodesics $G,G_1,G_2$ are such that $G_1,G_2$ are concurrent in the hyperbolic line $\mathrm{L}(p_1,p_2)$ except when $\alpha_1\alpha_2\in\Omega$; in the latter case, they are always ultraparallel. In conclusion, if $\alpha_1\alpha_2\notin\Omega$, $R$ is regular elliptic and we are done; if $\alpha_1\alpha_2=\delta\in\Omega$, $R$ is loxodromic and, since the line $\ell$ contains the vertex $3\delta$, it is tangent to the deltoid at this vertex. (Regarding this fact, see also Corollary \ref{cor:tancetype}.)
\end{proof}

\begin{rmk}\label{rmk:lines}
Let $\alpha_i,\beta_i\in\mathbb S^1\setminus\Omega$ be parameters, $i=1,2$, and consider the lines $\ell_1:=\tau_{\alpha_1,\alpha_2}(\mathbb R)$ and $\ell_2:=\tau_{\beta_1,\beta_2}(\mathbb R)$. The previous lemma, the definition of $\tau$, and the injectivity of the function $\zeta\mapsto2\zeta+\zeta^{-2}$, $\zeta\in\mathbb S^1$, immediately imply that $\ell_1=\ell_2$ iff $\alpha_1\alpha_2=\beta_1\beta_2$. Equivalently, $\alpha_1\alpha_2=\beta_1\beta_2$ iff $\tau_{\alpha_1,\alpha_2}(1)=\tau_{\beta_1,\beta_2}(1)$.
\end{rmk}

\begin{lemma}\label{lemma:dd}
Let\/ $\alpha_1,\alpha_2\in\mathbb S^1\setminus(\Omega\cup-\Omega)$ be parameters, $\alpha_1\alpha_2\notin\Omega$, such that there exists no\/ $\delta\in\Omega$ satisfying\/ $\delta\alpha_1\sim-\alpha_1$ and\/ $\delta^{-1}\alpha_2\sim-\alpha_2$. Consider the line\/ $\ell:=\tau_{\alpha_1,\alpha_2}(\mathbb R)=\tau_{-\alpha_1,-\alpha_2}(\mathbb R)$ {\rm(see Remark\/ \ref{rmk:lines}).} Then\/ $\tau_1(t):=\tau_{\alpha_1,\alpha_2}(t)$ and\/ $\tau_2(s):=\tau_{-\alpha_1,-\alpha_2}(s)$, $t,s\in\mathbb R$, parameterize\/ $\ell$ in opposite directions.
\end{lemma}

\begin{proof}
First, consider the case $\alpha_1=\alpha_2=:\alpha$. We have $\tau_1(1)=\tau_2(1)$ and $\ell$ is tangent to the deltoid at this point by Lemma \ref{lemma:tangent}. The result follows from the observation that $\tau_1(1)=\tau_2(1)$, $\tau_1(0)=\alpha^2-2\alpha^{-1}$, and $\tau_2(0)=\alpha^2+2\alpha^{-1}$ are pairwise distinct points in the deltoid.

Back to the general case, assume that the lines are parameterized in the same direction. This means that we can take $s_0\gg1$ and $t_0\gg1$ such that $\tau_1(t_0)=\tau_2(s_0)$. Let $p_i,q_i\in\BV$ be such that $\tance(p_1,p_2)=t_0$ and $\tance(q_1,q_2)=s_0$. We obtain the relation $R_{\alpha_2}^{p_2}R_{\alpha_1}^{p_1}=R_{-\alpha_2}^{q_2}R_{-\alpha_1}^{q_1}$ between the loxodromic isometries $R_{\alpha_2}^{p_2}R_{\alpha_1}^{p_1}$ and $R_{-\alpha_2}^{q_2}R_{-\alpha_1}^{q_1}$ (possibly after conjugating, say, the second one by an isometry $I\in\SU(2,1)$; abusing notation, we write $q_1,q_2$ instead of $Iq_1,Iq_2$). Applying an \hbox{$f$-bending} to $R_{\alpha_2}^{p_2}R_{\alpha_1}^{p_1}$, we make the primitives of $\alpha_1$ and $\alpha_2$ equal (see Remark \ref{rmk:sameangles}). Similarly, we assume that the primitives of $-\alpha_1,-\alpha_2$ are equal. In other words, we arrive at a relation of the form
$R_{\delta_1\alpha}^{p_2'}R_{\alpha}^{p_1'}=R_{\delta_2\beta}^{q_2'}R_{\beta}^{q_1'}$,
where $\alpha\sim\alpha_1$; $\delta_1\alpha\sim\alpha_2$ and $\beta\sim-\alpha_1$; $\delta_2\beta\sim-\alpha_2$. Since $f$-bendings preserve the products of parameters, we have $\alpha^2=\delta_2\delta_1^{-1}\beta^2$ which implies $\alpha=\pm\delta^2\beta$, where $\delta:=\delta_2\delta_1^{-1}$. Now, it follows from the relation $R_{\alpha}^{p_2'}R_{\alpha}^{p_1'}=R_{\delta\beta}^{q_2'}R_{\beta}^{q_1'}$ that
$\tau_{\alpha,\alpha}(t_0')=\tau_{\delta\beta,\beta}(s_0')$ for some $t_0'\gg1,s_0'\gg1$. Hence,
$$\tau_{\alpha,\alpha}(t_0')=\tau_{\delta\beta,\beta}(s_0')=
\delta\tau_{\pm\delta\alpha,\pm\delta\alpha}(s_0')=\tau_{\pm\alpha,\pm\alpha}(s_0').$$
It follows from the previously considered case that the sign in the above expression must be $+$, that is, $\beta=\delta\alpha$. This leads to $\delta\alpha_1\sim-\alpha_1$ and $\delta^{-1}\alpha_2\sim-\alpha_2$, a contradiction.
\end{proof}

\begin{prop}\label{prop:gdd}
Let\/ $\alpha_i,\beta_i\in\mathbb S^1\setminus\Omega$ be parameters, $i=1,2$, such that\/ $\alpha_1\alpha_2=\beta_1\beta_2$ and\/ consider the line\/ $\ell:=\tau_{\alpha_1,\alpha_2}(\mathbb R)=\tau_{\beta_1,\beta_2}(\mathbb R)$. Then\/ $\tau_1(t):=\tau_{\alpha_1,\alpha_2}(t)$ and\/ $\tau_2(s):=\tau_{\beta_1,\beta_2}(s)$, $t,s\in\mathbb R$, parameterize\/ $\ell$ in the same direction iff there exists\/ $\delta\in\Omega$ satisfying\/ $\delta\alpha_1\sim\beta_1$ and\/ $\delta^{-1}\alpha_2\sim\beta_2$.
\end{prop}

\begin{proof}
Suppose that there exists $\delta\in\Omega$ such that $\delta\alpha_1\sim\beta_1$ and $\delta^{-1}\alpha_2\sim\beta_2$. Take $t_0\gg1$ and let $p_1,p_2$ be such that $\tance(p_1,p_2)=t_0$. By Proposition \ref{prop:fbendingangles}, there exists an $f$-bending relation
$R_{\delta^{-1}\alpha_2}^{p_2}R_{\delta\alpha_1}^{p_1}=
R_{\beta_2}^{q_2}R_{\beta_1}^{q_1}$.
Taking $s_0:=\tance(q_1,q_2)>1$, we have $\tau_{\alpha_1,\alpha_2}(t_0)=\tau_{\beta_1,\beta_2}(s_0)$ which implies that $\tau_1(t)$ and $\tau_2(s)$ parameterize $\ell$ in the same direction.

Conversely, assume that there does not exist $\delta\in\Omega$ satisfying $\delta\alpha_1\sim\beta_1$ and $\delta^{-1}\alpha_2\sim\beta_2$. Note that, if $\alpha_1\alpha_2\in\Omega$, then the $\delta\in\Omega$ such that $\delta\alpha_1\sim\beta_1$ also satisfies $\delta^{-1}\alpha_2\sim\beta_2$. Therefore, we have $\alpha_1\alpha_2\notin\Omega$.

Suppose that $\tau_1(t)$ and $\tau_2(s)$ parameterize $\ell$ in the same direction. As in the proof of the previous lemma, we obtain a relation $R_{\alpha_2}^{p_2}R_{\alpha_1}^{p_1}=R_{\beta_2}^{q_2}R_{\beta_1}^{q_1}$ with $\tance(p_1,p_2)=t_0\gg1$ and $\tance(q_1,q_2)=s_0\gg1$. Applying $f$-bendings if necessary, we can assume that $\alpha_i,\beta_i\notin\Omega\cup-\Omega$.

Take $\delta_1\in\Omega$ such that $\delta_1\alpha_1\sim\beta_1$ and consider the relation $R_{\delta_1^{-1}\alpha_2}^{p_2}R_{\delta_1\alpha_1}^{p_1}=
R_{\beta_2}^{q_2}R_{\beta_1}^{q_1}$. An \hbox{$f$-bending} of $R_{\delta_1^{-1}\alpha_2}^{p_2}R_{\delta_1\alpha_1}^{p_1}$ sending $\delta_1\alpha_1$ to $\beta_1$ does not exist since, otherwise, the equality $\alpha_1\alpha_2=\beta_1\beta_2$ would imply $\delta_1^{-1}\alpha_2\sim\beta_2$.

Take $\delta_2\in\Omega$ such that $\delta_2\alpha_1\sim-\beta_1$ and consider the relation $R_{\delta_2^{-1}\alpha_2}^{p_2}R_{\delta_2\alpha_1}^{p_1}=
R_{\beta_2}^{q_2}R_{\beta_1}^{q_1}$. Assume that there exists an $f$-bending of $R_{\delta_2^{-1}\alpha_2}^{p_2}R_{\delta_2\alpha_1}^{p_1}$ sending $\delta_2\alpha_1$ to $-\beta_1$. This $f$-bending sends $\delta_2^{-1}\alpha_2$ to $-\beta_2$ and gives rise to the relation $R_{-\beta_2}^{p_2'}R_{-\beta_1}^{p_1'}=R_{\beta_2}^{q_2}R_{\beta_1}^{q_1}$. A cube root of unity $\delta\in\Omega$ such that $\delta\beta_1\sim-\beta_1$ and $\delta^{-1}\beta_2\sim-\beta_2$ does not exist because, otherwise, $\delta_2\delta^{-1}\alpha_1\sim\beta_1$ and $\delta\delta_2^{-1}\alpha_2\sim\beta_2$. Hence, by Lemma \ref{lemma:dd}, $\sigma p_1'\sigma p_2'=-\sigma q_1\sigma q_2$, a contradiction.

Finally, take $\delta_3\in\Omega$ such that $\delta_3(-\alpha_1)\sim\beta_1$. We will show that there must exist an $f$-bending of $R_{\beta_2}^{q_2}R_{\beta_1}^{q_1}$ sending $\beta_1$ to $\delta_3(-\alpha_1)$.

Let $I_j=\arc(\omega^{j-1},\omega^j)$ stand for the component of $\beta_1$. Given an arbitrary parameter $\eta\in\mathbb S^1\setminus\Omega$, we denote by $\langle\eta\rangle$ the representative of $\eta$ in $I_j$. The point $\langle\alpha_1\alpha_2\rangle$ divides $I_j$ in the open arcs $\arc\big(\omega^{j-1},\langle\alpha_1\alpha_2\rangle\big)$ and $\arc(\langle\alpha_1\alpha_2\rangle,\omega^j)$. Let $I_j^0$ be the one containing $\beta_1$ and, $I_j^1$, the other. Similarly, the point $-\omega^{j+1}$ divides $I_j$ in the open arcs $\arc(\omega^{j-1},-\omega^{j+1})$ and $\arc(-\omega^{j+1},\omega^j)$; we call $J_0$ the one containing $\beta_1$ and, $J_1$, the other. It is easy to see that, given an arbitrary $\eta\in\mathbb S^1\setminus(\Omega\cup-\Omega)$, the points $\langle\eta\rangle$ and $\langle-\eta\rangle$ lie in distinct arcs among $J_0,J_1$.

Since there is no $f$-bending of $R_{\delta_1^{-1}\alpha_2}^{p_2}R_{\delta_1\alpha_1}^{p_1}$ sending $\delta_1\alpha_1$ to $\beta_1$ and no $f$-bending of $R_{\delta_2^{-1}\alpha_2}^{p_2}R_{\delta_2\alpha_1}^{p_1}$ sending $\delta_2\alpha_1$ to $-\beta_1$, we have $\beta_1,\langle-\beta_1\rangle\in I_j^0$ and $\langle\alpha_1\rangle\in I_j^1$ (see Remark \ref{rmk:flimits}). So, $\langle-\alpha_1\rangle$ also belongs to $I_j^0$ which leads to the required $f$-bending. We obtain
$R_{\beta_2}^{q_2}R_{\beta_1}^{q_1}=
R_{-\delta_3^{-1}\alpha_2}^{q_2'}R_{-\delta_3\alpha_1}^{q_1'}$, that is,
$R_{\alpha_2}^{p_2}R_{\alpha_1}^{p_1}=R_{-\alpha_2}^{q_2'}R_{-\alpha_1}^{q_1'}$. Again, Lemma \ref{lemma:dd} provides a contradiction.
\end{proof}

\begin{defi}
Let $\alpha_i,\beta_i\in\mathbb S^1\setminus\Omega$, $i=1,2$, be parameters such that $\alpha_1\alpha_2=\beta_1\beta_2$. Assume that there does not exist $\delta\in\Omega$ satisfying $\delta\alpha_1\sim\beta_1$ and $\delta^{-1}\alpha_2\sim\beta_2$. Consider the line\/ $\ell:=\tau_1(\mathbb R)=\tau_2(\mathbb R)$, where $\tau_1(t):=\tau_{\alpha_1,\alpha_2}(t)$ and $\tau_2(s):=\tau_{\beta_1,\beta_2}(s)$ and take a pair $(t_0,s_0)$ with $t_0>1$ and $s_0<0$ such that $\tau_1(t_0)=\tau_2(s_0)$ does not belong to the interior of Goldman's deltoid. Let $p_1,p_2$ and $q_1,q_2$ be pairs of nonisotropic nonorthogonal distinct points satisfying $\tance(p_1,p_2)=t_0$ and $\tance(q_1,q_2)=s_0$. Clearly, $\sigma q_1\sigma q_2=-1$ and $\sigma p_1\sigma p_2=1$. Since $R_{\alpha_2}^{p_2}R_{\alpha_1}^{p_1}$ and $R_{\beta_2}^{q_2}R_{\beta_1}^{q_1}$ are loxodromic, there exists $I\in\SU(2,1)$ such that $R_{\alpha_2}^{p_2}R_{\alpha_1}^{p_1}=R_{\beta_2}^{Iq_2}R_{\beta_1}^{Iq_1}$. We call such relation a {\it single change of sign.}
\end{defi}

Let $\alpha_i,\beta_i\in\mathbb S^1\setminus\Omega$ be parameters satisfying $\alpha_1\alpha_2\neq \beta_1\beta_2$. Then, by Remark \ref{rmk:lines}, the lines
$\ell_1:=\tau_{\alpha_1,\alpha_2}(\mathbb R)$ and $\ell_2:=\tau_{\beta_1,\beta_2}(\mathbb R)$
are distinct. So, they intersect inside the deltoid (including the boundary, i.e., they intersect in
the set $f(\tau)\leqslant0$, where $f$ is defined in
Subsection \ref{subsec:isometries}). Let $t_0,s_0\in\mathbb R$ be such that $\tau_{\alpha_1,\alpha_2}(t_0)=
\tau_{\beta_1,\beta_2}(s_0)$. Take pairs $p_1,p_2$ and $q_1,q_2$ of distinct nonisotropic nonorthogonal points such that $\tance(p_1,p_2)=t_0$ and $\tance(q_1,q_2)=s_0$. Assume that $t_0,s_0\ne1$ and let $\tilde p_1,\tilde p_2$ be the points in $\mathrm L(p_1,p_2)$ respectively orthogonal to $p_1,p_2$; similarly, let $\tilde q_1,\tilde q_2$ be the points in $\mathrm L(q_1,q_2)$ respectively orthogonal to $q_1,q_2$.
The isometries $R_{\alpha_2}^{p_2}R_{\alpha_1}^{p_1}$, $R_{\beta_2}^{q_2}R_{\beta_1}^{q_1}$, by construction, are regular elliptic and have the same trace.

If $t_0\notin[0,1]$ and $s_0\notin[0,1]$, by Proposition \ref{prop:nochangesofsign} and formula \eqref{eq:conjugate}, one of the isometries $R_1:=R_{\alpha_2}^{p_2}R_{\alpha_1}^{p_1}$,
$\widetilde R_1:=R_{\alpha_2}^{\tilde p_2}R_{\alpha_1}^{\tilde p_1}$ is conjugate to one of the isometries $R_2:=R_{\beta_2}^{q_2}R_{\beta_1}^{q_1}$, $\widetilde R_2:=R_{\beta_2}^{\tilde q_2}R_{\beta_1}^{\tilde q_1}$ since there are three distinct conjugacy classes of regular elliptic isometries with the same trace. It follows that either
$R_{\alpha_2}^{p_2}R_{\alpha_1}^{p_1}=R_{\beta_2}^{Iq_2}R_{\beta_1}^{Iq_1}$ or
$R_{\alpha_2}^{p_2}R_{\alpha_1}^{p_1}=R_{\beta_2}^{I\tilde q_2}R_{\beta_1}^{I\tilde q_1}$ holds for some $I\in\SU(2,1)$.

When $t_0\in[0,1]$, then $\mathrm{L}(p_1,p_2)$ is spherical and $R_1$ is conjugate to $\widetilde R_1$. So, it can happen that none of $R_1,\widetilde R_1$ is conjugate to one of $R_2,\widetilde R_2$. In this case, we take $-\alpha_1,-\alpha_2$ instead of $\alpha_1,\alpha_2$ in the above construction. Now, by Lemma \ref{lemma:dd}, the parameter $t_0'$ corresponding to the intersection point of the lines $\ell_1,\ell_2$ satisfy $t_0'\notin[0,1]$. If it is still the case that none of $R_1,\widetilde R_1$ is conjugate to one of $R_2,\widetilde R_2$, this means that $s_0\in[0,1]$ and we also take $-\beta_1,-\beta_2$ instead of $\beta_1,\beta_2$ thus obtaining an orthogonal relation.

\begin{defi} A relation constructed as above is a {\it {\rm(}length\/ $4${\rm)} orthogonal relation\/} (due to Lemma \ref{lemma:ortnotortl4}, the lines $\mathrm L(p_1,p_2)$ and $\mathrm L(Iq_1,Iq_2)$ are orthogonal).
\end{defi}

As the next theorem shows, every (generic) length $4$ relation is a consequence of the previously introduced ones.

\begin{thm}
\label{thm:alll4}
Let\/ $p_1,p_2$ and\/ $q_1,q_2$ be pairs of distinct nonisotropic nonorthogonal points and let\/ $\alpha_i,\beta_i\in\mathbb S^1\setminus\Omega$ be parameters, $i=1,2$. Assume that\/ $\mathrm{L}(p_1,p_2)$ and\/ $\mathrm{L}(q_1,q_2)$ are noneuclidean. A length\/ $4$ relation\/ $R_{\alpha_2}^{p_2}R_{\alpha_1}^{p_1}=\delta R_{\beta_2}^{q_2}R_{\beta_1}^{q_1}$ follows from bending, $f$-bending, change of orientation, change of components, simultaneously change of signs, single change of sign, and orthogonal relations.
\end{thm}

\begin{proof}
Taking (say) $\delta\beta_2$ in place of $\beta_2$, we can suppose that $\delta=1$.

Assume $\alpha_1\alpha_2=\beta_1\beta_2$. Suppose that there does not exist a cube root of unity $\zeta\in\Omega$ such that $\zeta\alpha_1\sim\beta_1$ and $\zeta^{-1}\alpha_2\sim\beta_2$. By Proposition \ref{prop:gdd} and Lemma \ref{lemma:ortnotortl4}, we have $\tance(p_1,p_2)>1$ and $\tance(q_1,q_2)<0$ (or vice-versa); moreover, $R:=R_{\alpha_2}^{p_2}R_{\alpha_1}^{p_1}=R_{\beta_2}^{q_2}R_{\beta_1}^{q_1}$ is loxodromic. Therefore, the relation is a single change of sign.

Suppose that there exists $\zeta\in\Omega$ such that $\zeta\alpha_1\sim\beta_1$ and $\zeta^{-1}\alpha_2\sim\beta_2$. Hence, modulo change of components, we can assume $\alpha_i\sim\beta_i$, $i=1,2$. By Theorem \ref{thm:all4relations} we can assume that, modulo bending and $f$-bending, $\sigma p_1\ne\sigma q_1$ or $\sigma p_2\ne\sigma q_2$. By Proposition \ref{prop:gdd}, $\sigma p_1\sigma p_2=\sigma q_1\sigma q_2$.

We arrive at the following cases:

\smallskip

$\bullet$ $\sigma p_1=\sigma q_2$ and $\sigma p_2=\sigma q_1$. Modulo a change of orientation, we can assume that the relation has the form $R_{\alpha_1}^{p_2'}R_{\alpha_2}^{p_1'}=R_{\beta_2}^{q_2}R_{\beta_1}^{q_1}$ with $\sigma p_i'=\sigma q_i$ and $\alpha_i\sim\beta_i$. Let $\zeta\in\Omega$ be such that $\zeta\alpha_1\sim\beta_2$. We have $\zeta^{-1}\alpha_2\sim\zeta^{-1}\beta_2\sim\alpha_1\sim\beta_1$. Now, a change of components reduces this case to an already considered one.

$\bullet$ $\sigma p_i=-\sigma q_i$. Since, in this case, $\sigma p_1=\sigma p_2$ and $\sigma q_1=\sigma q_2$, it follows from Proposition \ref{prop:fbendingangles} that, modulo $f$-bending, we can take $\alpha_i=\beta_i$, $i=1,2$. Now, by Proposition \ref{prop:nochangesofsign}, the product $R_{\beta_2}^{q_2}R_{\beta_1}^{q_1}$ cannot be regular elliptic and therefore a simultaneous change of signs (say, of the form $R_{\beta_2}^{q_2}R_{\beta_1}^{q_1}=R_{\beta_2}^{q_2'}R_{\beta_1}^{q_1'}$ where $\sigma q_i'=\sigma p_i$) reduces the relation to bendings.

\smallskip

Finally, when $\alpha_1\alpha_2\ne\beta_1\beta_2$, we have an orthogonal relation.
\end{proof}

\section{Modifying \textit{n}-gons}
\label{sec:modifying}

A {\it special elliptic $n$-gon} is a configuration of $n$ nonisotropic points $p_1,\ldots,p_n\in\PCV\setminus\SV$ along with $n$ parameters $\alpha_1,\ldots,\alpha_n\in\mathbb S^1\setminus\Omega$ satisfying the following properties:

\smallskip

(P1) at most one point $p_i$ is positive;

(P2) $p_i$ is not equal nor orthogonal to $p_{i+1}$ (index mod $n$);

(P3) $R_{\alpha_n}^{p_n}\ldots R_{\alpha_1}^{p_1}=\delta$ for some $\delta\in\Omega$;

(P4) $\Pi\alpha_i\neq\delta$.

\smallskip

We sometimes say that a relation $R_{\alpha_n}^{p_n}\ldots R_{\alpha_1}^{p_1}=\delta$ is a special elliptic $n$-gon (or, simply, an $n$-gon) if the points $p_1,\dots,p_n$ and parameters $\alpha_1,\dots,\alpha_n$ satisfy the above properties.

Given an $n$-gon $R_{\alpha_n}^{p_n}\ldots R_{\alpha_1}^{p_1}=\delta$, we can use the length 4 relations obtained in Sections \ref{sec:bendings} and \ref{sec:n-o-solutions} to modify it into a different relation. Note that parameters and signatures of points are invariants of bendings; the product of the parameters, their components, and the signatures of the points are invariants of bendings and $f$-bendings. We are interested in the following problem: given an $n$-gon, can we obtain, by bending (or bending and $f$-bending) such $n$-gon, every other $n$-gon with the same invariants? We focus on the case $n=5$.

\begin{prop}
\label{prop:pentsrt}
Let\/ $R:=R_{\alpha_5}^{p_5}R_{\alpha_4}^{p_4}R_{\alpha_3}^{p_3}R_{\alpha_2}^{p_2}
R_{\alpha_1}^{p_1}=\delta$ be a special elliptic pentagon. Then there exists\/ $i$ such that\/ $R_{\alpha_{i+1}}^{p_{i+1}}R_{\alpha_{i}}^{p_{i}}R_{\alpha_{i-1}}^{p_{i-1}}$
is strongly regular {\rm (}indices\/ {\rm mod} $5${\rm )}.
\end{prop}

\begin{proof}
By Definition~\ref{defi:stronglyregular}, if none of the
$R_{\alpha_{i+1}}^{p_{i+1}}R_{\alpha_{i}}^{p_{i}}R_{\alpha_{i-1}}^{p_{i-1}}$
is strongly regular, then every $p_i$ lies in the same
complex line $L$. Applying $R$ to the polar point of $L$ we obtain $\Pi\alpha_i=\delta$ and this contradicts~(P4).
\end{proof}

\begin{prop}
\label{prop:pentsrtlox}
Let\/ $R_{\alpha_5}^{p_5}R_{\alpha_4}^{p_4}R_{\alpha_3}^{p_3}R_{\alpha_2}^{p_2}R_{\alpha_1}^{p_1}=\delta$
be a special elliptic pentagon such that\/ $R_{\alpha_{i+1}}^{p_{i+1}}R_{\alpha_{i}}^{p_{i}}$ is loxodromic for some\/ $i$. Then we can bend this pentagon in such a way that the new obtained configuration satisfies {\rm (}indices\/ {\rm mod} $5${\rm):}

\smallskip
$\bullet$ $R_{\alpha_{i+1}}^{p_{i+1}}R_{\alpha_{i}}^{p_{i}}
R_{\alpha_{i-1}}^{p_{i-1}}$ is strongly regular for all\/ $i$

\smallskip

$\bullet$ $R_{\alpha_{i+1}}^{p_{i+1}}R_{\alpha_{i}}^{p_{i}}$ is loxodromic\/ for all $i$.
\end{prop}

\begin{proof}
By Proposition~\ref{prop:pentsrt} we can assume that, say, $R_{\alpha_3}^{p_3}R_{\alpha_2}^{p_2}R_{\alpha_1}^{p_1}$ is strongly regular.
If $R_{\alpha_4}^{p_4}R_{\alpha_3}^{p_3}R_{\alpha_2}^{p_2}$ is not strongly
regular, then $p_2,p_3,p_4$ lie in a same complex line $L$. The intersection between $L$ and $\mathrm L(p_1,p_2)$ is only the point $p_2$ since $p_1,p_2,p_3$ do not lie in a same complex line. Without losing the strong regularity of $R_{\alpha_3}^{p_3}R_{\alpha_2}^{p_2}R_{\alpha_1}^{p_1}$, we bend $R_{\alpha_2}^{p_2}R_{\alpha_1}^{p_1}$ thus
removing the point $p_2$ from the line $\mathrm L$ and obtaining new points $p_1',p_2'$ such that $R_{\alpha_4}^{p_4}R_{\alpha_3}^{p_3}R_{\alpha_2}^{p_2'}$ is strongly regular. Clearly, we may now assume that every $R_{\alpha_{i+1}}^{p_{i+1}}R_{\alpha_{i}}^{p_{i}}R_{\alpha_{i-1}}^{p_{i-1}}$
is strongly regular.

The rest is analogous to the corresponding part of the proof of Corollary \ref{cor:connectSlox}.
\end{proof}

\begin{thm}
\label{thm:connectpent}
Let\/ $R_{\alpha_5}^{p_5}R_{\alpha_4}^{p_4}R_{\alpha_3}^{p_3}R_{\alpha_2}^{p_2}R_{\alpha_1}^{p_1}=\delta$
and\/
$R_{\alpha_5}^{q_5}R_{\alpha_4}^{q_4}R_{\alpha_3}^{q_3}R_{\alpha_2}^{q_2}R_{\alpha_1}^{q_1}=\delta$
be special elliptic pentagons. Suppose that\/ $\sigma p_i=\sigma q_i$ for all\/ $i$ and that at least one of\/ $R_{\alpha_{i+1}}^{p_{i+1}}R_{\alpha_i}^{p_i}$, as well as at least one of\/ $R_{\alpha_{j+1}}^{q_{j+1}}R_{\alpha_j}^{q_j}$, is loxodromic. Then, up to conjugacy, the pentagons can be connected by finitely many bendings.
\end{thm}

\begin{proof}
By Proposition \ref{prop:pentsrtlox}, we can assume $R_{\alpha_{i+1}}^{p_{i+1}}R_{\alpha_{i}}^{p_{i}}R_{\alpha_{i-1}}^{p_{i-1}}$ strongly regular and $R_{\alpha_{i+1}}^{p_{i+1}}R_{\alpha_{i}}^{p_{i}}$ loxodromic for all $i$ (indices mod $5$). The same is true if we take the $q_i$'s in place of the $p_i$'s. Together with the fact that $\sigma p_i=\sigma q_i$, this implies that we can make $\tance(p_4,p_5)=\tance(q_4,q_5)$ by means of bendings.
Now, since $R_{\alpha_3}^{p_3}R_{\alpha_2}^{p_2}R_{\alpha_1}^{p_1}$ and
$R_{\alpha_3}^{q_3}R_{\alpha_2}^{q_2}R_{\alpha_1}^{q_1}$ are strongly regular
and $\trace(\delta R_{\overline\alpha_4}^{p_4}R_{\overline\alpha_5}^{p_5})
=\trace(\delta R_{\overline\alpha_4}^{q_4}R_{\overline\alpha_5}^{q_5})$ (see Remark \ref{rmk:partform}), we can assume, by Corollary~\ref{cor:connectSlox}, that there exists $I\in\SU(2,1)$ such that $Iq_i=p_i$ for $i=1,2,3$. Now, $R_{\alpha_5}^{p_5}R_{\alpha_4}^{p_4}=R_{\alpha_5}^{Iq_5}R_{\alpha_4}^{Iq_4}$ due to $R_{\alpha_5}^{p_5}R_{\alpha_4}^{p_4}R_{\alpha_3}^{p_3}R_{\alpha_2}^{p_2}R_{\alpha_1}^{p_1}=
R_{\alpha_5}^{Iq_5}R_{\alpha_4}^{Iq_4}R_{\alpha_3}^{Iq_3}R_{\alpha_2}^{Iq_2}R_{\alpha_1}^{Iq_1}$. It~remains to observe that the former relation is a bending relation by Theorem \ref{thm:bendingsarebendings}.
\end{proof}

The following proposition is a kind of generalization of Propositions \ref{prop:fbendingangles} and \ref{prop:hypfbendingangles} to the case of relations of arbitrary length.

\begin{prop}
\label{prop:fbendtangles}
Let\/ $p_1,\dots,p_n\in\mathbb PV\setminus\SV$ with\/ $p_i$ distinct from and nonorthogonal to\/ $p_{i+1}$, $i=1,\dots,n-1$ and such that at most one of the\/ $p_i$'s is positive. Let\/ $\alpha_1,\dots,\alpha_n\in\mathbb S^1\setminus\Omega$ be parameters such that at least one of the isometries\/ $R_{\alpha_{i+1}}^{p_{i+1}}R_{\alpha_i}^{p_i}$ is loxodromic. Given parameters\/ $\beta_1,\dots,\beta_n\in\mathbb S^1\setminus\Omega$ satisfying\/ $\beta_i\sim\alpha_i$ for all\/ $i$ and\/ $\Pi\alpha_i=\Pi\beta_i$, there exist\/ $q_1,\dots q_n$ with\/ $\sigma p_i=\sigma q_i$ for all\/ $i$ such that\/ $R_{\beta_n}^{q_n}\ldots R_{\beta_1}^{q_1}$ can be obtained from\/ $R_{\alpha_n}^{p_n}\ldots R_{\alpha_1}^{p_1}$ by finitely many bendings and\/ $f$-bendings. Furthermore, at least one of the $R_{\beta_{i+1}}^{q_{i+1}}R_{\beta_i}^{q_i}$ is loxodromic.
\end{prop}

\noindent{\it Proof.} The proof is by induction on $n$. The case $n=2$ follows from Propositions \ref{prop:fbendingangles} and \ref{prop:hypfbendingangles}. Assume that the fact holds for $n-1$ and suppose that we have shown that, by means of finitely many bendings and $f$-bendings, one is able to deform $R_{\alpha_n}^{p_n}\ldots R_{\alpha_1}^{p_1}$ into $R_{\alpha_n'}^{p_n'}\ldots R_{\alpha_2'}^{p_2'}R_{\beta_1}^{q_1}$ with at least one of the $R_{\alpha_{i+1}'}^{p_{i+1}'}R_{\alpha_i'}^{p_i'}$ loxodromic. The parameters $\alpha_i'$ and the points $p_i'$ will therefore satisfy the conditions of the proposition and we are done.

\begin{wrapfigure}{r}{0.22\textwidth}
\centering
\includegraphics[scale=0.8,trim={0 0 0 0},clip]{figs/fbendangles-vf.mps}
\end{wrapfigure}

In what follows, $d(-,-)$ stands for the distances measured along the circle. For simplicity, consider the case $\Arg\beta_1>\Arg\alpha_1$. Let $\omega_i\in\{1,\omega,\omega^2\}$ be such that $\omega_i=1$ when $\alpha_i\in I_0$, $\omega_i=\omega$ when $\alpha_i\in I_1$, and
$\omega_i=\omega^2$ when $\alpha_i\in I_2$ (see the paragraph above Definition~\ref{defi:samecomponent} for the definitions of $\omega,I_0,I_1,I_2$). As in the proof of Proposition \ref{prop:pentsrtlox}, we can assume that $R_{\alpha_2}^{p_2}R_{\alpha_1}^{p_1}$ is loxodromic. Let $\varepsilon>0$ be small (say, smaller than $d(\beta_1,\omega\omega_1)$). Given $\eta\in I_0$ such that $\eta\alpha_1\sim\alpha_1$ and $\overline\eta\alpha_2\sim\alpha_2$, there exists by Proposition~\ref{prop:hypfbendingangles} an $f$-bending sending $\alpha_1$ to $\eta\alpha_1$ and $\alpha_2$ to $\overline\eta\alpha_2$. In this way, we ``increase'' $\alpha_1$ in the direction of $\beta_1$ (which is the same as that of $\omega\omega_1$); however, this process ``decreases'' $\alpha_2$ and one may not be able to reach $\beta_1$ before $\alpha_2$ arrives at $\omega_2$. We obtain new parameters $\alpha_1',\alpha_2'$ and it is possible to assume that $d(\alpha_2',\omega_2)<\varepsilon$ and $d(\alpha_1',\omega\omega_1)>\varepsilon$. (Indeed, if $d(\alpha_1',\omega\omega_1)$ becomes smaller than $\varepsilon$ during the $f$-bending, this suffices to make $\alpha_1'=\beta_1$ and we are done.) The next step is to move $\alpha_3$ in the direction of $\omega_3$. Bending the loxodromic isometry $R_{\alpha_2'}^{q_2}R_{\alpha_1'}^{q_1}$ (this is the isometry obtained after the $f$-bending) we make $R_{\alpha_3}^{p_3}R_{\alpha_2'}^{q_2}$ loxodromic. By taking $\varepsilon$ small enough, we may assume $d(\alpha_3,\omega\omega_3)>\varepsilon$ and proceed as before so as to obtain $d(\alpha_2',\omega_2)<\varepsilon$ (from now on, we abuse notation and always write $\alpha_i',q_i$ for the new parameters and points that are obtained after bendings and $f$-bendings). Now, $f$-bending $R_{\alpha_3}^{p_3}R_{\alpha_2'}^{q_2}$, we (obtain new parameters $\alpha_2',\alpha_3'$ as well as new points $q_2,q_3$ and) can assume that $d(\alpha_3',\omega_3)<\varepsilon$. Bending $R_{\alpha_3'}^{q_3}R_{\alpha_2'}^{q_2}$ if necessary, we make $R_{\alpha_2'}^{q_2}R_{\alpha_1'}^{q_1}$ loxodromic and, as before, it is possible to assume that $d(\alpha_2',\omega_2)<\varepsilon$ and $d(\alpha_1',\omega\omega_1)>\varepsilon$.

Iterating this procedure, we reach the situation where $d(\alpha_i',\omega_i)<\varepsilon$ for $i=2,\dots,n$ and $d(\alpha_1',\omega\omega_1)>\varepsilon$. Let $\langle\alpha_i\rangle$ stand for the representative of $\alpha_i$ that lies in $I_0$, $i=2,\dots,n$ (see Definition~\ref{defi:samecomponent}). The existence of such a configuration of parameters $\alpha_i'$ for arbitrarily small $\varepsilon$ implies that $\beta_1$ does not belong to $\arc\big(\alpha_1,\alpha_1\langle\alpha_2\rangle\langle\alpha_3\rangle
\dots\langle\alpha_n\rangle\big)$ because $d\big(1,\langle\alpha_i\rangle\big)=d(\omega_i,\alpha_i)$, $i=2,\dots,n$. It is easy to see that $\Pi\alpha_i=\Pi\beta_i$ and $\alpha_i\sim\beta_i$ imply $\alpha_1\langle\alpha_2\rangle\langle\alpha_3\rangle\dots\langle\alpha_n\rangle=
\beta_1\langle\beta_2\rangle\langle\beta_3\rangle\dots\langle\beta_n\rangle$. In other words, $\beta_1$ does not belong to $\arc\big(\alpha_1,\beta_1\langle\beta_2\rangle\langle\beta_3\rangle
\dots\langle\beta_n\rangle\big)$. This contradicts $\Arg\alpha_1<\Arg\beta_1$.$\hfill\square$

\medskip

The next theorem follows directly from Propositions~\ref{prop:pentsrtlox},~\ref{prop:fbendtangles},
and Theorem~\ref{thm:connectpent}:

\begin{thm}
\label{thm:connectspecpent}
Let\/ $R_{\alpha_5}^{p_5}R_{\alpha_4}^{p_4}R_{\alpha_3}^{p_3}R_{\alpha_2}^{p_2}R_{\alpha_1}^{p_1}=\delta$ and\/ $R_{\beta_5}^{q_5}R_{\beta_4}^{q_4}R_{\beta_3}^{q_3}R_{\beta_2}^{q_2}R_{\beta_1}^{q_1}=\delta$ be special elliptic pentagons with\/ $\sigma p_i=\sigma q_i$ and such that at least one of the isometries\/ $R_{\alpha_{i+1}}^{p_{i+1}}R_{\alpha_i}^{p_i}$, as well as at least one of the isometries\/ $R_{\beta_{i+1}}^{q_{i+1}}R_{\beta_i}^{q_i}$, is loxodromic. Assume that the corresponding parameters are in the same components, $\alpha_i\sim\beta_i$, and that the products of parameters are the same, $\Pi\alpha_i=\Pi\beta_i$. Then, up to conjugacy, the pentagons are connected by finitely many bendings and\/ $f$-bendings.
\end{thm}

\bibliographystyle{plain}
\bibliography{references}

\noindent
{\sc Felipe A.~Franco}

\noindent
{\sc Departamento de Matem\'atica, IMECC, Universidade Estadual de Campinas, Brasil}

\noindent
\url{felipefranco@ime.unicamp.br}, \url{f.franco.math@gmail.com}

\medskip

\noindent
{\sc Carlos H.~Grossi}

\noindent
{\sc Departamento de Matem\'atica, ICMC, Universidade de S\~ao Paulo, S\~ao Carlos, Brasil}

\noindent
\url{grossi@icmc.usp.br}

\end{document}